\documentclass[a4paper,11pt]{article} %

\usepackage{amsmath} %
\usepackage{amsthm} %
\usepackage{amssymb} %
\usepackage{epsfig} %
\usepackage{psfrag} %
\usepackage[mathscr,mathcal]{eucal} %
\usepackage{enumerate} %
\usepackage{dsfont} %
\usepackage[margin=1in]{geometry} %
\usepackage{enumitem}
\usepackage{graphicx}
\usepackage{twoopt}
\usepackage{xcolor}
\usepackage{hyperref}
\hypersetup{colorlinks=true, citecolor=blue, urlcolor=blue, linkcolor=blue}

\usepackage{wasysym} %

  \ifx\LabelFigloaded\MYundefined\relax
  \else
   \fi

  \def\LabelFigloaded{\relax}

  \chardef\LabelFigCatAt\the\catcode`\@
  \catcode`\@=11

 \let\LabelFigwlog@ld\wlog
 \def\wlog#1{\relax}

 \ifx\\\MYundefined@
    \let\\\relax
 \fi

  \def\ms@g{\immediate\write16}

 \def\N@wif{\csname newif\endcsname }
 \def\Temp@ {\N@wif\ifIN@}
 \ifx\INN@\MYundefined@
    \else \let\Temp@\relax
 \fi
 \Temp@

  \def\IN@{\expandafter\INN@\expandafter}
  \long\def\INN@0#1@#2@{\long\def\NI@##1#1##2##3\ENDNI@
    {\ifx\m@rker##2\IN@false\else\IN@true\fi}%
     \expandafter\NI@#2@@#1\m@rker\ENDNI@}
  \def\m@rker{\m@@rker}
 
  \newtoks\Initialtoks@  \newtoks\Terminaltoks@
  \def\SPLIT@{\expandafter\SPLITT@\expandafter}
  \def\SPLITT@0#1@#2@{\def\TTILPS@##1#1##2@{%
     \Initialtoks@{##1}\Terminaltoks@{##2}}\expandafter\TTILPS@#2@}

 \def\Shifted@@#1#2#3{\setbox0=\hbox{#3}%
   \raise -\dp0\vbox {\kern-#2%
       \hbox {\kern#1\unhbox0\kern-#1}%
           \kern#2}}

 \newcount\gridcount
 \newbox\auxGridbox@ \newbox\hGridbox@ \newbox\vGridbox@
 \newbox\Labelbox@ \newbox\auxLabelbox@
 \newbox\Coordinatebox@
 \newtoks\Labeltoks@
 \newdimen\Wdd@ \newdimen\Htt@
 \newdimen\Wddd@ \newdimen\Httt@
 
 \def\Wr@{\immediate\write16}

 \newdimen\GL@wd
 \def\GridLineWidth#1{\GL@wd=#1}

 \def\gobble#1{}
 \def\EdgeErr@{\Wr@{}%
      \Wr@{\string\Edges\space argument
      1, 10, 100 or 1000 please\string!}%
      }

 \newcount\Edgect@

 \def\Sweepup#1\endSweepup{}

 \def\SetEdges@{%
    \edef\Zr@@s{\expandafter\gobble\number\Edgect@\empty}%
        \count255=0\Zr@@s\relax
        \ifnum\count255=\z@\else\EdgeErr@\show\tailtest\fi
        \count255=1\Zr@@s\relax
        \ifnum\count255=\Edgect@\relax\else\EdgeErr@\show\leadtest\fi
    \EdgGl@b\edef\Zr@s{\expandafter\gobble\Zr@@s\empty}
    \ifnum\Edgect@>\@ne\relax\EdgGl@b\let\L@Dc\empty
        \else\EdgGl@b\edef\L@Dc{\string.}\fi
    \ifnum\Edgect@>\@ne\relax
        \EdgGl@b\edef\Edgescale@##1{\divide##1 by \Edgect@}%
        \else\EdgGl@b\edef\Edgescale@##1{}\fi
    }

 \def\Edges#1{\Edgect@=#1\relax
     \let\EdgGl@b\global \SetEdges@}

 \Edges{1}

 \def\hhrule{\hrule height \GL@wd\vskip-.\GL@wd}

 \def\hRule@{%
   \advance\gridcount -2%
   \vfil\hhrule\vfil
   \llap{\smash{\raise -2.5pt
     \hbox{\L@Dc\number\gridcount\Zr@s\kern2pt}}}%
   \hhrule
   }

\def\vvrule{\vrule width \GL@wd \kern-\GL@wd}

 \def\vRule@{\advance\gridcount 2%
   \hfil\vvrule\hfil
   \setbox\auxGridbox@=\vbox to 0pt
      {\vskip \Htt@\vskip 2pt
        \hbox to 0pt{\hss\L@Dc\number\gridcount\Zr@s\hss}\vss}%
      \wd\auxGridbox@=0pt \box\auxGridbox@
   \vvrule
   }

 \def\PlaceGrid@@{\gridcount=10 
  \setbox\hGridbox@=\hbox{%
        \hbox{%
             \hskip-.4pt\vrule
             \vbox to \Htt@{%
               \offinterlineskip\parindent=\z@\relax
               \hbox to \Wdd@{\hfil}
               \hRule@\hRule@\hRule@\hRule@
               \vfil\hhrule\vfil}%
             \vrule\hskip-.4pt}
    }%
  \gridcount=0%
  \setbox\vGridbox@=\hbox{%
      \vbox{\offinterlineskip\parindent=0pt\hsize=0pt
         \vskip-.4pt\hrule%
         \hbox to \Wdd@{%
                 \vtop to \Htt@{\vfil}%
                 \vRule@\vRule@\vRule@\vRule@
                 \hfil\vvrule\hfil}%
         \hrule\vskip-.4pt}}%
  \wd\hGridbox@=0pt\ht\hGridbox@=0pt
  \wd\vGridbox@=0pt\ht\vGridbox@=0pt
  \hbox{\box\hGridbox@\box\vGridbox@}%
  }

 \def\LabelsGlobal{\def\LabGl@b{\global}}
 \def\LabelsLocal{\def\LabGl@b{}}
 \LabelsGlobal 

 \def\SetLabels#1\endSetLabels{%
   \LabGl@b\Labeltoks@={#1()\\}%
   }

 \LabGl@b\Labeltoks@={()\\}

 \def\ShowGrid{\LabGl@b\let\PlaceGrid@\PlaceGrid@@}
 \def\HideGrid{\LabGl@b\let\PlaceGrid@\relax}
 \def\Grids{\ShowGrid\LabGl@b\let\GridSwitch@\ShowGrid}
 \def\noGrids{\HideGrid\LabGl@b\let\GridSwitch@\HideGrid}

 \noGrids

 \def\bAdjust@@{%
     \setbox\auxLabelbox@=\hbox{\raise \dp\auxLabelbox@
            \box\auxLabelbox@}}
 \def\bAdjust@{\let\vAdjust@\bAdjust@@}

 \def\eAdjust@@{\dimen0=-.5\ht\auxLabelbox@
     \advance\dimen0 by .5\dp\auxLabelbox@
     \setbox\auxLabelbox@=
            \hbox{\raise\dimen0\box\auxLabelbox@}}
 \def\eAdjust@{\let\vAdjust@\eAdjust@@}

 \def\tAdjust@@{%
     \setbox\auxLabelbox@=\hbox{\raise-\ht\auxLabelbox@
            \box\auxLabelbox@}}
 \def\tAdjust@{\let\vAdjust@\tAdjust@@}

 \let\vAdjust@\relax

 \def\lAdjust@{\let\hAdjust@\rlap}
 \def\rAdjust@{\let\hAdjust@\llap}

 \let\hAdjust@\relax\let\vAdjust@\relax

 \def\FetchLabel@#1(#2)#3\\{%
     \IN@0#2@@\ifIN@
        \setbox0=\hbox{\ignorespaces#1#3\unskip}%
        \ifdim\wd0>0pt
           \ms@g{}%
           \ms@g{ !!! Bad label(s)? !!!}%
           \message{ #1(#2)#3}%
        \fi
        \def\LabelMole@##1\endFetchLabel@{%
            \IN@0()\\@##1@%
            \ifIN@\def\Temp@{\FetchLabel@##1\endFetchLabel@}%
            \else\def\Temp@{}%
            \fi
            \Temp@
           }%
     \else
       \ignorespaces#1\unskip
       \setbox\auxLabelbox@=%
         \hbox to 0pt{\hss\ignorespaces\hAdjust@
          {\ignorespaces#3\unskip}\hss}%
       \vAdjust@
       \let\hAdjust@\relax\let\vAdjust@\relax
       \AugmentLabelBox@@{#2}%
       \ht\Labelbox@=0pt\dp\Labelbox@=0pt
       \let\LabelMole@\FetchLabel@%
     \fi\LabelMole@}

 \newtoks\XYSep@ 
 \def\SetXYSeparator#1{%
     \IN@0#1@@\ifIN@\XYSep@{*}%
     \else
     \XYSep@{#1}%
     \fi
     }

 \SetXYSeparator*

 \def\AugmentLabelBox@@#1{%
     \IN@0\the\XYSep@ @#1@\ifIN@
       \SPLIT@0\the\XYSep@ @#1@%
       \setbox\Labelbox@=\hbox to 0pt{%
         \unhbox\Labelbox@
         \Shifted@@{\the\Initialtoks@\Wddd@}%
         {\the\Terminaltoks@\Httt@}%
         {\box\auxLabelbox@}}%
     \else
         \ms@g{}%
         \ms@g{ !!! Bad insertion point. !!!}%
         \message{ (#1\ this point was rejected.)}%
     \fi
    }

 \def\FetchOption@#1[#2]#3\endFetchOption@{%
    \def\temp{#1}
    \ifx\temp\empty
       \Edgect@=#2\relax
       \let\EdgGl@b\relax
       \SetEdges@
       \Cleaner@#3%
    \fi}

 \def\Cleaner@#1[@]{\Labeltoks@{#1}}
     
 \def\PlaceLabels@@{\mathsurround=0pt
     \def\Cr@{\\}%
     \let\L\lAdjust@\let\R\rAdjust@
     \let\B\bAdjust@\let\E\eAdjust@\let\T\tAdjust@
     \expandafter\FetchOption@\the\Labeltoks@[@]\endFetchOption@
     \Wddd@=\Wdd@ \Edgescale@\Wddd@ 
     \Httt@=\Htt@ \Edgescale@\Httt@
     \expandafter\FetchLabel@\the\Labeltoks@\endFetchLabel@
     \box\Labelbox@
     }%

 \let \PlaceLabels@\PlaceLabels@@

 \def\AffixLabels#1{\setbox\Coordinatebox@=\hbox{#1}%
      \Wdd@=\wd\Coordinatebox@ \Htt@=\ht\Coordinatebox@
      \advance\Htt@ \dp\Coordinatebox@
      \hbox{\copy\Coordinatebox@\kern-\Wdd@ 
           \Shifted@@{0pt}{-\dp\Coordinatebox@}%
           {\PlaceLabels@\PlaceGrid@}%
           \kern\Wdd@}%
      \GridSwitch@ 
      \LabGl@b\Labeltoks@{()\\}%
      }
 
   \let\wlog\LabelFigwlog@ld   
   \catcode`\@=\LabelFigCatAt  

                                By

              Raymond S\'eroul <A18645@FRCCSC21.BITNET>
                                and 
              Laurent Siebenmann <lcs@topo.math.u-psud.fr>
    
              VERSIONS: July 1991, Oct 1991, Jan 1992, July 1992

INTRODUCTION

      This labelling package is intended for TeX users who
rely on non-TeX sources for for their graphics inserts.  It
provides means for adding TeX labels to such inserts with a
minimum of fuss. 

       For most labels, TeX users have in the past found it
reasonably convenient to rely on non-TeX sources. Typical
occasions when an inescapable need for TeX labels seemed to
arise are

 (a) when the graphics program lacks certain exotic or complex
mathematical symbols

 (b) when the very highest typographical quality is wanted for the
labels

 (c) when labels included with the graphics fail to print, 
 and you cannot figure out why (cf. boxedeps.doc).  The labels
 provided by labelfig.tex are 100

       Since this package first appeared, many users, who in the
past scarcely dreamed of using TeX labels, have come to use
nothing but.  So it is now appropriate to add

Intoxication Warning:  TeX labels may be addictive and expensive. 

     If you have a fast preview you may disagree, and even find
that this package provides an agreeable paste-up environment; see
extra applications at end.

     Note to publishers: It is possible and convenient to ultimately
export the TeX labels produced by labelfig.tex to become an integral
part of the EPS file. This is often desired by a publisher who typically
uses an "upmarket" graphics or page layout program, with which the
staff is skilled in perfecting figures.  See Appendix I for
a recipe.

     The authors are grateful to Patrick Ion of Math Reviews for
helpful comments and encouragement.

BASIC INSTRUCTIONS

    After reading in the macro file using

preview or proof your figure with a coordinate grid printed on
top, by typing the following:

    \ShowGrid  
    \AffixLabels{<the graphics insertion>}

Here <the graphics insertion> is what you would type to insert
the graphics object alone without the grid.  This must provide
for the space around it. For example <the graphics insertion>
might well be \BoxedEPSF{MyFigure scaled 700} using the
boxedeps.tex macro package (from same source); this provides a
TeX box containing the encapsulated PostScript insert specified by
the file MyFigure. \AffixLabels{...} provides the grid (supposing
\ShowGrid is present) and later, once you have specified labels
using the grid, it will "tack on" the labels.

     The grid is a sort of (usually elongated) checkerboard of
ten rows and ten columns and its (internal) partitions are by
default numbered  .1, ... ,.9  both horizontally (X-coordinate
running left to right) and vertically (Y-coordinate running bottom
to top).  Thus the points enclosed by the grid correspond to the
points of the unit square in the cartesian "X-Y" plane, the lower
left corner corresponding to the origin (0,0).  By extrapolation,
the full page corresponds to a larger rectangle in the plane.

     These coordinates serve to position labels as follows.
Before the \AffixLabels{...} command type label specifications:

  \SetLabels
   (<X-coordinate>*<Y-coordinate>) <first label> \\
   .
   .
   .
   (<X-coordinate>*<Y-coordinate>)  <last label> \\
  \endSetLabels

Each row specifies one label and is terminated by \\.  In each
row, the position indicator comes first; it is written as a
standard cartesian point except that the X- and Y- coordinates
are separated by * rather than a comma because TeX allows a
comma as decimal point. There are no dimension units to specify
as the unit is the grid itself.

     By default, this cartesian point specifies where the middle
of the baseline of the label will be located.  However if you precede
the point by \L [or \R] the left [or right] edge of the baseline will
be located there. Similarly you may also precede the point by \T, \E,
or \B to vertically align the top equator or bottom of the label box
at the specified point.  This gives nine standard positions of
the label with respect to the insertion point --- corresponding to
the eight principle points of the compas and the center

                     \L\T     \T      \R\T

                     \L\E     \E      \R\E

                     \L\B     \B      \R\B

But this neglects the default "baseline" level of TeX,
giving potentially three more positions

                     \L    <no tag>   \R

For text, the baseline level is often the preferred. Its relation to
the others is variable. It will often coincide with the bottom level,
as happens for "X".  But it is often distinct, as for "g", in which
case you have in all 12 distinct positions rather than 9.

     It is convenient to think of this specification of label
position as attaching the label by a thumb-tack to the coordinate
grid. There are up to twelve positions of the thumb-tack on the
label, while the position of the thumb-tack on the coordinate grid is
arbitrary.  Normally, one choses the position of the thumb-tack on
the label to be the one that is the closest to the item being
labeled.  There are good reasons for this "rule of thumb":

   (a)  It facilitates correct positioning at first try.

   (b)  If the scale of the figure must be altered after labels
have been affixed, the labels have a good chance of remaining well
positioned.

   (c)  The visible grid need not extend beyond the "bounding box"
for the figure, because the best preferred position is always
(at least almost) within the bounding box .

The second reason is particularly important. Indeed it often
happens that scale has to be altered after labelling begins, in
order to either provide space for the labels, or to adjust
proportions between the labels and the figure.  (The size of labels
is unaffected by scaling.)

     Here is an artificial but self-contained test which uses
TeX rules to make a graphics object.

TEST

    Do not skip this!

 \def\FrameIt#1{\hbox{\vrule$\vcenter {\hrule\kern3pt%
             \hbox {\kern3pt #1\kern3pt}%
               \kern3pt\hrule}$\relax\vrule}}

 \def\Caption#1#2{\FrameIt{%
       \vtop {\hsize=#1\relax \parindent=0pt
         \leftskip=0pt \rightskip=0pt plus15pt
         \parfillskip=0pt
         \lineskip=1pt\baselineskip=0pt
         #2}}}

 \def\FirstQuadrant{\hbox to 100pt{\vrule\vbox to 100pt{%
        \hbox to 100pt{\hfil}\vfil\hrule}\hss}}

  \SetLabels
    \R(.5*.2) $\zeta\,\cdot$\\
    (.9*-.10) $\xi$\\
    \R(-.03*.9) $\eta$\\
    \T(.5*.9) \Caption{70pt}{%
          \it The norm of
          $g(\xi+i\eta)$ is indicated on
          contours of this invisible surface.}\\
  \endSetLabels

  \AffixLabels{\FirstQuadrant} \end

  Note that the coordinates to use for labels are indicated on the
edges of the grid (when visible) corresponding to the conventional
x- and y- axes of the Cartesian plane. By default the grid is
1-by-1. However, by the command \Edges{100}, you can change this
to 100-by-100 and many users find this alternative most
convenient. Place the command \Edges{...} in your style file (or
header) since its effect is is global. Other possible edge values
are 10 and 1000.

  If you use the command \Edges{...} at all, do so with care.  For
if you accidentally delete an \Edges{...} command your labels will
abruptly be badly misplaced and may logically but mysteriously
generate "dimension too big" errors under TeX and "off page" errors
under your driver.  

  You can dictate the edgescale for an individual figure by giving
the scale in brackets immediately after \SetLabels.  Thus, to
import into an article using say \Edge{100} a figure labelled using
another edgescale, say the original 1-by-1 default, you can use
\SetLabels[1]...\endSetLabels.

GETTING IT DOWN PAT

     Complicated labeling deserves the same respect as
complicated mathematics.  Do not expect it to come out perfect the
first time!  What is needed in either case is a mechanism to
repeatedly typeset troublesome pieces.

     One mechanism is always available.  One does complicated
labelling in a separate "test" file involving just the figure being
labelled;  a texpert will know how to \dump TeX's current state as
a temporary format that restarts rapidly at each retry.  Usually,
one then pastes the completed labelled figure back into the main
TeX file, but, of course, one can also \input it as an auxiliary
file.

     If you do not have a TeXpert at handy, here is a first
approximation to an efficient setup. By deletions reduce a copy
of your article to just a few lines before and after the figure.
Now label the figure, and finally, copy and paste the labelled
figure to the original article. Then copy the next figure to label
into this testbed and repeat. The TeXpert can improve the  speed
at which TeX starts up, by compiling a format specifically for
your article; just one caution: best NOT include in the format
ephemeral details of setup like \Set<mydriver>ArtSpecials (from
boxedeps.tex because this reads  figure dimensions which you may
change during your work session.

     An improved mechanism to repeatedly typeset troublesome
pieces is now available on the Macintosh; it is called LinoTeX;
see the same ftp sources.  It could be set up on many types
of computer.

     Before using labelfig.tex to attach labels to a graphics
object inserted using boxedeps.tex or BoxedArt.tex, make it a
firm rule to carefully adjust the bounding box using the trimming
commands of these packages, and also at least tentatively scale
and position the object. Beware of changing the grid inadvertently
after the labels have been positioned.  For example, correcting
the bounding box of a PostScript graphics object can foul up the
labels by changing the coordinate grid to which the labels are
attached. This is particularly true for the trimming  commands of
boxedeps.tex and BoxedArt.tex. However, as noted already, change
of scale is much less disruptive, and modest adjustments should be
well tolerated.

     Sometimes the labels protrude so far from the bounding box
of a figure that the figure has to be repositioned.  Best do this
by ad hoc spacing, say using \hglue and \vglue; altering the
bounding box would create a vicious circle.

     Remember that you are responsible for preventing labels
from overlapping. You are responsible for all label typography
including size and style. A label is really just about anything
that can be put in a TeX box. Note that spaces at the beginning
and end of labels will normally be suppressed; if you really want
them you must protect them with TeX braces.

     This package temporarily sets the \mathsurround parameter
of TeX to zero  while the labels are being affixed. This is done
because nonzero \mathsurround space would influence the position
of left and right aligned labels; then, when a texpert or printer
modifies mathsurround, diagram labeling might be disastrously
altered. There is a small price to pay involving labels that are
formatted as caption boxes including mathematics: you  may want or
need to specify an explicit mathsurround space within the caption
box; it will not influence anything outside.

     Those hostile to the use of * as separator between
the X and Y coordinates of label insertion points, are free to
impose another using \SetXYSeparator{<the new separator>}.  
Americans may prefer "," to "*" since they never use a 
comma as a decimal point; on the other hand, * may be more visible.

APPENDIX (I)  MERGING labelfig.tex LABELS INTO AN EPSF GRAPHICS OBJECT.

     As promised in the introduction, here is a recipe useful for
publishers. It works at least on Macintosh and at least for vectorized
graphics and Adobe type1 fonts.  (There is surely a similar recipe for
PCs under MSWindows.)

 (a)  Use boxedeps.tex utility to integrate the figure given by the eps
file, "x.eps" say, with a visible frame around it.  See
\ShowDisplacementBoxes command in boxedeps.tex.  To get precise results
automatically it is important to use the \Trim... commands of
boxedeps.tex making the "DisplacementBox" neatly fit the figure.

 (b)  Use the TeX printer driver and LaserWriter (versions >= 8.1.1) to
export to an EPSF the DVI page containing the integrated, labelled
figure. You now have an EPS file  "xx.eps"  that contains too much, and at
the wrong scale, and at wrong position.

 (c)  Convert the EPSF to an Adode Illustrator format EPSF using
the shareware utility called epsConvert by Sam Weiss
1993-- (currently $25).

 (d)  In Illustrator (or a compatible program), group the labels and the
"DisplacementBox"; copy them to the clipboard and paste them into "x.ps".
This step requires that all the label fonts be "visible to the Macintosh.

 (e)  Translate and scale the pasted group consisting of the labels plus
the "DisplacementBox" so as to make the "DisplacementBox" the bounding
box of (labelless) figure represented by "x.eps".  At this point the
labels will be correctly placed on the figure "x.eps".

 (f)  Ungroup and delete the "DisplacementBox".  The result is the
desired single EPS file, "x+.eps" say, It contains the original figure
plus its labels.  

     Using grouping and ungrouping appropriately in "x+.eps", a
publisher's staff can very efficiently improve label positions etc.

APPENDIX II)  SOME EXOTIC APPLICATIONS

     The grid of labelfig.tex is analogous to a light-table in
classical page makeup with wax or latex glue.  In principle, you
can use it to compose any page from its indivisible parts.  This
even has some of the artisanal charm of classical paste-up
provided you have a fast screen preview to make the process
"interactive".

     In practice labelfig.tex is a tool for nonstandard jobs.
Here are a few going beyond the labelling already discussed.

(I)  GRAPHICS INTEGRATION.

     This is accomplished by treating the imported graphics
objects as labels.  The underlying graphics object is then
typically an empty  \vbox to <dimension>{\vfill} in a TeX
\midinsert...\endinsert construction.  A label line
might be of the form

   (.1*.1) \\

The exact form of the special command varies from driver to
driver.  However, in the case of encapsulated PostScript graphics
(EPSF norm), by relying on boxedeps.tex, one can have the
following standard syntax (independant of driver  (see
boxedeps.doc for details.
  
  (.1*.1) \BoxedEPSF{MyFigure scaled <scale in mils>}\\

This may be slow since it requires TeX to read the PostScript
file to read bounding box using many complex macros.  So you
may want to try

  (.1*.1) \EPSFSpecial{MyFigure}{<scale in mils>}\\

which is fast and driver independant, but it squashes the
bounding box, normally to its lower left corner.

     Similarly for graphics of the Macintosh PICT norm ---
using BoxedArt.tex (same sources) in place of boxedeps.tex.

     This approach to integration is to be recommended when
one is assembling a composite graphics object.

 (II)  COMMUTATIVE DIAGRAM ENHANCEMENT

     Commutative diagrams or arrays of mathematical objects
connected by arrows of various sorts are common in mathematics.
The mathematical objects require the use of TeX.  Recently TeX
acquired a good collection of arrows of all slopes --- that of
LamSTeX --- plus pwerful macros to build the diagrams.

     However, even the LamSTeX collection is often
inadequate; it lacks for example double shafted arrows, dotted
arrows and curved arrows. Fortunately it is possible to produce
such arrows on an individual basis using sophisticated graphics
programs such as Illustrator and AldusFreehand (both serving
the EPSF norm) or using Metafont (with its public domain norm).
Since the creation of each new arrow is a work of love, you
probably want to limit the number of arrows by using LamSTeX
for most arrows. The 40K commutative diagram module of LamSTeX
has been adapted to work with AmSTeX and a copy may be posted
with LabelFig and related files. Unfortunately no one has yet
offered a version that works with Plain TeX or LaTeX.

       Suffice it here to say that when the exotic arrow has
been somehow imported into TeX, labelfig.tex treats it as a
label that one affixes to the commutative diagram.  Two other
steps will be treated in separate notes, namely the matter of
extracting the dimension specifications for the arrow and the
construction of the arrow --- for these steps are far from
unique and often depend intimately on your computer environment. 
Notes for the Macintosh-Textures-Illustrator combination are
found in the file ExoticArrows.doc.

 (III) NESTING 

Ingenuity pays off in exploiting labelfig.tex. One can
mix graphics and typography quite freely.  labelfig.tex is good
for freeform or overlapping arrangements, while boxedeps.tex (or
BoxedArt.tex) is best for regimented non-overlapping
arrangements --- and the two can be combined.

     The default behavior of labelfig.tex is not ideal 
for nesting objects, because to prevent trouble for beginners
the register for labels is globally cleared when \AffixLabels
concludes.  But there are switches available

      \LabelsGlobal      \LabelsLocal

which change this.  To understand this, extend the above test 
by something like:

 \LabelsLocal

 \SetLabels
    (.5*.5) AAA\\
 \endSetLabels

 {
 \SetLabels
    (.5*.5) ZZZ\\
 \endSetLabels
   \AffixLabels{\FirstQuadrant}
 }

   \AffixLabels{\FirstQuadrant}

     There are however potential pitfalls.  Neither
labelfig.tex nor boxedeps.tex has been tested under extreme
conditions. Problems may occur if their procedures are
indiscriminately nested. For boxedeps.tex (not labelfig.tex)
there is a precise cause for worry, namely many of its
variables are "global", which means that TeX braces will not
provide the protection one might expect.

COMMAND SUMMARY FOR labelfig.tex

  Here [...] means optional (one or zero)
       [...]* means any number of such constructs

  \SetLabels
    [[<P>](<X><Sep><Y>) <label> \\]*
  \endSetLabels
  \ShowGrid  
  \AffixLabels{<the figure>}

   --- <P> is tack position, one of eleven or empty
              order irrelevant

                   \L\T      \T      \R\T

                   \L\E      \E      \R\E

                     \L               \R

                   \L\B      \B      \R\B

   --- (<X><Sep><Y>) insertion point;
  <Sep> is separator, = * by default;
  \SetXYSeparator{<Sep>} changes it.
   <X> and <Y> are real numbers

  --- <label> a label to attach 

  --- <the figure> the figure to label 

  \GlobalLabels (default)     
  \LocalLabels  setting for nested constructs.

 \Grids makes ALL grids appear; \HideGrid then makes just next disappear.
 \noGrids returns to default.  The commands are always global.

 \GridLineWidth{<dimension>} adjusts width of grid lines. Default is very
small, to give "hairline" effect. If your grid lines are missing try
setting \GridLineWidth{1pt}.

 \Edges#1 globally changes the edge size of all grids to the numerical 
value #1, which must be 1, 10, 100, or 1000.  The default is 1.

VERSION HISTORY.
 --- Jan 1993: \Edges#1 and [??] option after \SetLabels
 --- July 1992: \Grids, \noGrids, \HideGrid;
       Gridlines become hairlines; \GridLineWidth{<dimension>}.
 --- Oct 1991, Jan 1992: \SetXYSeparator{<Sep>},  \LabelsGlobal,
       \LabelsLocal.
 --- July 1991: first release

Address for bugs and other feedback:

        Raymond S\'eroul
        IREM and Lab. de Typographie Informatise
        Univ. Rene Descartes
        Strasbourg

    Tel 33-88-41-63-45
    Email:  A18645@FRCCSC21.BITNET

        Laurent Siebenmann
        Mathematique, Bat. 425,
        Univ de Paris-Sud,
        91405-Orsay,
        France

    Tel 33-1-6941-7949; 
    Email: lcs@topo.math.u-psud.fr

\theoremstyle{plain}

\newtheorem{theorem}{Theorem}[section]
\newtheorem{lemma}[theorem]{Lemma}
\newtheorem{corollary}[theorem]{Corollary}
\newtheorem{proposition}[theorem]{Proposition}
\newtheorem{claim}[theorem]{Claim}

\newtheorem{question}[theorem]{Question}

\newtheorem{definition}[theorem]{Definition}

\theoremstyle{definition}

\newtheorem{remark}[theorem]{Remark}

\numberwithin{equation}{section}
\numberwithin{figure}{section}

\newcommand{\MR}[1]{MR#1}
\newcommand{\ARXIV}[1]{arXiv:#1}
\newcommand{\DOI}[1]{DOI:#1}

\newcommand{\veps}{\varepsilon}
\newcommand{\ind}{\mathds{1}}

\newcommand{\lora}{\longrightarrow}

\newcommand{\N}{\mathbb{N}}
\newcommand{\Z}{\mathbb{Z}}
\newcommand{\R}{\mathbb{R}}

\newcommand{\fsubset}{\subset \subset}
\newcommand{\dist}{\mathrm{dist}}
\newcommand{\ball}{\mathcal{B}}
\newcommand{\volume}{\mathrm{Vol}}

\newcommand{\intboundary}{\partial_{\mathrm{int}}}
\newcommand{\extboundary}{\partial_{\mathrm{ext}}}

\newcommand{\grp}{\Gamma}
\newcommand{\Aut}{\mathrm{Aut}}
\newcommand{\cost}{\mathsf{cost}}
\newcommand{\Cay}{\mathsf{Cay}}
\newcommand{\subgrp}{\leq}

\newcommand{\origin}{o}
\newcommand{\edge}{\sim}
\newcommand{\graph}{G}
\newcommand{\degree}{d}
\newcommand{\tree}{\mathbb{T}}

\newcommand{\Cheeger}{\mathrm{h}}
\newcommand{\spectral}{\rho}

\newcommand{\trace}{\mathrm{Tr}}

\newcommand{\E}{\mathbb{E}}
\newcommand{\Var}{\mathrm{Var}}
\renewcommand{\P}{\mathbb{P}}
\newcommand{\prob}{\P}

\newcommand{\POI}{\mathsf{Poisson}}
\newcommand{\Bernoulli}{\mathsf{Bernoulli}}

\newcommand{\Falg}{\mathcal{F}}

\newcommand{\Green}{\mathcal{G}}
\newcommand{\hitTime}{T^+}

\newcommand{\capacity}{\mathrm{cap}}

\newcommand{\cL}{\mathcal{L}}
\newcommand{\cE}{\mathcal{E}}

\newcommand{\latticeanimals}{\mathcal{H}}
\newcommand{\wormnu}[1]{\nu^{\text{worm;}{#1}}}

\newcommand{\zoospace}{Z}
\newcommand{\zooPP}{\mathbf{Z}}
\newcommand{\zooalgebra}{\mathcal{Z}}
\newcommand{\zooPPalgebra}{\boldsymbol{\mathcal{Z}}}
\newcommand{\animals}{\mathcal{X}}

\newcommand{\LTmeasure}{\mu}
\newcommand{\totalsize}{\Sigma}

\begin{document}


\title{Nonamenable Poisson zoo}

\author{G\'{a}bor Pete\footnote{
HUN-REN Alfr\'ed R\'enyi Institute of Mathematics, Re\'altanoda u. 13-15, Budapest 1053 Hungary, and Department of Stochastics, Institute of Mathematics, Budapest University of Technology and Economics, M\H{u}egyetem rkp.~3., Budapest 1111 Hungary. Emails: gabor.pete AT renyi.hu, roksan AT math.bme.hu}
\and 
S\'andor Rokob$^*$}

\date{\today}

\maketitle

\begin{abstract}
In the Poisson zoo on an infinite Cayley graph $\graph$, we take a probability measure $\nu$ on rooted finite connected subsets, called lattice animals, and place i.i.d.~Poisson($\lambda$) copies of them at each vertex. If the expected volume of the animals w.r.t.~$\nu$ is infinite, then the whole $\graph$ is covered for any $\lambda>0$. If the second moment of the volume is finite, then it is easy to see that for small enough $\lambda$ the union of the animals has only finite clusters, while for $\lambda $ large enough there are also infinite clusters. Here we show that:
\begin{enumerate}
\item If $\graph$ is a nonamenable free product, then for ANY $\nu$ with infinite second but finite first moment and any $\lambda>0$, there will be infinite clusters, despite having arbitrarily low density.
\item The same result holds for ANY nonamenable $\graph$, when the lattice animals are worms: random walk pieces of random finite length.
\end{enumerate}
It remains open if the result holds for ANY nonamenable Cayley graph with ANY lattice animal measure $\nu$ with infinite second moment.
\begin{enumerate}
\item[3.]
We also give a Poisson zoo example $\nu$ on $\tree_\degree\times \Z^5$ with finite first moment and a UNIQUE infinite cluster for any $\lambda>0$.
\end{enumerate}

\medskip

\noindent \textsc{Keywords:} random walk interlacements, Poisson Boolean model, invariant percolation, phase transitions, nonamenable groups, measurable cost

\medskip

\noindent \textsc{AMS MSC 2020:} 60K35, 82B41, 37A20

\end{abstract}

\tableofcontents

\section{Introduction}
\label{s.intro}

\subsection{Results and questions}
\label{ss.results}

The \textbf{Poisson zoo} is a very simple correlated {site percolation} model on infinite transitive graphs, introduced in this generality by Ráth and Rokob \cite{RathRokob2022}. Intuitively, we drop random rooted connected sets, called \textbf{lattice animals}, at each vertex of the graph, independently, with intensity $\lambda>0$. One could think of a random collection of balls with random radii, or of random walk trajectories with some random finite lengths. In this paper, we will establish quite general scenarios when we can take $\lambda$ arbitrarily small, making the overall coverage density arbitrarily small, but still, large animals occur often enough to find each other and produce infinite connected components. This behaviour is very different from i.i.d.~Bernoulli percolation, and is interesting for several reasons.

More precisely, we start with a locally finite infinite graph $\graph=(V,E)$ and a subgroup of graph automorphisms $\grp\leq \Aut(\graph)$  that acts transitively on $V$. We also have a probability measure $\nu=\nu_o$ on rooted connected sets $(o,H)$ with $o\in H \subseteq V$, called {lattice animals}, where the measure is invariant under root-preserving automorphisms in $\grp$. For any $x\in V$, we can take any graph automorphism $\varphi_{x} \in \grp$ taking $o$ to $x$, and translate $\nu_o$ to $\nu_x := \nu_o \circ \varphi_{x}^{-1}$. It is easy to see that, because of the invariance of $\nu_o$ under root-preserving automorphisms, $\nu_x$ does not depend on the choice of $\varphi_{x}$. (See Subsection~\ref{subsection:poisson_zoo_on_transitive_graphs} for more details.) We fix $\lambda>0$. For each $x\in V$, we  sample an independent $\POI(\lambda)$ variable $N_x^\lambda$, then $N_x^\lambda$ independent animals from $\nu_x$, denoted by $\{H^x_i\}_{i=1}^{N_x}$. We then consider the subset of $V$ covered by all the animals,
\begin{equation}\label{e.zoo}
	\mathcal{S}_\nu^\lambda := \bigcup_{x\in V} \bigcup_{i=1}^{N_x} H^x_i\,,
\end{equation}
and the main question is whether
\begin{equation}\label{e.lambdac}
	\lambda_c = \lambda_c(\graph,\nu) := \inf\left\{ \lambda > 0 :  \P\big(\mathcal{S}_\nu^\lambda \textrm{ has some infinite connected component}\big)>0  \right\}
\end{equation}
is strictly between 0 and $\infty$. In the simplest example, Bernoulli site percolation, where animals are just single vertices, $0< \lambda_c$ holds for any bounded degree graph \cite[Theorem 6.47]{LyonsPeres2016}, while the much more difficult direction, $\lambda_c < \infty$, is known for all Cayley graphs that are not finite extensions of $\Z$ \cite{Duke} and for all uniformly transient graphs, even without transitivity \cite{EST}.

To state our results, we need one more technical but crucial assumption: $\graph$ should be not only transitive, but also, the action of $\grp$ on it should be {unimodular}; see Definition~\ref{def:unimodularity} below. For instance, any finitely generated Cayley graph $\graph$ of any group $\grp$, with the group acting on $\graph$, is fine. Unimodularity ensures the following basic facts; see Lemmas~\ref{lemma:two_silly_lemmas} and~\ref{lemma:size_biasing_for_Poisson_zoo}:
\begin{itemize}
	\item If $\E_{\nu} |H| = \infty$, then, for any $\lambda > 0$, the entire graph is covered by lattice animals.
	\item If  $\E_{\nu} |H| < \infty$, then the number of animals covering any given vertex $x$ has distribution $\POI\big( \lambda \E_{\nu} |H| \big)$. In particular, for small $\lambda$, the density of covered vertices is small.
	\item If $\E_{\nu} |H|^2<\infty$, there is a percolation phase transition: $\lambda_c(\graph,\nu) \in (0,\infty)$. The main reason is that the expected total size of animals covering a given vertex is $\lambda\, \E_{\nu} |H|^2$. Note the size-biasing effect: the second moment of the animal size governs the first moment in any cluster exploration.
\end{itemize}

On $\R^d$, if the animals are \textit{balls} of random radii, this is called the \textbf{Poisson Boolean model}, and Gou\'er\'e \cite{Gouere2008} proved that $\E_{\nu} |H| < \infty$ (now $|H|$ denoting Lebesgue measure) already implies the existence of a non-trivial $\lambda_c \in (0,\infty)$. This was extended by an easy coupling argument to graph metric balls in $\Z^d$ in \cite{RathRokob2022}, and by copying the proof of \cite{Gouere2008} to transitive graphs of polynomial volume growth in \cite{ColettiMirandeGrynberg2020}. The easy general facts displayed above were noticed by R\'ath and Rokob \cite{RathRokob2022}, who proved that, for \textbf{worms}, where the animals are random walk trajectories of a random length, for $d\ge 5$, already  $\E_{\nu}|H|^{2-o(1)}=\infty$ suffices for $\lambda_c=0$, where the $o(1)$ power is just a poly-$\log\log$ factor. It has remained open if $\E_{\nu}|H|^{2}=\infty$ for worms is actually sufficient for $\lambda_c=0$; see \cite[Question 1.11]{RathRokob2022}. In any case, that work has shown that balls and worms are almost on the two extremes of the possible behaviours on $\Z^d$, $d\ge 5$.

In the present paper, we close the gap for worms on any nonamenable unimodular transitive graph, and show that the difference between worms and balls and any other lattice animal measure disappears in many (possibly all?) nonamenable examples.

\begin{theorem}[Random length worms on any nonamenable group]\label{t.worms}
	If  $\graph$ is \emph{any} nonamenable unimodular transitive graph, and the lattice animals are simple random walk trajectories of random length, satisfying $\E_\nu |H|^2=\infty$, then the Poisson zoo has $\lambda_c=0$.
\end{theorem}

\begin{figure}[htbp]
	\SetLabels
	(0.7*0.35)\textcolor{red}{$E_{n-1}$}\\
	(1*0.35)\textcolor{blue}{$E_{n}$}\\
	\endSetLabels
	\centerline{
		\AffixLabels{
			\includegraphics[width=2.8 in]{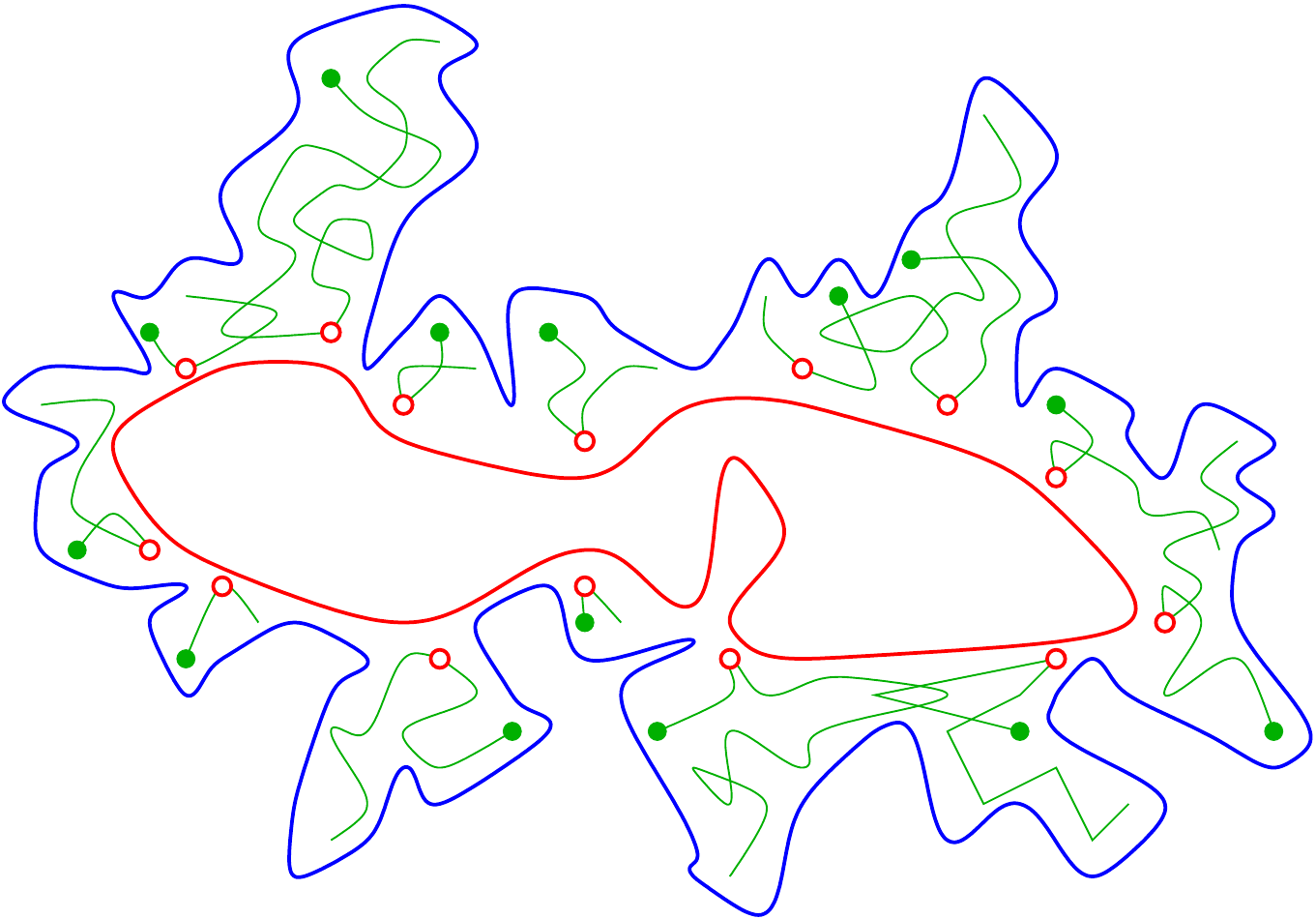}
		}
	}
	\caption{Exploring a cluster: at any given stage $E_{n-1}$, there are many exposed vertices on its boundary where new worms (the green trajectories) can touch it, creating a much larger $E_n$.}
	\label{f.lincap}
\end{figure}

The key advantage of nonamenable transitive graphs compared to $\Z^d$, high $d$, is that the  \textbf{random walk capacity} of \textit{any} finite set $S\subset V(\graph)$ is \emph{linear} in the volume $|S|$; see Lemma~\ref{lemma:capacity_bound}, first proved possibly in \cite{Teixeira2009} and \cite{BenjaminiNachmiasPeres2011}. This makes it possible to build an exploration process, with enough new worms touching the already explored cluster from the outside for the exploration never to stop; see Figure~\ref{f.lincap}. The supercriticality of the exploration, for an arbitrary $\lambda>0$, will be fuelled by $\E_\nu |H|^2=\infty$: after size-biasing, this corresponds to the exploration step having an infinite mean. The difficulty in the argument is that the group is completely general (besides being nonamenable), we do not have any structural information, hence neither the dynamic renormalization approach of \cite{RathRokob2022} makes any sense, nor the exploration could be done in a usual branching process manner. Instead, we will make sure that the exploration keeps going forever with positive probability by using the trick that a successful induction may work by requiring more than what is minimally necessary: not only will the exploration keep going, but it will grow fast. The argument will rely on first and second moment estimates for the growth in each step of the exploration.

\begin{theorem}[General animals on free products]\label{t.free}
	If $\graph$ is a nonamenable  unimodular transitive graph that is obtained as a {free product} $\graph_1 \star \graph_2$, then the Poisson zoo with \emph{any} lattice animal measure $\E_\nu |H|^2=\infty$ satisfies $\lambda_c=0$.
\end{theorem}

For the definition of the free product of transitive graphs, see Definition~\ref{def:free_product}, and for two examples, Figure~\ref{f.free}. The simplest example is a regular tree, and our result is new and interesting already there. The proof is again via an exploration process, with the main idea being that, even without an exact understanding of how new pieces can touch an already explored part of the cluster (which was provided by the notion of  random walk capacity in the case of worms), the free product structure provides enough independent tries in disjoint areas of the graph for a branching process argument. One subtlety here is that the two factors in the free product can behave very differently, and the infinite second moment of $\nu$ may correspond to growth only in one of the factors. We will handle this asymmetry by sprinkling.

\begin{figure}[htbp]
	\SetLabels
	(-0.1*0.2){$\Z_3 \star \Z_4$}\\
	(1.05*0.2){$\Z \star \Z$}\\
	\endSetLabels
	\centerline{
		\AffixLabels{
			\includegraphics[width=3.3 in]{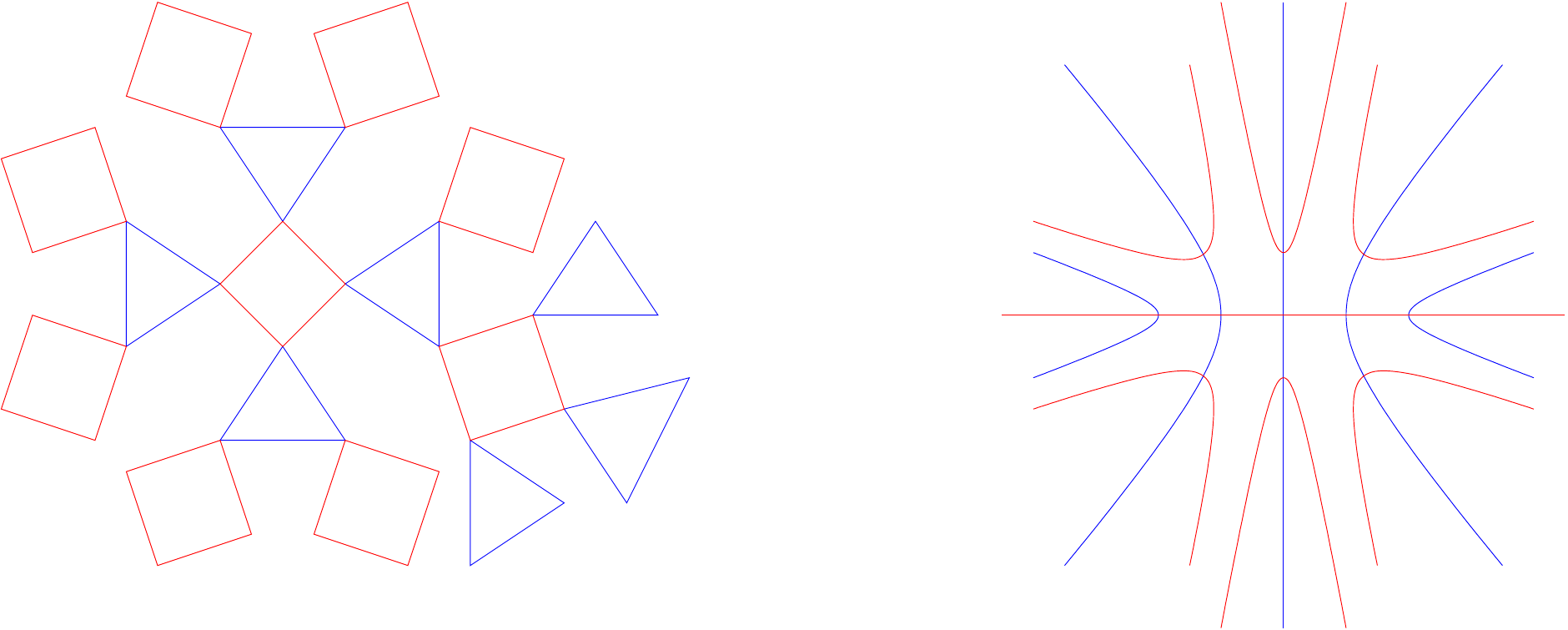}
		}
	}
	\caption{Two examples of free products of graphs.}
	\label{f.free}
\end{figure}

The simplest joint generalization of our two theorems would be the following.

\begin{question}[General animals on general nonamenable groups]\label{q.general}
	If  $\graph$ is \emph{any} nonamenable unimodular transitive graph, is  $\lambda_c=0$ for \emph{every} lattice animal measure with  $\E_\nu |H|^2=\infty$? 
\end{question}

For Bernoulli and many other percolation processes on nonamenable groups, besides the appearance of infinite clusters, there is typically a second important phase transition, which we can also define for the Poisson zoo:
\[
\lambda_u=\lambda_u(\graph,\nu) := \inf\left\{ \lambda > 0 :  \P\big(S_\nu^\lambda \textrm{ has a unique infinite cluster}\big)>0  \right\}.
\]
In the case of Bernoulli percolation (site or bond, does not matter), a well-known conjecture of Benjamini and Schramm \cite{BenjaminiSchramm1996} is that any nonamenable Cayley graph $\graph$ has $0< p_c(\graph) < p_u(\graph) \leq 1$; it has been proved in \cite{PSN,Thom} that any finitely generated nonamenable group $\grp$ has a finite generating set $S$ such that the Cayley graph $\graph=\Cay(\grp,S)$ satisfies this. Furthermore, one-ended Cayley graphs $\graph$ (meaning that removing any finite subset of the vertices results in a single infinite connected component) should always satisfy $p_u(\graph) < 1$, proved for finitely \emph{presented} ones \cite{BaB, Timar:cut}. 

For Bernoulli percolation and, more generally, any insertion tolerant invariant percolation, it was proved by Lyons and Schramm \cite{LySch} that infinite clusters, whenever they exist, are \textbf{indistinguishable} by $\grp$-invariant properties (e.g., either all infinite clusters are recurrent, or all of them are transient), and, as a consequence of this, uniqueness monotonicity holds: for every $p > p_u$, there is a unique infinite cluster. The Poisson zoo is not quite insertion tolerant, but is very close to it, and indistinguishability of infinite clusters is indeed proved in \cite{Damis}, also implying uniqueness for any $\lambda > \lambda_u(
\nu)$. 

Regarding the existence of an intermediate regime of non-unique infinite clusters, it was proved by Bowen \cite{Bowen2019}, motivated by the so-called dynamical von Neumann-Day problem \cite{GaboriauLyons}, that for any nonamenable Cayley graph and any $\veps>0$ there exists a worm measure $\nu$ with intensity $\lambda$ (called finitary interlacements there) such that the density is $\lambda \, \E_\nu |H| < \veps$ and there are infinitely many infinite clusters. 

However, there exist nonamenable Poisson zoos where there are no three phases, just one. A comment from Tom Hutchcroft, together with the results of \cite{RathRokob2022} and the indistinguishability of infinite clusters, easily imply the following, interesting for $d\ge 3$:

\begin{proposition}\label{p.immuniq}
	There is a Poisson zoo $\nu$ on $G=\tree_\degree\times \Z^5$ with $\E_\nu|H|<\infty$ and $\lambda_u(G,\nu)=0$.
\end{proposition}

How widespread is this phenomenon of immediate uniqueness? This question brings us to a small detour, the next subsection,  with Question~\ref{q.unique} as its main point. If the Reader finds this discussion too terse, they may just skip it, except for the observation that the Poisson zoo is always a factor of IID process, which implies that it is ergodic.

\subsection{One motivation: groups with FIID sparse unique clusters and fixed price 1}\label{ss.cost}

Besides Question~\ref{q.general}, another motivation for this work (even though we will not make here real progress in this direction) comes from the theory of \textbf{measurable cost} for groups, Cayley graphs, and probability measure preserving (p.m.p.)~group actions, introduced in \cite{Levitt} and studied by Gaboriau in several influential papers; see \cite{Gabo:free, Gabo:ICM}. It is an ergodic theoretic or probabilistic version of the rank of a group (which is the minimal number of generators needed) that has actually turned out to be a more robust, more interesting invariant of groups and p.m.p.~group actions than the algebraic rank itself (see \cite{Abert:ICM} for this story). Let us give here just the following two definitions, for Cayley graphs $\graph=\Cay(\grp,S)$ given by a finite generating set $S$ of $\grp$:
\begin{equation}\label{e.cost}
	\cost(\grp,S) := \frac12 \inf \left\{ \E_\mu[\deg(o)] \,\left|\; 
	\parbox{22em}{$\mu$ { is a }$\grp$-invariant probability measure on\\
		connected spanning subgraphs of $\Cay(\Gamma,S)$} \right. \right\};
\end{equation}
\begin{equation}\label{e.costFIID}
	\cost^*(\grp,S) := \frac12 \inf \left\{ \E_\mu[\deg(o)] \,\left|\; 
	\parbox{22em}{$\mu$ is a factor of i.i.d.~(FIID) measure on\\
		connected spanning subgraphs of $\Cay(\Gamma,S)$} \right. \right\}.
\end{equation}

Here, \textbf{factor of i.i.d.~(FIID)} means that the measure $\mu$ is a $\grp$-equivariant measurable function $\psi$ of i.i.d.~variables $\xi=\{ \xi_v \sim \mathsf{Unif}[0,1] : v\in V(\graph) \}$ or $\{\xi_e \sim \mathsf{Unif}[0,1] : e\in E(\graph)\}$:
\[
	\psi(\xi) \sim \mu, \hskip 2 cm  \gamma \cdot \psi(\xi) = \psi(\gamma\cdot \xi), \quad  \forall \gamma \in \grp,
\]
with the usual translation action of $\grp$ on a configuration $\xi$ or $\psi(\xi)$, namely, $(\gamma\cdot\xi)(e) = \xi(\gamma^{-1}(e))$ for every $e\in E$, or similarly for $v \in V$. Equivalently (but somewhat loosely), one has a measurable (i.e., locally approximable) coding map that decides the status of each edge $e$ (or vertex $v$) from the i.i.d.~variables viewed from $e$ (or $v$). See, e.g.,~\cite{Lyons} for an introduction to FIID processes. Three examples of interest to us that also help digest the definition:
\begin{itemize}
\item Take i.i.d.~Bernoulli$(p)$ bond percolation with $p>p_c$, then delete every finite cluster (clearly an equivariant measurable procedure). 
\item The Wired Minimal Spanning Forest, where, given $\{\xi_e \sim \mathsf{Unif}[0,1] : e\in E\}$, we delete the edge with the largest label in every cycle, including ``generalized cycles'' going through infinity; see, e.g., \cite[Section 11]{LyonsPeres2016}. 
\item An FIID \textit{site} percolation is the set $S^\lambda_\nu$ covered by our Poisson zoo: \eqref{e.zoo} is clearly a FIID  construction.
\end{itemize}
An easy but important corollary of being an FIID process is {\bf ergodicity} w.r.t.~the action of $\grp$. For a proof, just approximate the process by a block factor, meaning that the coding map that determines the value at a vertex (or edge) depends only on a bounded neighbourhood of the vertex (or edge), then run the analogue of the usual proof of ergodicity for Bernoulli percolation, e.g., \cite[Lemma 12.4]{PGG}. 

Now, back to cost, the reason for singling out FIID measures in \eqref{e.costFIID} is that a finer version of cost can be defined for p.m.p.~group actions (using the notion of graphings, see \cite{LovaszBook} and \cite{Gabo:ICM}), so that the cost of a group is the infimum cost among all its free actions, while a theorem of Ab\'ert and Weiss \cite{AW} implies that FIID actions have the supremum cost. The \textbf{fixed price} question of Gaboriau \cite{Gabo:free} is if all free actions of any given group actually have the same cost, and if it is also independent of the generating set; that is, whether for any group $\grp$ and any generating set $S$, we have $\cost(\grp,S) = \cost^*(\grp,S)=\cost(\grp,\grp)$. 

It is well-known that infinite amenable Cayley graphs have fixed price 1; this follows, e.g., from \cite[Theorem 5.3]{BLPS:GAFA}. So, all amenable groups are ``as small as $\Z$'' in the sense of cost. It was proved in \cite{Gabo:free} that free groups on $d$ free generators, with any finite generating set $S$, have fixed price $d$. This was the first example of a non-trivial measured group theoretical invariant capturing some sort of ``largeness''. But there also exist nonamenable Cayley graphs with fixed price 1, with several recent breakthroughs: lattices in higher rank Lie groups, using Ideal Poisson-Voronoi Tessellations \cite{FMW}, and the direct product of any two infinite groups, using a similar but simpler argument \cite{Khezeli}. It is quite open exactly which groups have cost 1 or fixed price 1; another well-known question of Gaboriau is whether they are characterized by their \textbf{first $\ell^2$-Betti number} being 0, which means having no non-constant harmonic functions with finite Dirichlet energy \cite[Section 10.8]{LyonsPeres2016}. As an important step in this direction, Hutchcroft and Pete \cite{HuPe} proved that infinite Kazhdan (T) groups have cost 1, but their construction is not at all FIID, hence fixed price remains open. 

That last proof starts with the observation \cite[Proposition 2.1]{HuPe} that if $\Cay(\grp,S)$ has, for any $\veps>0$, an invariant site or bond percolation with density at most $\veps$ and a unique (thus infinite) cluster, then $\cost(\grp,S)=1$. Moreover, if this \textbf{arbitrarily sparse unique cluster (SUC)} can be given in a FIID way, then $\cost^*(\grp,S)=1$. Since the Poisson zoo is always a FIID measure, and, as our above results show, it often produces infinite clusters already at arbitrarily small densities, sometimes even a unique one, it makes sense to ask the following, aiming to generalize Proposition~\ref{p.immuniq}:

\begin{question}[FIID sparse unique cluster]\label{q.unique}\ 
	\begin{enumerate}
		\item[\textbf{(a)}]  Give interesting examples of nonamenable Cayley graphs $\graph$ with FIID percolations with a sparse unique cluster. 
		\item[\textbf{(b)}] In particular, does every Cayley graph with first $\ell^2$-Betti number being 0 admits, for \emph{every} $\veps>0$,  some lattice animal measure $\nu$ with $\E_\nu |H| < \infty$ such that $\lambda_u(\graph,\nu) \, \E_\nu |H| < \veps$? Or even a single lattice animal measure with $\lambda_u(\graph,\nu)=0$?
	\end{enumerate}
\end{question}

Having FIID sparse infinite clusters, without the requirement of uniqueness, is not difficult:
\begin{itemize}
	\item Consider i.i.d.~{Bernoulli percolation} on a Cayley graph at $p=p_c+\veps$. Assuming $\theta(p_c)=0$, proved for nonamenable Cayley graphs in \cite{BLPS:GAFA}, the density of infinite clusters is small. Delete all finite clusters.
	\item Consider the {Wired Minimal Spanning Forest}, mentioned above. This is infinitely many one-ended trees for any nonamenable Cayley graph \cite[Theorem 11.12]{LyonsPeres2016}. Prune the leaves, repeatedly, $N$ times. For large enough $N$, the overall density will be small.
	\item Take random interlacements \cite{Sznitman2010, DrewitzRathSapozhnikov2014}, the canonical Poisson point process of bi-infinite simple random walk trajectories, at low intensity. The set of covered points is a FIID process \cite{BorbenyiRathRokob}, for non-trivial reasons: as opposed to worms, a trajectory here does not have a starting point where the random choices for the entire trajectory could be made locally.
\end{itemize}
However, on nonamenable Cayley graphs, these constructions inherently give infinitely many infinite clusters: for Bernoulli percolation, this is the $p_c < p_u$ conjecture; for interlacements, this is proved in \cite{TeixeiraTykesson2013}. The Poisson zoo has a chance to yield unique sparse infinite clusters on many nonamenable groups, Proposition~\ref{p.immuniq} hopefully being only a first example.

Very recently, triggered by our work, the preprint \cite{GrebReck} has strengthened \cite{FMW} to produce FIID sparse unique clusters on lattices in higher rank Lie groups with Kazhdan's property (T), then the preprint \cite{DAGKRW} has combined the ideas from \cite{FMW,GrebReck,Khezeli} to produce FIID SUCs in certain products, such as $\tree_k\times\tree_k$. So far, it seems to be quite a bit harder to produce FIID SUCs than fixed price 1. Actually, another recent development \cite{CsMP,Mikolaj} is that FIID sparse clusters (unique or not) in many groups (in fact, all exact groups, a.k.a.\ having property (A); see those two papers for the definitions and references) are necessarily ``thin'', meaning ``almost hyperfinite'', and this {\it might} exclude the possibility of their uniqueness in some of these groups (say, in all Gromov-hyperbolic groups). But even accepting this guess from \cite{CsMP,Mikolaj}, it remains presently unclear if it should mean that the FIID SUC property is strictly stronger than fixed price 1, or that fixed price 1 should fail, e.g., for hyperbolic Kazhdan (T) groups.

Let us finally mention that a positive answer to Question~\ref{q.general} would solve another question from \cite{CsMP,Mikolaj}: in non-exact groups, the Poisson zoo supported on large small-scale expanders would produce sparse but not thin FIID clusters, not just weak FIID clusters (meaning weak limits of FIIDs) that is presently achieved in~\cite[Theorem 1.3]{CsMP}, building on~\cite{Krakow}.

\subsection{Further related works}\label{ss.litera}

The introduction of the general Poisson zoo in \cite{RathRokob2022} was inspired by several Poisson soup percolation models, as follows. 

The first one is the \textbf{Poisson Boolean model} with random radii in $\R^d$, see, e.g., \cite{MeesterRoy1996, Gouere2008, DeMu, DeTa} and the many references therein, and with a fixed radius in hyperbolic space \cite{Tykesson2007, Tykesson2009}, where results similar to nonamenable Bernoulli percolation have been shown, such as $0<\lambda_c<\lambda_u<\infty$. Percolation results in spaces beyond $\R^d$ seem to stop here; there are some results about visibility of the ideal boundary in the vacant set \cite{BenjaminiJonassonSchrammTykesson2009, TykessonCalka2013}.

The second main family of models is \textbf{random walk loop soups} \cite{Lupu,ChangShapoznikov2016}, especially interesting for their relationships to level sets of the Gaussian Free Field, and the related models of finitary interlacements (a special case of the worms model), studied in \cite{Bowen2019, PYZ, CPZ22}.

Earlier results on {\bf permanent supercriticality}, meaning $\lambda_c=0$, were \cite[Theorem 1.2]{TeixeiraUngaretti2017} on {ellipses percolation} in $\R^2$, and \cite[Remark 1.1 and Claim 1.2]{Chang2021} on Bernoulli hyper-edge percolation on $\Z^d, \, d \geq 3$.  In the recent preprint \cite{convexgrain}, for the Poisson Boolean model in $\R^d$ with quite general random convex shapes, criteria for $\lambda_c=0$ (called ``robustness'' there) were given. 

One may also consider Poisson soups of infinite objects. The \textbf{random interlacements process} on $\Z^d$, already mentioned above, was introduced by Sznitman \cite{Sznitman2010}, as the local distributional limit of simple random walk on a $d$-dimensional torus $(\Z / N \Z)^d, \, d \geq 3$, where we run the walk up to times comparable to the volume of the torus, and let $N \to \infty$. The limit can be described as the trace of a Poisson point process on the space of labeled bi-infinite random walk  trajectories modulo time-shift, and it is a canonical object related to capacity. This description also works on any transient edge-weighted graph \cite{Teixeira2009}. A percolation phase transition that makes sense in this setting is whether all these infinite trajectories form a single infinite component, and it was proved in \cite{TeixeiraTykesson2013} that this is the case for all positive intensities if{f} the underlying graph is amenable. The harder direction here is the existence of a non-uniqueness phase in the nonamenable case, proved by dominating the interlacements set by a certain branching random walk. Another phase transition to study is percolation in the vacant set, a recent breakthrough on $\Z^d$ being \cite{DukeEquality}. To our nonamenable setting, more relevant is the proof of the existence of a supercritical phase for percolation in the vacant set in \cite[Theorem 5.1]{Teixeira2009}, using a branching process idea, somewhat similarly to our Theorem~\ref{t.free}. It was shown recently in \cite{SapoMu} that transient amenable transitive graphs always have at most one vacant infinite component. Using a coupling with the Gaussian Free Field, it is proved in \cite{DPR} that certain graphs, including all transient Cayley graphs of polynomial growth, have a supercritical phase for vacant percolation. To our knowledge, it is not known for general transient transitive graphs that there is a supercritical phase, nor that nonamenable graphs have a non-uniqueness phase. 

In a {\bf social network} model of Benjamini, studied on general infinite graphs in \cite{socialfrogs}, each vertex starts with an independent $\POI(\lambda)$ number of random walkers, who walk for time $t$ before they die, and any pair of walkers get acquainted to each other if they are at the same vertex at any moment of time. It is proved in \cite[Theorem 3]{socialfrogs} that on any regular nonamenable graph, for any $\lambda>0$, there is some $t_c(\lambda)<\infty$ such that an infinite cluster of friends emerges. Their proof method is a much simpler version of our exploration process for Theorem~\ref{t.worms}, in which they do not need to use the size-biasing effect.

Another natural Poisson soup of infinite objects is the \textbf{Poisson cylinder model}, defined in \cite{TykessonWindisch2012} as a Poisson point process on the space of lines in $\R^d$, where a multiplicative factor of the intensity measure determines the density of the lines and where every line in the process is taken as the axis of a bi-infinite cylinder of radius $1$. Phase transitions for vacant percolation were proved here. Connectedness of the occupied set for all positive intensities was proved in \cite{BromanTykesson2016}. One can also generalize the model to hyperbolic spaces $\mathbb{H}^d$ by changing the lines of $\R^d$ to geodesic, and ask if the permanent connectedness is a result of $\R^d$ being a ``small'' space. Indeed, it was proved in \cite{BromanTykesson2015} that in any $\mathbb{H}^d$, $d \geq 2$, there exists some critical intensity below which there are infinitely many infinite clusters, but above which there is a unique one. Let us note that the proof of the existence of a non-unique phase again uses a branching process construction. 

The abundance of independence in a Poisson point process simplifies many arguments, but one may wonder how much we lose by considering such processes instead of more traditional $\{0,1\}$-valued invariant percolations. In a recent paper \cite{FrosstromGantertSteif2024}, the authors call a $\{ 0,1 \}$-valued process on some (discrete) set \textbf{Poisson representable} if it can be constructed as a union of subsets coming from a Poisson process on the collection of subsets, and prove, e.g., that all positively associated Markov chains on $\{0,1\}^\Z$ are Poisson representable, while the Ising model on $\Z^d$, $d\ge 2$, and on the complete graph, for certain temperatures, are not.

We close this subsection with the question that we find the most attractive one in the amenable setting. It is not addressed in~\cite{convexgrain}, being a borderline case there.

\begin{question}[Random rays in the plane]\label{q.Z2R2}
	Let the lattice animal measure $\nu_0$ be a straight line segment of random length, started at $o$, in a uniform random direction. On $\Z^2$, this means one of the four lattice directions with probability $1/4$ each, but one could also look at a Poisson point process on $\R^2$ as starting points, directions given by $\mathsf{Unif}[0,2\pi)$ variables.  Does $\E_{\nu}|H|^2=\infty$ suffice for $\lambda_c(\nu)=0$?
\end{question}

A growth process for physical networks built on random rays in $\Z^2$ was defined in \cite{physnetw}, as a member of a large family of models. It was argued non-rigorously and by simulations that this model has a mean-field degree exponent, vaguely suggesting that also for Question~\ref{q.Z2R2} we could observe the ``mean-field'' behaviour, meaning that the finiteness of the second moment would decide about the phase transition.

\subsection{Structure of the paper}
\label{subsection:structure_of_the_paper}

The rest of the paper is organized as follows. 

In Section~\ref{section:setup}, we collect all the necessary definitions and notions needed either for the precise statements or the proofs of our main theorems. We define amenability, unimodularity, and the free product of transitive graphs in Subsection~\ref{ss.graphs}. We provide some basic definitions and lemmas on random walk capacity in Subsection~\ref{ss.rw}. The basics of the Poisson zoo, after the rushed Introduction, are given carefully in Subsection~\ref{subsection:poisson_zoo_on_transitive_graphs}. The proof strategies for Theorem~\ref{t.free} (general animals on free products) and Theorem~\ref{t.worms} (worms on general nonamenable unimodular graphs) are explained in Subsections~\ref{subsection:poisson_zoo_on_free_products} and~ \ref{subsection:random_length_worms_model_introduction}, respectively. The basic properties of Poisson point processes that will be used in the arguments are summarized in Subsection~\ref{subsection:point_processes_of_rooted_animals}.

In Section \ref{section:bounds_on_fat} we will study the first and second moments of the ``fat'', which is the newly found occupied set in each step of our exploration process, in either of the two settings. In Section~\ref{section:exploration_processes} we analyze the growth of our exploration processes, proving the main two theorems. Finally, in Section~\ref{s.immuniq}, we prove Proposition~\ref{p.immuniq} on immediate uniqueness for a zoo on $\tree_\degree \times \Z^5$.

\subsection{Acknowledgments}

We are grateful to Tom Hutchcroft for inspiring Proposition~\ref{p.immuniq}, and to Bal\'azs R\'ath for several useful conversations. Three people so far have independently suggested that the Poisson zoo should be renamed the Aquarium; we do not thank them.
\smiley{}


\section{Preliminaries}
\label{section:setup}

Most of the definitions and theorems introduced here can be found in the union of \cite{LyonsPeres2016, PGG, Woess2000}. 

Some general notation that we will use are $\N := \{ 0, 1, \ldots  \}$ and $\N_{> 0} := \{ 1, 2, \ldots \}$ and their obvious modifications for $\R$. If we are given two real numbers $a, b \in \R$, then $a \wedge b := \min \{ a, b \}$ and $a \vee b := \max \{ a, b \}$.

\subsection{Graph theory definitions}\label{ss.graphs}

We will work on undirected, locally finite, connected, infinite graphs $\graph=(V,E)$, sometimes with a distinguished vertex $\origin \in V(\graph)$, referred to as the origin. We will use the shorthand $x \edge y$ to denote that the vertices $x, y \in V(\graph)$ are neighbours: $\{ x,y \} \in E(\graph)$.

For finite subsets we will use the notation $A \fsubset V(\graph)$. If $A \subseteq V(\graph)$, then the interior and exterior vertex boundaries of $A$ are 
\begin{equation}
	\label{def:vertex_boundaries}
	\intboundary A := \{ x \in A \, : \, \exists \, y \in A^c, \, x \edge y \}, \quad \text{ and } \quad
	\extboundary A := \{ x \in A^c \, : \, \exists \, y \in A, \, x \edge y \}.
\end{equation}
By the discrete closure of the set $A$ we mean $\overline{A} := A \cup \extboundary A$.

Distances with respect to the usual graph metric will be denoted by $\dist_{\graph}(\cdot, \cdot)$. Balls are denoted by 
\begin{equation}
	\ball(x,R) := \{ y \in V(\graph) \, : \, \dist_{\graph}(x,y) \leq R \}.
\end{equation}

An automorphism of $\graph$ is a bijection $V(\graph)$ to itself that preserves adjacency. The group of all automorphisms of $\graph$ is denoted by $\Aut(\graph)$. 
The graph is called transitive if $\Aut(\graph)$ acts transitively on $V(\graph)$. In such a case, the degrees of the vertices are constant, denoted by $\degree := \degree(\graph)$. Importantly to our work, edge transitivity does not follow; see, e.g., the first example on Figure~\ref{f.free}. The most important source of transitive graphs are Cayley graphs $\graph=\Cay(\grp,S)$ with some symmetric generating set $S = S^{-1} \subset \grp$, where $E(\graph):=\{ (x,x s) : x \in \grp, s\in S \}$. Here, $\grp$ itself acts transitively by $\gamma\cdot x := \gamma x$.

Recall the definition of the exterior vertex boundary from \eqref{def:vertex_boundaries}. By the Cheeger constant of the (infinite) graph $\graph$ we mean
\begin{equation}
	\label{def:Cheeger_constant}
	\Cheeger(\graph) :=
	\inf \left\{
	\frac{| \extboundary A| }{|A|} \, : \,
	A \subset V(\graph), \, 
	0 < |A| < \infty	
	\right\}.
\end{equation}
Intuitively, $\Cheeger(\graph)$ is an indicator of ``bottlenecks'' in the graph. The same notion could be defined using a different notion of boundary (interior vertex boundary or edge boundary), but since our graphs will always be locally finite and transitive, the difference between these notions is within constant multipliers. And this is what will matter: we call a graph $\graph$ \textbf{amenable} if $\Cheeger(\graph) = 0$, and nonamenable if $\Cheeger(\graph) > 0$. 

The most notable example of amenable graphs are the Euclidean lattices $\Z^d$, $d \geq 1$.
On the other hand, it is easy to prove \cite[Exercise 6.1]{LyonsPeres2016} that for the $\degree$-regular tree $\tree_{\degree}$ we have $\Cheeger( \tree_{\degree} ) = \degree - 2 $, and hence it is nonamenable for $\degree > 2$.

The importance of the following mass conservation principle in percolation theory was noted in \cite{BLPS:GAFA}, following \cite{Haggstrom}. It will be crucial also for us.

\begin{definition}[Unimodularity]
	\label{def:unimodularity}
	Let $\grp$ be a transitive subgroup of automorphism of the graph $\graph$. If we have the \textbf{mass transport principle (MTP)}
	\begin{equation}
		\label{eq:MTP}
		\sum_{y \in V(\graph)} f(y,x) =
		\sum_{y \in V(\graph)} f(x,y)
	\end{equation}
	for any diagonally $\grp$-invariant $f  :  V(\graph) \times V(\graph) \longrightarrow [0, \infty]$ function and for any $x \in V(\graph)$, then we say that the action of $\grp$ is \textbf{unimodular}.	If the action of $\Aut(\graph)$ is unimodular on $\graph$, then we simply say that $\graph$ is unimodular.
\end{definition}

Every Cayley graph is unimodular, since, for every diagonally $\grp$-invariant function $f$, 
\[
	\sum_{y \in \grp} f(y,x)=\sum_{y\in \grp} f(xy^{-1}y,xy^{-1}x)=\sum_{z\in\grp} f(x,z),
\]
as $y\mapsto z:=xy^{-1}x$ is a bijection. Moreover, every amenable transitive graph is unimodular, as well (see, e.g., \cite[Proposition 8.13]{LyonsPeres2016}).

However, in the nonamenable case, there are non-unimodular transitive graphs. The simplest such graph is Trofimov's grandparent graph, obtained from a $\degree$-regular tree by distinguishing an end (an infinite ray starting at any fixed vertex), thinking of the neighbour of each vertex $x$ towards the end as the parent of $x$, then connecting every vertex to its (unique) grandparent; see \cite[Example 7.1 and Section 8.2]{LyonsPeres2016}. The automorphism group of this graph is just the subgroup of $\Aut(\tree_\degree)$ that fixes that distinguished end, which does not satisfy MTP: if $f(x,y)$ is 1 when $y$ is the parent of $x$, otherwise 0, then the outgoing mass is 1, while the incoming is $\degree$.

\medskip

Next, we give a detailed description of the free product of transitive graphs, following \cite[Chapter II. 9.C]{Woess2000}. Before the definition, let us note that we only deal with the free product of two graphs here, since that will be enough for our purposes. However, one can easily extend the definition for countably many terms (as it is done in \cite{Woess2000}).

\begin{definition}
	\label{def:free_product}
	Suppose that we are given two non-trivial connected transitive graphs $\graph_i = (V(\graph_i), E(\graph_i))$, $i = 1,2$, with disjoint vertex sets and some distinguished vertices $\origin_i$, $i = 1,2$. The \textbf{free product}
	\begin{equation}
		\label{eq:free_product}
		\graph := (V(\graph), E(\graph)) :=
		\graph_1 \star \graph_2
	\end{equation}
	is constructed in the following way.
	The vertices of $\graph$ are all the finite {words}:
	\begin{equation}
		\label{def:vertices_of_free_product}
		V(\graph) :=  \left\{
		x_1 \ldots x_n \, \left| \, 
		\parbox{19.5em}{
			$n \in \N; \, x_j \in V(\graph_1) \cup V(\graph_2) \setminus \{ \origin_1, \origin_2 \}; \\
			x_k \in V(\graph_i) \implies x_{k+1} \notin V(\graph_i), \, i=1,2$
		}		
		\right. \right\} 
		\cup
		\left\{ \origin \right\}
	\end{equation}
	where $\origin$ denotes the {empty word}, the distinguished vertex of $\graph$, obtained by identifying each $\origin_i$, $i = 1,2$ with $\origin$.
	If we are given words $x = x_1 \ldots x_k \in V(\graph)$ and $y = y_1 \ldots y_m \in V(\graph)$ with $x_k \in V(\graph_i)$ and $y_1 \notin V(\graph_i)$, then $xy$ stands for the concatenation as words. In addition, for all $x \in V(\graph)$ we define $x \origin = \origin x = x$.
	Furthermore, adjacency in $\graph$ is given as follows: for $i = 1,2$, if $\{ x,y \} \in E(\graph_i)$, then $\{ zx, zy \} \in E(\graph)$ for all $z = z_1 \ldots z_{\ell} \in V(\graph)$, with $z_{\ell} \notin V(\graph_i)$.
\end{definition}

In words, we take copies of $\graph_i$, $i=1,2$ and glue them together by identifying $\origin_1$ and $\origin_2$ into a single vertex $\origin$. 
Then, inductively, at each vertex $v$ of $\graph_i$ attach a copy of $\graph_j$, $j \neq i$, in such a way that $v$ is identified with $\origin_j$ from the the new copy of $\graph_j$.
So, from any vertex of the product one can either go in the direction of the component $\graph_1$ or $\graph_2$, and these two subgraphs only meet in the given vertex.
Consequently, it is also meaningful to introduce the $\graph_i$-neighbourhood of a vertex $x \in V(\graph)$ as
\begin{equation}
	\label{def:free_product_term_neighbourhood}
	\mathcal{N}_i(x) := \left\{
	y \in V(\graph) \, : \,
	x \edge y \text{ in } \graph_i 
	\right\}
\end{equation}
for $i = 1,2,$.
That is, $\mathcal{N}_i(x)$ is the subset of those neighbours of $x$ that can be reached using a $\graph_i$-edge.

One well-known example of such a free product is the $\degree$-regular tree $\tree_{\degree}$, which is the free product of a single edge and $\tree_{\degree-1}$; or, iterating this, the $\degree$-wise free product of a single edge. In particular, $\tree_4$ is the free product of $\Z$ with itself, as on the second picture of Figure~\ref{f.free}. 

The most relevant properties of a free product can be summarized as the following collection of observations.
As their proofs are obvious from the definition, they are omitted.

\begin{claim}
	\label{claim:properties_of_free_product_of_graphs}
	\begin{enumerate}
		\item[]
		\item The free product of two connected, locally finite and transitive graphs is again connected, locally finite and transitive.
		\item Every vertex of a free product is a cutpoint: removing it results in at least two disjoint infinite subgraphs.
		\item Given any cycle in a free product, all of its edges are either from the first or the second component. That is, the free product cannot introduce new cycles.
	\end{enumerate}
\end{claim}

\subsection{Random walks, spectral radius, capacity}\label{ss.rw}

Let us denote the space of all $V(\graph)$-valued infinite nearest neighbour trajectories indexed by nonnegative integers by $W := W(\graph)$. This countable space can be endowed with a natural $\sigma$-algebra $\mathcal{W}$, the one generated by the canonical coordinate maps.

\begin{definition}
	For $x \in V(\graph)$, let $P_x$ denote the law of the simple random walk $( X(n) )_{n \in \N}$ on $\graph$ which starts from the vertex $x$, i.e., $X(0) = x$. Let us also denote the corresponding expectation by $E_x$. The law $P_x$ can be considered as a probability measure on $(W, \mathcal{W})$.
\end{definition}


On any graph, simple random walk is a reversible Markov chain, with stationary measures proportional to the degrees. The chain is described by the transition probabilities
\begin{equation}
	\label{def:transition_probabilities_of_random_walks}
	p_n(x,y) := P_x(X(n) = y), \qquad x, y \in V, \, n \in \N.
\end{equation}

On a nonamenable graph, one expects that these transition probabilities decay fast with $n$. This is indeed the case, and to argue why, let us introduce the notion of spectral radius of a graph; note that it does not matter which starting vertex we choose, due to connectedness.

\begin{definition}
	\label{def:spectral_radius}
	The \textbf{spectral radius} of the graph $\graph$ is the quantity
	\begin{equation}
		\label{eq:spectral_radius}
		\spectral(\graph) := 
		\limsup_{n \to \infty}
		p_n(\origin, \origin)^{1 \slash n}.
	\end{equation}
\end{definition}

\noindent
Obviously, $\spectral(\graph) < 1$ means that $p_n(\cdot, \cdot)$ decreases exponentially fast. The celebrated result of Kesten, Cheeger, Dodziuk and Mohar (see \cite[Theorem 10.3]{Woess2000} or \cite[Theorem 6.7]{LyonsPeres2016} or \cite[Theorem 7.3]{PGG}) tells us that this property characterizes nonamenability, defined at~(\ref{def:Cheeger_constant}):
\begin{equation}
	\label{eq:relation_of_spectral_radius_and_Cheeger_constant}
	\spectral(\graph) < 1 
	\quad \Longleftrightarrow \quad
	\Cheeger(\graph) > 0.
\end{equation}

Some random stopping times we will use extensively throughout the analysis of the random length worms model are as follows. For any $w \in W$ and $K \fsubset V(\graph)$, let
\begin{align}
	\label{nota:entrance_time}
	T_K(w) & := \inf \{ \, n \geq 0 \, : \, w(n) \in K\, \}, \quad \text{first entrance time}, \\
	\label{nota:hitting_time}
	T_K^+(w) & := \inf \{ n \geq 1 \, : \, w(n) \in K \}, \quad \text{first hitting time},
\end{align}
where we define $\inf \emptyset  = + \infty$. 
To simplify notation, the case of sets that consist of only one vertex will be denoted by $T_x(w) := T_{\{ x \}}(w)$ and $T_x^+(w) := T_{\{ x \}}^+(w)$.
With the latter in hand, let us recall that the graph $\graph$ is called \textbf{recurrent} if 
\begin{equation}
	\label{def:recurrence_of_graphs}
	P_{\origin} \left( T_{\origin}^+(X) < + \infty \right) = 1,
\end{equation}
and \textbf{transient} otherwise, which does not depend on the chosen vertex for a connected graph, and hence is a well-defined property of a graph.

It is an easy fact that a nonamenable graph is always transient: \eqref{eq:relation_of_spectral_radius_and_Cheeger_constant} implies that 
\[
	E_o \big|\{n \geq 0 : X(n)=o\}\big| = \sum_{n\ge 0} p_n(o,o) < \infty,
\]
hence the number of returns to $o$ is almost surely finite, hence \eqref{def:recurrence_of_graphs} fails. The next theorem, a more general version of which can be found in \cite[Propositon 9.3]{PGG} or \cite[Proposition 6.9]{LyonsPeres2016}, says much more.

\begin{theorem}
	\label{thm:positive_constant_speed_of_SRW_on_nonamenable_transitive_graphs}
	Given a nonamenable and transitive graph $\graph$, there exists a constant $c(\graph) > 0$ such that 
	\begin{equation}
		\label{eq:positive_constant_speed_of_SRW_on_nonamenable_transitive_graphs}
		P_{\origin} \left( \liminf_{t \to \infty} \frac{\dist_{\graph}(\origin, X(t))}{t} \geq c(\graph)  \right) = 1\,.
	\end{equation} 
\end{theorem}

\noindent
In words, for nonamenable (and transitive) graphs, the distance between the start and the end points of a random walk running for some long enough time is almost surely comparable to the number of steps it took. Obviously, this distance provides a lower bound on the size of the trace of the walker, hence one can derive the following corollary.

\begin{corollary}
	\label{coro:number_of_steps_and_size_of_trace_of_SRW_are_comparable_on_nonamenable_graphs}
	Given a nonamenable and transitive graph $\graph$, for any $\veps > 0$, there exists $c(\veps) > 0$ such that for any $t > 0$ we have
	\begin{equation}
		\label{eq:number_of_steps_and_size_of_trace_of_SRW_are_comparable_on_nonamenable_graphs}	
		P_{\origin} \Big(
		\big| X[0, t] \big| \geq c(\veps) \, t
		\Big) > 1 - \veps.
	\end{equation}
\end{corollary}

Given $K \fsubset V(\graph)$ and $x \in K$, let us define the \textbf{equilibrium measure} $e_K(x)$ of $x$ with respect to $K$ by
\begin{equation}
	\label{def:equilibrium_measure}
	e_K(x) := P_x \left( T_K^+ = \infty \right),
\end{equation}
and the \textbf{capacity} of $K$ by
\begin{equation}
	\label{def:capacity}
	\capacity(K):=\sum_{x \in K} e_K(x).
\end{equation}
Obviously, the equilibrium measure is concentrated on the interior vertex boundary of a set. Due to this and the reversibility of the simple random walk, heuristically speaking, the capacity of a set measures the visibility of the set from the point of view of a random walk coming from infinity: the larger the capacity, the more visible the set is.
As an example, it follows immediately from the definitions that the capacity of a one-element set is
\begin{equation}
	\label{eq:capacity_of_one_vertex}
	\capacity(\{ x \}) = 
	P_{x}( \hitTime_x = \infty ) = 1 \slash \Green(x, x)
\end{equation}
for every $x \in V(\graph)$, where $\Green(x,y) := \sum_{n = 0}^{\infty} p_n(x,y)$ is \textbf{Green's function}. 

Although there are many properties of  capacity that can be used to examine the behaviour of random walks, in our context we will be interested only in one of them. Namely, on nonamenable graphs, the capacity of a set is comparable to its size, which follows, for example, from \cite[Eq.~(4.4)]{Teixeira2009} or \cite[Lemma 2.1]{BenjaminiNachmiasPeres2011}.

\begin{lemma}[Capacity is linear]
	\label{lemma:capacity_bound}
	For any $K \fsubset V$ we have 
	\begin{equation}
		\label{eq:capacity_bound}
		(1 - \spectral(\graph)) \cdot |K| \leq
		\capacity(K) \leq
		|K|,
	\end{equation}
	where recall that $\spectral(\graph)$ denotes the spectral radius of the graph. In particular, by \eqref{eq:relation_of_spectral_radius_and_Cheeger_constant}, on a nonamenable graph $\graph$ the capacity and the cardinality of a finite vertex set are comparable.
\end{lemma}

In the proof of Theorem~\ref{t.worms}, we will also use the following restricted version of capacity.

\begin{corollary}
	\label{coro:lower_bound_on_partial_square_term_capacity}
	Let $A \fsubset V(\graph)$ be finite and connected, $B \subseteq \extboundary A$. Then we have
	\begin{equation}
		\label{eq:lower_bound_on_partial_square_term_capacity}
		\sum_{y \in B} P_y^2 \left( T_{\overline{A}}^+ = +\infty \right) \geq
		(1 - \spectral(\graph))^2 \cdot \left| \overline{A} \right| - ( | \extboundary A| - |B|).
	\end{equation}
\end{corollary}

\begin{proof}
	Since $B \subseteq \extboundary A$, the sum can be ``decomposed'' as the subtraction
	\begin{equation*}
		\sum_{y \in B} P_y^2 \left( T_{\overline{A}}^+ = + \infty \right) =
		\sum_{y \in \extboundary A} P_y^2 \left( T_{\overline{A}}^+ = + \infty \right) - 
		\sum_{y \in \extboundary A \setminus B} P_y^2 \left( T_{\overline{A}}^+ = + \infty \right).
	\end{equation*}
	For the first term, we have
	\begin{equation*}
		\sum_{y \in \extboundary A} P_y^2 \left( T_{\overline{A}}^+ = + \infty \right) \geq 
		\left( 1 - \spectral(\graph) \right)^2 \cdot \left| \overline{A} \right|,
	\end{equation*}
	since, by the Cauchy-Schwarz inequality, we can write and then rearrange the following:
	\begin{align*}
		& |\overline{A}| \cdot 
		\sum_{y \in \extboundary A} P_y^2 \left( T_{ \overline{A}}^+ = + \infty \right) = 
		\left( \sum_{x \in \overline{A}} 1^2 \right) \cdot \left( \sum_{y \in \overline{A}} P_y^2 \left( T_{ \overline{A}}^+ = + \infty \right) \right) \geq \\ & \qquad \geq 
		\left( \sum_{x \in \overline{A}} P_y \left( T_{\overline{A}}^+ = +\infty \right) \right)^2 = 
		\capacity \left( \overline{A} \right)^2 \overset{\eqref{lemma:capacity_bound}}{\geq} 
		(1- \spectral(\graph))^2 \cdot |\overline{A}|^2.
	\end{align*}
	Meanwhile, for the second term, one can use the trivial upper bound
	\begin{align*}
		\sum_{y \in \extboundary A \setminus B} P_y^2 \left( T_{\overline{A}}^+ = + \infty \right) \leq 
		\left| \extboundary A \right| - \left| B \right|.
	\end{align*}
	Subtracting the two bounds finishes the proof.
\end{proof}

\subsection{The Poisson zoo on transitive graphs}
\label{subsection:poisson_zoo_on_transitive_graphs}

We now introduce the Poisson zoo carefully.


\begin{definition}
	\label{def:lattice_animal}
	Given any vertex $x \in V(\graph)$, let $\latticeanimals_x$ be the set of finite, connected, nonempty subsets $H$ of $V$ that contain $x$, called the lattice animals rooted at $x$.
\end{definition}

Let $\grp \subgrp \Aut(\graph)$ be a subgroup of the automorphism group of the graph, which still acts transitively on $V$.
Obviously, the action of an automorphism $\varphi \in \grp$ can be extended to any set of vertices $H \subseteq V(\graph)$ as
$\varphi(H) := \left\{ \varphi(h) \, : \, h \in H \right\}$. Moreover, if $x,y \in V(\graph)$ and $\varphi \in \grp$ is such that $\varphi(x) = y$, then $\varphi$ defines a bijection between $\latticeanimals_x$ and $\latticeanimals_y$.

Given a vertex $x \in V(\graph)$, the stabilizer of $x$ under the action of $\grp$ is the subgroup
\begin{equation}
	\label{def:stabilizer_subgroup}
	\grp_{x} := \left\{ \varphi \in \grp \, : \, \varphi (x) = x \right\} \subgrp \grp.
\end{equation}

\begin{definition}
	\label{def:grp-rotation-invariant_prob_measure}
	We say that a probability measure $\nu_{\origin}$ on $\latticeanimals_{\origin}$ is ($\grp$-)rotation invariant if, for any $\varphi \in \grp_{\origin}$ and any $H \in \latticeanimals_{\origin}$, we have
	$\nu_{\origin}( \varphi(H) ) = \nu_{\origin}(H)$.
\end{definition}

A class of examples is when the measure of an animal depends only on its cardinality.

Since $\grp$ acts transitively on $\graph$, for all vertices $x \in V(\graph)$ 
\begin{equation}
	\label{def:the_distinguished_automorphism_taking_origin_to_x}
	\text{there exists at least one \linebreak automorphism $\varphi_x \in \grp$ such that $\varphi_x (\origin) = x$.}
\end{equation}
As a consequence, if we are given a $\grp$-rotation invariant measure $\nu_{\origin}$ on $\latticeanimals_{\origin}$, we can push it forward with $\varphi_x$ to obtain a $\grp$-rotation invariant measure on $\latticeanimals_x$ for any $x \in V$:
\begin{equation}
	\label{def:translated_prob_measure}
	\nu_x := \nu_{\origin} \circ \varphi_{x}^{-1}.
\end{equation}
Moreover, if $\gamma \in \grp$ is such that $\gamma (\origin) = x$, then for any $H \in \latticeanimals_{\origin}$ we have $\gamma(H) \in \latticeanimals_{x}$ and $\nu_x( \gamma(H) ) = \nu_{\origin} ( \varphi_{x}^{-1} \gamma(H) ) = \nu_{\origin} (H)$, since $\varphi_x^{-1} \gamma \in \grp_{\origin}$ and $\nu_o$ is rotation-invariant. This also means that all automorphisms $\varphi_x$ in \eqref{def:translated_prob_measure} give the same measure $\nu_x$.

\begin{definition}
	\label{def:poisson_zoo}
	Given the $\grp$-rotation invariant measure $\nu := \nu_{\origin}$ and some intensity parameter $\lambda \in \R_{>0}$, let us consider independent Poisson random variables $N_{x,H}^{\lambda} \sim \POI( \lambda \cdot \nu_x(H) )$ indexed by $x \in V(\graph)$ and $H \in \latticeanimals_x$.
	We say that $N_{x,H}^{\lambda}$ is the number of copies of the animal $H$ (rooted at $x$).
	We call the collection of random variables
	\begin{equation}
		\label{eq:poisson_zoo}
		\mathcal{N}_{\nu}^\lambda := 
		\left\{ 
		N^\lambda_{x,H} \, : \, x \in V(\graph), H \in \latticeanimals_x 
		\right\}
	\end{equation}
	the \textbf{Poisson zoo} at level $\lambda$ (with measure $\nu$) and the random set
	\begin{equation}
		\label{def:trace_of_poisson_zoo}
		\mathcal{S}_{\nu}^\lambda := 
		\bigcup_{x \in V} \bigcup_{H \in \latticeanimals_x} \bigcup_{i = 1}^{ N_{x,H}^{\lambda} } \, H
	\end{equation}
	the \textbf{trace} of the Poisson zoo at level $\lambda$.
\end{definition}

This is the same definition as (\ref{e.zoo}) in the Introduction, via $N^\lambda_x := \sum_{H \in \latticeanimals_x} N_{x,H}^{\lambda}$. Note also that, since the set of possible animals is countable, there is a positive chance for the occurrence of multiple instances of the same animal with the same root. Although this seems redundant and might suggest to the reader that another distribution should be used to define the model, in the small intensity regime that we are mostly focusing on, i.e., when $\lambda$ is close to $0$, the probability of such events is negligible. Meanwhile, as we will see later in Section~\ref{section:setup}, having a Poisson distribution is a very useful tool in the analysis of the model.

We say that a subset $S \subseteq V(\graph)$ of sites percolates if $S$ has at least one infinite connected component in the induced subgraph. From the rotational invariance of $\nu$ and the factor of i.i.d.~construction \eqref{def:trace_of_poisson_zoo}, it follows immediately that
\begin{equation}
	\label{obs:poisson_zoo_is_ergodic}
	\text{ the law of $\mathcal{S}_{\nu}^\lambda$ is invariant and ergodic under the action of $\grp$.}
\end{equation}
Since the event of existence of an infinite cluster is invariant (under the action of the group $\grp$), we obtain that the value of $\prob \left( \mathcal{S}_{\nu}^\lambda \text{ percolates} \right)$ can be either $0$ or $1$.

For parameters $0<\lambda<\lambda'<\infty$, there are monotone couplings of $X\sim \POI(\lambda)$ and $X'\sim  \POI(\lambda')$ variables such that $X\leq X'$ almost surely; one comes from the obvious monotone coupling of $\mathsf{Exponential}$ variables, another from coupling $\mathsf{Binomial}$ variables and taking the limit. Thus there is also a monotone coupling of the Poisson zoos $\mathcal{N}_{\nu}^{\lambda}$ and $\mathcal{N}_{\nu}^{\lambda'}$ (with the same $\nu$) such that
\begin{equation}
	\label{obs:monotone_coupling_off_poisson_zoos}
	\mathcal{S}_{\nu}^{\lambda} \subseteq \mathcal{S}_{\nu}^{\lambda'} \quad \text{if} \quad
	0 \leq \lambda \leq \lambda'.
\end{equation}
Hence, there is a critical intensity $\lambda_c=\lambda_c(\graph,\nu)$, as defined in~(\ref{e.lambdac}), such that for all $\lambda < \lambda_c$ the trace $\mathcal{S}_{\nu}^{\lambda}$ almost surely does not percolate, while for all $\lambda>\lambda_c$ it almost surely does.

The basic question for any such percolation model concerns the non-triviality of phase transition: given a $\graph$ (and $\grp$), for which choices of $\nu$ do we have $\lambda_c \in (0,\infty)$?

Let us first argue that if $p_c := p_c^{\text{site}}(\graph) < 1$ for i.i.d.\ Bernoulli site percolation, then $\lambda_c<+\infty$ holds for any choice of a rotation invariant measure $\nu$. 
As mentioned in the Introduction, this is the case for all Cayley graphs that are not finite extensions of $\Z$ \cite{Duke} and for all uniformly transient graphs, even without transitivity \cite{EST}. 
Indeed, for any $x \in V(\graph)$, any $H \in \latticeanimals_x$ contains the vertex $x$ of $\graph$.  
Therefore, the trace $\mathcal{S}_{\nu}^\lambda$ of the Poisson zoo at level $\lambda$ stochastically dominates Bernoulli site percolation with density $p=1-e^{-\lambda}$,  which implies $\lambda_c \leq -\ln\left( 1-p_c \right) < +\infty$.

Therefore, the real question is whether the Poisson zoo model has a non-trivial subcritical phase (i.e., $\lambda_c>0$) or is it supercritical for all $\lambda >0$ (i.e., $\lambda_c=0$).
Let us now recall two general results from \cite{RathRokob2022} that give sufficient conditions for $\lambda_c=0$ and $\lambda_c>0$, respectively. 

\begin{definition}
	\label{def:moments_of_nu}
	Given $k \in \N$, let us denote the $k$'th moment of the cardinality of a random subset with law $\nu$ by
	\begin{equation}
		\label{eq:moments_of_nu}
		m_k := m_k(\nu):= \sum_{H \in \latticeanimals_{\origin}} |H|^k \cdot \nu(H).
	\end{equation}
\end{definition}

\begin{lemma}[Lemmas 1.5 and 1.6 in \cite{RathRokob2022}]
	\label{lemma:two_silly_lemmas}
	Assume that the action of $\grp \leq \Aut(\graph)$ on the infinite graph $\graph$ is transitive and unimodular (see Definition~\ref{def:unimodularity}).
	\begin{enumerate}[label=(\arabic*)]
		\item 
		\label{lemma:two_silly_lemmas:infinite_m1_implies_percolation}
		If $m_1 = \infty$, then for any $\lambda >0$ we have $\mathcal{S}_{\nu}^\lambda = V(\graph)$. This implies, in particular, that $\mathcal{S}_{\nu}^\lambda$ percolates for any $\lambda >0$.
		\item
		\label{lemma:two_silly_lemmas:finite_m2_implies_no_percolation}
		If $m_2 < \infty$, then, for any $\lambda \in (0, 1/ (\degree + 1) \cdot m_2 ))$, the set $\mathcal{S}_{\nu}^\lambda$ does not percolate, where $\degree$ denotes the valency of the graph.
	\end{enumerate}
\end{lemma}

These were proved in \cite{RathRokob2022} for the Euclidean lattice $\Z^d$, $d \geq 2$, but they easily extend to our general setting, assuming unimodularity. We omit the proofs here, especially that Lemma~\ref{lemma:size_biasing_for_Poisson_zoo} below is a version of these lemmas, explaining how unimodularity and size-biasing lead to the relevance of the first and second moments. Our results do require unimodularity in a crucial way --- see Remark~\ref{remark:unimodularity_is_needed}.

\subsection{Proof strategy for the Poisson zoo on free products }
\label{subsection:poisson_zoo_on_free_products}

A key example where we can show  the reverse of part~\ref{lemma:two_silly_lemmas:finite_m2_implies_no_percolation} of Lemma~\ref{lemma:two_silly_lemmas}, for completely general animals, are $\degree$-regular trees $\tree_{\degree}$ with $\degree \geq 3$. The proof, with some non-negligible extra work, applies more generally, to unimodular free products (see Definition~\ref{def:free_product}), giving us Theorem~\ref{t.free}. In this subsection, we sketch our proof strategy for this theorem.

The proof in \cite{RathRokob2022} for Lemma~\ref{lemma:two_silly_lemmas}~\ref{lemma:two_silly_lemmas:finite_m2_implies_no_percolation} relies on a domination by a subcritical branching process \textit{from above}. In the present paper, in order to prove percolation for every $\lambda>0$, we will embed a supercritical branching process \textit{inside} the percolation cluster of a given vertex. As mentioned above, supercriticality at every $\lambda$ should (and will) be provided by $m_2(\nu)=\infty$, but how do we get a branching structure? The basic observation is that, since in a free product every vertex is a cutpoint that separates two directions, cf.\ Claim \ref{claim:properties_of_free_product_of_graphs}, one can go either in the direction of the first component of the product or the second, and the animals living on these disjoint subgraphs must be also disjoint. Thus, by the defining properties of the underlying Poisson point process, we have that these animals must be independent of each other.
Consequently, we can embed a Galton-Watson branching process into the vertices of the graph by defining the offspring of an individual to be certain vertices associated to the exterior boundary of the set consisting of the union of the traces of the animals visiting only those neighbours of the individual that are accessed from the not yet considered direction.
However, since the automorphism subgroup under which the Poisson zoo is invariant acts transitively only on the vertices, not on the edges, it could happen that the given rotation-invariant measure $\nu$ could not produce enough animals that are visiting from that direction, and thus would not give rise to a supercritical reproduction mean; see Figure~\ref{f.free2} for this ``parity issue''. For any given occupied set explored so far, we can ensure large further expected growth only through its exterior boundary vertices in one of the directions; see Corollary \ref{coro:lower_bound_on_expected_size_of_fat_of_animals_through_one_vertex_on_free_product}.

\begin{figure}[tb]
	\centerline{
			\includegraphics[width=2.5 in]{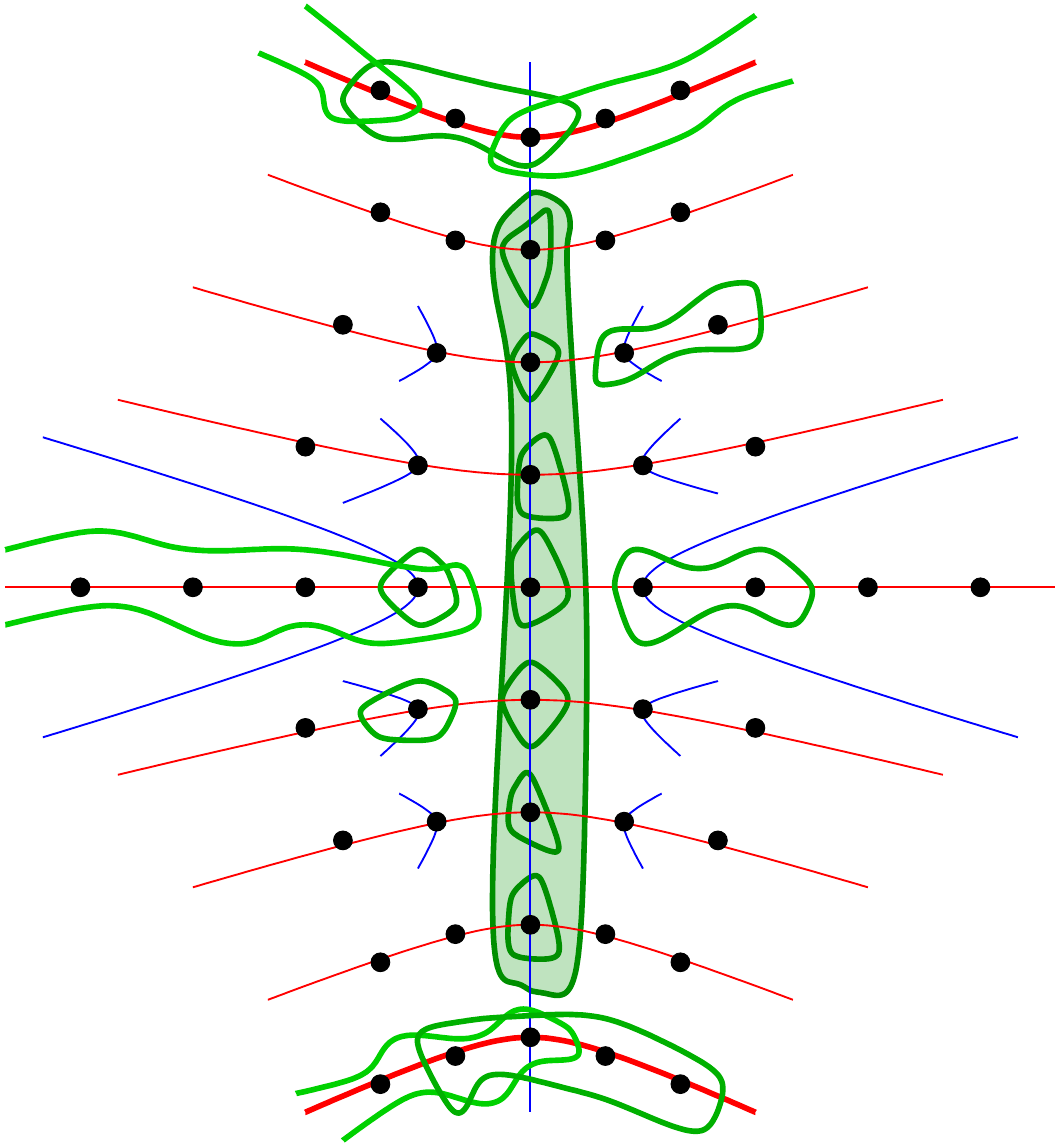}
			}
	\caption{\textbf{A parity issue.} On the free product $\Z\star\Z$, generated by the letters $a$ and $b$, take a $\nu$ with $m_2(\nu)=\infty$ that is supported only on words in one of the generators, say $a$ (the ``horizontal'' red generator in the picture). It may happen that we have already explored a large connected set, shaded green in the picture, which has of course a large total exterior boundary, but only a small one in the $b$ direction. But only these are the exterior boundary vertices from where the $a$ direction is completely unexplored, where we can have a lot of new animals covering that vertex, exploiting the size-biasing effect and the $m_2(\nu)=\infty$ condition.}
	\label{f.free2}
\end{figure}

Fortunately, this asymmetry caused parity issue can be remedied using a standard method of percolation theory: sprinkling.
Vaguely, if we arrived at some vertex in the exterior boundary of the set visited by animals explored so far from a direction that is not beneficial for further progress, then, instead of considering this vertex as an offspring, we take a neighbour in the right direction. On the one hand, this extra step causes a hole (just a single vertex) in the embedded branching process, but, on the other hand, it will open up territories towards the good direction. The good news is that, if we use only a positive fraction of animals for constructing the branching process, then the hole can be patched using the animals ``growing out of the hole'' from the other positive fraction, independently from everything else. Of course, this has a probability cost, worsening the reproduction mean by a positive factor, but, given the infinite mean for the offspring, we will still remain in the supercritical phase (cf.\ Claim \ref{claim:lower_bound_on_reproduction_rate_branching_process}), proving Theorem~\ref{t.free}.

\medskip

Wishing to progress further, beyond free products, there are serious issues. The tree-like inner structure of the free product guarantees that we can separate the animals into not just independent bundles, but non-intersecting ones (which then can be used as the offsprings). One cannot aim for such a disjointness of general animals in a general nonamenable graph, but may try to experiment with some relaxation of disjointness, where there is still a control over the intersections of the bundles. However, given the varieties in animal shapes and in the geometries of Cayley graphs, this seems like a really hard problem. One case where this can be done is when the animals are random walk trajectories, taking us to our next result.

%
%
%

\subsection{Restatement and proof strategy for the worms model}
\label{subsection:random_length_worms_model_introduction}

Random walk trajectories are very natural, locally growing random sets, worth studying for a variety of reasons. On nonamenable graphs, they grow linearly (Theorem~\ref{thm:positive_constant_speed_of_SRW_on_nonamenable_transitive_graphs} above), hence they may be expected to produce large clusters, possibly achieving $\lambda_c(\nu)=0$ in great generality.


\begin{definition}
	\label{def:random_length_worms_model}
	Let us consider a probability mass function $\eta: \N_{>0} \to [0,1] $ and an $\N_{>0}$-valued random variable $\mathcal{L}$ satisfying $\prob(\mathcal{L}=\ell)=\eta(\ell), \, \ell \in \N_{>0}$. 
	We call the distribution of $\mathcal{L}$ the length distribution of the worms. Let us also consider a nearest neighbour simple random walk $(X(t))_{t \in \N}$ on the vertices of $\graph$ starting from $X(0)=\origin$, independent of $\mathcal{L}$.
	Let us define the probability measure $\wormnu{\eta}$ on $\mathcal{H}_{\origin}$ by 
	\begin{equation}
		\label{def:worm_intensity_measure}
		\wormnu{\eta}(H) := P_{\origin}( \{ X(0), X(1), \ldots, X(\mathcal{L} - 1) \} = H )
	\end{equation}
	for each $H \in \mathcal{H}_{\origin}$.
	Given any $\lambda >0$, the Poisson zoo $\mathcal{N}_{\eta}^{\lambda} := \mathcal{N}_{\wormnu{\eta}}^{\lambda}$ built upon this measure via Definition \ref{def:poisson_zoo} is called the random length worms model at level $\lambda$, and the corresponding random set $\mathcal{S}_{\eta}^\lambda := \mathcal{S}_{\wormnu{\eta}}^\lambda$ is the random length worms set at level $\lambda$.
\end{definition}

Note that the law of a randomly stopped random walk on a graph $\graph$ is $\Aut(\graph)$-rotation invariant in the sense of Definition~\ref{def:grp-rotation-invariant_prob_measure}, making this zoo model meaningful.
Moreover, the moment conditions of Lemma~\ref{lemma:two_silly_lemmas} are easy to check:

\begin{lemma}[Time versus trace]\label{l.nuell}
	On any nonamenable transitive graph, the lattice animal measure $\wormnu{\eta}$ corresponding to a worm length measure $\eta$, from Definition~\ref{def:random_length_worms_model}, satisfies $m_1(\wormnu{\eta})<\infty$ and $m_2(\wormnu{\eta})=\infty$ if and only if $\E[\mathcal{L}]<\infty$ and $\E[\mathcal{L}^2]=\infty$.
\end{lemma}

\begin{proof}
	Obviously, $\big|X[0,\cL)\big| \leq \cL$, hence an infinite first or second moment for $\wormnu{\eta}$ implies the same for $\cL$. 
	
	For the other directions, we can use Corollary~\ref{coro:number_of_steps_and_size_of_trace_of_SRW_are_comparable_on_nonamenable_graphs} and the random walk steps being independent of the length $\mathcal{L}$, implying, for $j=1,2$ and any $\veps\in (0,1)$, the following:
	\begin{equation}
		\begin{aligned}
			m_j(\wormnu{\eta}) &= \sum_{\ell\ge 1} E_o \Big[ \big| X[0, \mathcal{L}) \big|^j \,\Big|\, \mathcal{L}=\ell \Big]  \, \eta(\ell)\\
			& \overset{\eqref{eq:number_of_steps_and_size_of_trace_of_SRW_are_comparable_on_nonamenable_graphs}}{\geq}  (1-\veps) \, c(\veps)^j  \sum_{\ell \geq 1} \ell^j \, \eta(\ell)  \ge c_j \, \E \left[ \mathcal{L}^j \right].
		\end{aligned}
	\end{equation}
	This finishes the proof.
\end{proof}

As a corollary, we immediately get that the next theorem is equivalent to Theorem~\ref{t.worms}.

\begin{theorem}
	\label{thm:random_length_worms_model_on_nonamenable_graphs}
	Let $\graph$ be an infinite graph, such that $\grp \subgrp \Aut(\graph)$ acts transitively and unimodularly on $\graph$.
	If $\mathcal{L} \sim \eta$ with $\E \left[ \mathcal{L}^2 \right] = +\infty$, then for any $\lambda>0$ the random length worms model $\mathcal{N}_{\eta}^\lambda$ is supercritical:
	$\prob(\, \mathcal{S}_{\eta}^\lambda \text{ percolates}\, )=1$.
\end{theorem}

As indicated above, the proof of this result uses the following process exploring the cluster of the origin.
Initially, we take the union of all those worms that ever visit the origin, giving the set $E_0$.
Assuming that $E_0$ is not empty and large enough, we can ``fatten'' this set in the next step using trajectories that 
$(1)$ have length at most $R$; and 
$(2)$ bounce back from the closure $E_0 \cup \extboundary E_0$, meaning that the worm visits the exterior boundary of $E_0$ once, but, apart from this visit, it avoids the closure entirely.
(Note here that, since we consider site percolation, it is not a problem that the edges between the set and its exterior boundary are not visited.)

If we denote the union of $E_0$ and the traces of the bounced back worms by $E_1$, then of course it can happen that $\extboundary E_1 \cap \extboundary E_0$ is not empty, meaning that no worm bounced back from the vertices of this intersection.
Consequently, to obtain independence between the iterations, in the next step we have to change property $(2)$ to the one: the worm of length at most $R$ bounces back from $\extboundary E_1 \setminus \extboundary E_0$ and never visits the closure $E_1 \cup \extboundary E_1$.
This is what we can call a \textbf{fattening} of $E_1$ through $\extboundary E_1 \setminus \extboundary E_0$ (with worms of length at most $R$), and what the exploration does as long as it can: given $E_n \supseteq E_{n-1}$ it fattens $E_n$ through $\extboundary E_n \setminus \extboundary E_{n-1}$ to obtain $E_{n+1}$ as the union of $E_n$ and the traces of the fattening worms.

Of course, to achieve that this exploration runs indefinitely with positive probability (and hence that $\mathcal{S}_{\eta}^\lambda$ percolates almost surely), we also need that $E_n$ is so large compared to $E_{n-1}$ that it is still well visible by the newcoming worms even through the ``lens'' of $\extboundary E_n \setminus \extboundary E_{n-1}$.
From the point of view of random walkers, this visibility of a set is characterized by its capacity, defined in~\eqref{def:capacity}.
Fortunately, as the capacity of a set is comparable to its size on a nonamenable graph (see Lemma~\ref{lemma:capacity_bound}), the condition above reduces to a condition on the sizes, that is, it is enough to guarantee that $|E_n| > a \cdot |E_{n-1}|$ with some large enough $a \gg 1$ depending only on the graph.
As we will see in Proposition~\ref{p.growth}, this follows if $R$ is chosen large enough.

Let us emphasize here that one of the main reasons why we could prove Theorem \ref{thm:random_length_worms_model_on_nonamenable_graphs} with the outlined argument is the aforementioned control on a worm: its trace, a trajectory, can bounce back from the already explored set.
Consequently, due to the defining properties of Poisson point processes, cf.\ Theorem \ref{thm:basic_properties_of_PPPs}, we can do the fattening of the set already explored through the available subset of its exterior boundary by independent packages of fattening worms (a package being the collection of walkers bouncing back from the same vertex), which makes it possible to derive first and second moment bounds on the size of the set occupied by the newcomers, the ``fat''; see Subsection~\ref{subsection:worms_on_nonamenable_graphs_fattening} for the details.

\medskip

Again, one may wonder how far this strategy could be pushed. Worms are able to touch any given set and bounce back, but general animals do not have such capability: it is not even known if for any Cayley graph $\graph$ there is a constant $K<\infty$ such that for any two finite sets $A,B\subset V(\graph)$ there is a graph automorphism $\gamma$ that achieves $0< \dist_\graph(A,\gamma(B)) < K$. See \cite[Question 3]{BandyopadhyaySteifTimar2010} or \cite[Exercise 5.10]{PGG}.
This gap cannot be bridged even by sprinkling, as used in the proof of Theorem \ref{t.free}. And even if this issue of touching was solved for some animals, it is not clear how these touching animals can overlap with each other, making the existence and linearity of any notion of ``capacity'' questionable.

%

\subsection{Point processes of rooted animals}
\label{subsection:point_processes_of_rooted_animals}

The goal of this subsection is to provide an easy to use, but still general enough object, namely a Poisson point process on a general enough space, that can be helpful in the analysis of the Poisson zoo and it special cases.
Let us mention that although there is a large overlap between the construction we will describe here and the one given in \cite[Section 5]{RathRokob2022}, we will build up the notations and the concepts from zero again.
On one hand, this will make this paper more self-contained, and on the other hand here we will use rooted simple sets (animals) as the points of our base space, instead of trajectories, i.e., indexed multisets, which were used in \cite{RathRokob2022}.

\subsubsection{The universal zoo}

Let us denote the space of all rooted animals, i.e., \textbf{the universal zoo} by
\begin{equation}
	\label{def:universal_zoo}
	\zoospace := \zoospace(\graph) := \left\{ (x,H) \, : \, H \in \latticeanimals_{x}, \, x \in V(\graph) \right\}.
\end{equation}
Since our graph $\graph$ is locally finite and each animal is a finite set, this set is countable.
Moreover, this space can be endowed with a natural $\sigma$-algebra $\zooalgebra$: all subsets of $\zoospace$.

Although for the different special cases of the Poisson zoo, different characteristics are relevant for the species --- e.g., the radius for the balls, or the number of steps taken for the worms --- we will only use the \textbf{volume} function on $\zoospace$, i.e.,
\begin{equation}
	\label{def:volume}
	\volume((x,H)) := \left| H \right|
\end{equation}
for any $(x,H) \in \zoospace$.

Obviously, this volume can be used to decompose $\zoospace$ into disjoint subspaces:
\begin{align}
	\label{def:subzoo_of_fixed_volume_animals}
	\zoospace =	\bigcup_{k = 1}^{\infty} Z_k,
	\qquad\text{where } 
	Z_k:=	\big\{ (x, H) \; : \; x \in V(\graph), \ H \in \latticeanimals_x, \ \volume((x,H)) = k \big\}.
\end{align}

\begin{definition}
	\label{def:set_hitting_animals}
	For any $K \subset \subset V(\graph)$, let us introduce the notation
	\begin{equation}
		\label{eq:set_hitting_animals}
		\zoospace(K) := 
		\left\{
		(x, H) \in \zoospace \, : \, H \cap K \neq \emptyset
		\right\},
	\end{equation}
	the subset of animals that hit the set $K$.
	Analogously, let us also define, for any $k \in \N_{>0}$,
	\begin{equation}
		\label{eq:set_hitting_animals_of_volume_k}
		\zoospace_k(K) := 
		\left\{
		(x, H) \in \zoospace \, : \, H \cap K \neq \emptyset, \, \volume((x,H)) = k
		\right\},
	\end{equation}
	and denote the finite union of such sets, for $R \in \N$, as
	\begin{equation}
		\label{eq:set_hitting_animals_of_volume_at_most_R}
		\zoospace^R(K) :=
		\bigcup_{k = 0}^{R} Z_k(K) =
		\left\{
		(x, H) \in \zoospace \, : \, H \cap K \neq \emptyset, \, \volume((x,H)) \leq R
		\right\}.
	\end{equation}
\end{definition}

\subsubsection{Point measures and their occupation measures}

Let us denote by $\zooPP$ the \textbf{space of point measures} on $\zoospace$, that is,
\begin{equation}
	\animals \in \zooPP \, \iff \,
	\animals = \sum_{i \in I} \delta_{(x,H)_i},
\end{equation}
where $I$ is a finite or countably infinite set of indices, $(x,H)_i \in \zoospace$ for all $i \in I$ and $\delta_{(x,H)}$ denotes the Dirac mass concentrated on $(x,H) \in \zoospace$. {The topology on $\zooPP$ is the discrete one.}

Let us denote by $\animals(A)$ the $\animals$-measure of $A \in \zooalgebra$. 
We can alternatively think of an $\animals \in \zooPP$ as a multiset of animals, where $\animals(\{ (x,H) \})$ is the number of copies of the animal $(x,H)$ contained in $\animals$.
Since later on we want to consider random variables defined on the space $\zooPP$ of the form $\animals(A)$ for any $A \in \zooalgebra$, we define $\zooPPalgebra$ to be the sigma-algebra on $\zooPP$ generated by such variables.

If $\animals  = \sum_{i \in I} \delta_{(x,H)_i} \in \zooPP$ and $A \in \zooalgebra$, let us denote by
\begin{equation}
	\label{def:restricted_zoo}
	\animals \ind[A] := \sum_{i \in I} \delta_{(x,H)_i} \ind[ (x,H)_i \in A ]
\end{equation}
the \textbf{restriction} of $\animals$ to the set $A$ of animals. Note that $ \animals \ind[A]$ is also an element of $\zooPP$. 

Given $(x,H) \in \zoospace$ and $\animals = \sum_{i \in I} \delta_{(x,H)_i} \in \zooPP$, we define the \textbf{traces} $\trace((x,H)) \subset V(\graph)$ and $\trace( \animals ) \subseteq V(\graph)$ by
\begin{equation}
	\label{def:trace_of_animals_and_zoo}
	\trace( (x,H) ):= H, \qquad \text{and} \qquad
	\trace\left( \animals \right):= \bigcup_{i \in I} \trace((x,H)_i);
\end{equation}
that is, 
$\trace\left( \animals \right)$ is the set of all sites occupied by $(x,H)_i, i \in I$.


Given $(x,H) \in \zoospace$ and $\animals = \sum_{i \in I} \delta_{(x,H)_i} \in \zooPP$, we define the \textbf{occupation measures} $\LTmeasure^{(x,H)}$ and $\LTmeasure^{\animals}$ on $\graph$ by
\begin{equation}
	\label{def:local_time_measure}
	\LTmeasure^{(x,H)} := \sum_{y \in H} \delta_{y}, \qquad \text{and} \qquad \LTmeasure^{\animals}:= \sum_{i \in I}  \mu^{(x,H)_i}.
\end{equation}
Thus, if $y \in V$, then $\LTmeasure^{(x,H)}(y) := \LTmeasure^{(x,H)}(\{y\}) = \ind \left[ y \in H \right]$ and $\LTmeasure^{\animals}(y) := \LTmeasure^{\animals}(\{y\})$ equals the number of sets containing $y$.  
Given some  $\animals = \sum_{i \in I} \delta_{(x,H)_i} \in \zooPP$, we define the \textbf{total size} or \textbf{total occupation measure} $\totalsize^{\animals}$ of the animals $(x,H)_i, i \in I$, by
\begin{equation}
	\label{def:total_local_time}
	\totalsize^{\animals} :=
	\mu^{\animals}(V) = \sum_{i \in I} \volume((x,H)_i).
\end{equation}

\subsubsection{The zoo as a Poisson point process}

Recall the lattice animal measure $\nu=\nu_o$ from Definition~\ref{def:grp-rotation-invariant_prob_measure}, and that~\eqref{def:translated_prob_measure} gives us, for any $x \in V(\graph)$, a translated version $\nu_{x}$ on $\latticeanimals_{x}$.  Together they induce a probability measure
\begin{equation}
	\label{def:induced_measure_on_zoospace}
	\nu := \mathop{\bigotimes}_{x \in V(\graph)} \nu_{x}
\end{equation}
on the universal zoo $Z$ defined in~\eqref{def:universal_zoo}. That is, for an animal $(x,H) \in \zoospace$, we have $\nu((x,H)) := \nu( \{ (x,H) \} ) = \nu_x(H)$.

\begin{definition}
	\label{def:Poisson_zoo_as_PPP}
	Given some $\lambda \in \R_{>0}$ and a rotation invariant probability measure $\nu_{\origin}$ on $\latticeanimals_{\origin}$, a random element $\animals$ of $\zooPP$ has law $\mathcal{Q}_{\nu}^{\lambda}$ if $\animals$ is a Poisson point process (abbreviated as PPP) on $\zoospace$ with intensity measure $\lambda \cdot \nu( (x,H) )$, where $\nu$ is the measure induced by $\nu_{\origin}$ as in \eqref{def:induced_measure_on_zoospace}.
\end{definition}

%

\begin{claim}
	\label{claim:Poisson_zoo_set_as_the_trace_of_a_PPP}
	The trace $\mathcal{S}_{\nu}^\lambda$ of the Poisson zoo of Definition~\ref{def:poisson_zoo}  can be alternatively defined as $\mathcal{S}_{\nu}^\lambda := \trace ( \animals )$, where $\animals \sim \mathcal{Q}_{\nu}^\lambda$.
\end{claim}


The special case of  the random length worms model deserves a separate notation:

\begin{definition}
	\label{def:random_length_worms_as_PPP}
	Given a probability mass function $\eta \, : \, \N_{>0} \rightarrow [0,1]$ for the random length worms model of Definition~\ref{def:random_length_worms_model}, the law $\mathcal{Q}_{\nu}^\lambda$ of the corresponding Poisson point process on $\zoospace$ will instead be denoted by $\mathcal{P}^\lambda := \mathcal{P}_{\eta}^{\lambda}$.	
\end{definition}

The main reason for looking at the Poisson zoo through the lenses of Poisson point processes is that for the latter there is a well developed theory and hence many tools to use for analysis. Fortunately, there are only a few of them are relevant for our results, which we will collect here from \cite{Kingman1993} and \cite[Chapter 3]{Resnick2008} after translating them into our setting.

\begin{theorem}
	\label{thm:basic_properties_of_PPPs}
	Let $\animals = \sum_{i \in I} \delta_{(x,H)_i} \sim \mathcal{Q}_{\nu}^\lambda$ be as in Definition \ref{def:Poisson_zoo_as_PPP} and given any $A \in \zooalgebra$ introduce the restricted measure $\nu_A(\cdot) := \nu ( \cdot \cap A)$.
	\begin{enumerate}[label=(\roman*)]
		\item 
		\label{PPP_prop:restriction}
		If $A \in \zooalgebra$, then the restricted process $\animals \ind [A]$ has law $\mathcal{Q}_{\nu_A}^\lambda$.
		\item 
		\label{PPP_prop:independence}
		If $A_1, A_2, \ldots \in \zooalgebra$ are such that $A_i \cap A_j = \emptyset$ for $i \neq j$, then the restricted processes $\animals \ind [A_1], \animals \ind[A_2], \ldots$ are independent.
		\item 
		\label{PPP_prop:colouring}
		Let $p \in [0,1]$ and $\{ \xi_i \}_{i \in I}$ a collection of independent $\Bernoulli(p)$ random variables, and consider the following decomposition:
		\begin{equation}
			\animals = \animals_1 + \animals_2 := \sum_{i \in I} \xi_i \delta_{(x,H)_i} + \sum_{i \in I} (1 - \xi_i) \delta_{(x,H)_i}.
		\end{equation}
		Then $\animals_1 \sim \mathcal{Q}_{\nu}^{p v}$ and $\animals_2 \sim \mathcal{Q}_{\nu}^{ (1-p)v }$ in law, and they are independent of each other.
		\item
		\label{PPP_prop:law}
		If $A \in \zooalgebra$, then $\animals(A)$ follows a Poisson distribution.
	\end{enumerate}
\end{theorem}

Let us note here that \eqref{PPP_prop:colouring} is sometimes called the colouring property of the Poisson point process, as intuitively it means that colouring all points of the process with two colours, red and blue, independently of each other, both coloured processes are again Poisson point processes, with the appropriately compensated intensity measures.
Moreover, properties \eqref{PPP_prop:independence} and \eqref{PPP_prop:law} are defining properties of the Poisson point process; see for example~\cite[Section 3.3]{Resnick2008}.

One consequence of the above theorem that we will use is that one can express the expectation of any nonnegative function $f(\animals)$ of a sample $\animals \sim \mathcal{Q}_{\nu}^{\lambda}$ using its intensity measure $\nu$. The proof of this formula follows the lines of \cite[Lemma 5.3]{RathRokob2022}, hence it is omitted.

\begin{lemma}
	\label{lemma:general_Campbell_formulas}
	Given $\animals = \sum_{i \in I} \delta_{(x,H)_i} \sim \mathcal{Q}_{\nu}^\lambda$ and a function $f \, : \, \zoospace \longrightarrow \R^+$, we have
	\begin{equation}
		\label{eq:general_Campbell_expectation_formula}
		\E \left[ \sum_{i \in I} f((x,H)_i) \right] =
		\lambda \cdot \sum_{ x \in V } \sum_{H \in \mathcal{H}_x}  f \left( (x,H) \right) \cdot \nu \left( (x,H) \right).
	\end{equation}
\end{lemma}

\section{First and second moment bounds on the fat}
\label{section:bounds_on_fat}

As we explained in Subsections~\ref{subsection:poisson_zoo_on_free_products} and~\ref{subsection:random_length_worms_model_introduction}, both of our theorems will be proved via certain exploration processes, fattening the so-far explored cluster of the origin step-by-step, by animals that intersect the exterior boundary of the current cluster but not the cluster itself. In both cases, we will need to have further restrictions on where these new animals touch the current cluster, making it useful to introduce the following notation.

\begin{definition}
	\label{def:point_measure_restricted_to_hit_B_but_not_the_closure_of_A}
	Given a point measure $\animals = \sum_{i \in I} \delta_{(x,H)_i} \in \zooPP$, some $R \in \N_{>0}$, a finite connected set $A \subseteq V(\graph)$ and a subset of its exterior boundary $B \subseteq \extboundary A$ let us introduce the following restriction of the point measure:
	\begin{equation}
		\label{eq:point_measure_restricted_to_hit_B_but_not_the_closure_of_A}
		\check{ \animals }_{A,B}^R :=
		\animals \, \ind \left[
		\zoospace^R(B) \setminus \zoospace^R \left( \overline{A} \setminus B \right)
		\right].
	\end{equation}
	That is, we consider those animals only that have volume at most $R$, hit $B$, but do not intersect $\overline{A}$ otherwise. Let us also introduce the notation 
	\begin{equation}
		\label{def:point_measure_restricted_to_hit_B_but_not_the_closure_of_A_total_local_size}
		\check{\mu}_{A,B}^R (x) := \mu^{ \check{ \animals }_{A,B}^R }(x)
		\qquad\text{and}\qquad 
		\check{ \totalsize }_{A,B}^R := \totalsize^{ \check{ \animals }_{A,B}^R } = \sum_{x \in V(\graph)} \check{\mu}_{A,B}^R (x)
	\end{equation}
	for the (total) occupation measure of the restricted point measure.
\end{definition}

The exploration processes will be introduced and analysed in Section~\ref{section:exploration_processes}, but
quite evidently, whatever the exact exploration is, we will need to prove lower bounds on $| \trace( \check{ \animals }_{A,B}^R ) |$: in expectation for the branching process argument on free products; in probability for the more complicated exploration process for worms on general graphs. In the next subsection, we will argue that, instead of the trace, it is enough to examine $\check{ \totalsize }_{A,B}^R$, obviously easier to handle. Then, in Subsection~\ref{ss.fatfree}, we will give first moment lower bounds for general zoos on free products, and in Subsection~\ref{subsection:worms_on_nonamenable_graphs_fattening}, first moment lower and second moment upper bounds for worms.

%
%
%

\subsection{Size-biasing and neglecting multiplicities}

Recall the notation $m_k$ from Definition \ref{def:moments_of_nu}. Due to Lemma~\ref{lemma:two_silly_lemmas}~\ref{lemma:two_silly_lemmas:infinite_m1_implies_percolation}, we can assume that $m_1 < \infty$. Given a fixed integer $R \in \N_{>0} \cup \{ + \infty \}$, called the truncation parameter (whose value will be chosen later), and any $k \in \N_{>0}$, we introduce the notion of truncated moments as
\begin{equation}
	\label{def:truncated_moments_of_nu}
	m_k^{R} := \sum_{H \in \mathcal{H}_{\origin}} |H|^k \cdot \ind[|H| \leq R] \cdot \nu_{\origin}(H).
\end{equation}

We defined unimodularity of a transitive graph in Definition~\ref{def:unimodularity}, via the mass transport principle~\eqref{eq:MTP}. An important consequence is the following \emph{size-biasing} phenomenon.

\begin{lemma}[Size-biasing in the Poisson zoo]
	\label{lemma:size_biasing_for_Poisson_zoo}
	Let us assume that $\grp \subgrp \Aut(\graph)$ is a transitive and unimodular subgroup of automorphisms, $\nu_{\origin}$ is a $\grp$-rotation invariant measure on $\latticeanimals_{\origin}$, $\nu$ is the measure on $\zoospace$ induced from $\nu_{\origin}$ via \eqref{def:induced_measure_on_zoospace} and $\animals = \sum_{i \in I} \delta_{(x,H)_i} \sim \mathcal{Q}_{\nu}^\lambda$. Given some $R \in \N \cup \{ + \infty \}$ and $x \in V(\graph)$, consider the restricted processes $\animals^R := \animals \ind \left[ \zoospace^R \right]$ and $\animals_x^R := \animals \ind \left[ \zoospace^R( \{ x \} ) \right]$. Then we have
	\begin{equation}
		\label{eq:size_biasing_for_Poisson_zoo}
		\E \left[  \LTmeasure^{\animals^R}(x)  \right] = \lambda \cdot m_1^R
		\qquad \text{and} \qquad
		\E \left[ \totalsize^{ \animals_x^R } \right] = \lambda \cdot m_2^R.
	\end{equation}
\end{lemma}

\begin{proof}
	Both equalities follow from the mass transport principle \eqref{eq:MTP} used with some appropriate mass transport functions.
	Both transport functions can be derived from the following collection of functions.
	For $H \in \latticeanimals_{\origin}$ and $y \in V(\graph)$ let us define
	\begin{equation}
		f_H^R(y,x) := \ind \left[ x \in \trace \left( y, \varphi_y(H) \right) \right] \cdot \ind \left[ |H| \leq R \right],
	\end{equation}
	where $\varphi_y \in \grp$ is the one granted by \eqref{def:the_distinguished_automorphism_taking_origin_to_x}.
	Note that $f_H^R$ is not necessarily a diagonally invariant function (hence not a mass transport function), since there could be automorphisms $\gamma \in \grp$  for which $\varphi_{ \gamma(y) }(H) \neq \gamma ( \varphi_y(H) )$ holds.
	However, the functions
	\begin{equation}
		\label{eq:mt_function_for_size_biasing}
		F_1^R(y,x) := \sum_{H \in \latticeanimals_{\origin}} f_H^R(y,x)
		\qquad \text{and} \qquad
		F_2^R(y,x) := \sum_{H \in \latticeanimals_{\origin}} |H| \cdot f_H^R(y,x)
	\end{equation}
	are mass transport functions.
	Indeed, it follows easily that they are diagonally invariant under the action of $\grp$, as it defines a permutation on $\latticeanimals_{\origin}$ without changing the volume of the sets.
	
	For the first equality of \eqref{eq:size_biasing_for_Poisson_zoo}, note that  $F_1^R(y,x)$ counts the number of animals of volume at most $R$ that are rooted in $y \in V(\graph)$ and contain $x$. Consequently,
	\begin{align*}
		\E \left[ \LTmeasure^{ \animals^R }(x) \right] 
		& \overset{\eqref{eq:general_Campbell_expectation_formula}}{=}
		\lambda \cdot \E_\nu \left[ \sum_{y \in V(\graph)} F_1^R(y,x) \right] \overset{\eqref{eq:MTP}}{=}
		\lambda \cdot \E_\nu \left[ \sum_{y \in V(\graph)} F_1^R(x,y) \right] \\ & \overset{(\ref{eq:mt_function_for_size_biasing},~\ref{def:translated_prob_measure})}{=}
		\lambda \cdot \sum_{H \in \latticeanimals_{\origin}} \left| H \right| \cdot \ind \left[ |H| \leq R \right] \cdot \nu_{\origin}(H)  \overset{\eqref{def:truncated_moments_of_nu}}{=}
		\lambda \cdot m_1^R.
	\end{align*}
	
	Meanwhile in $F_2^R(y,x)$ we also take into account the volume of the animals. 	
	Hence,
	\begin{align*}
		\E \left[ \totalsize^{ \animals_x^R } \right] 
		& \overset{\eqref{eq:general_Campbell_expectation_formula}}{=}
		\lambda \cdot \E_\nu \left[ \sum_{y \in V(\graph)} F_2^R(y,x) \right] \overset{\eqref{eq:MTP}}{=}
		\lambda \cdot \E_\nu \left[ \sum_{y \in V(\graph)} F_2^R(x,y) \right] \\ & \overset{(\ref{eq:mt_function_for_size_biasing},~\ref{def:translated_prob_measure})}{=}
		\lambda \cdot \sum_{H \in \latticeanimals_{\origin}} \left| H \right|^2 \cdot \ind \left[ |H| \leq R \right] \cdot \nu_{\origin}(H)  \overset{\eqref{def:truncated_moments_of_nu}}{=}
		\lambda \cdot m_2^R,
	\end{align*}
	proving the second equality of \eqref{eq:size_biasing_for_Poisson_zoo}.
\end{proof}

The appearance of the truncated second moment $m_2^R$ in the expected total occupation measure in~\eqref{eq:size_biasing_for_Poisson_zoo}, supplemented with the assumption that $m_2 = +\infty$, can be used to make this quantity sufficiently large, by taking $R$ large. Unimodularity is crucial for this, as shown by the following example.
Recall from the paragraphs right after Definition~\ref{def:unimodularity} that a regular tree with a distinguished end is not unimodular.

\begin{remark}
	\label{remark:unimodularity_is_needed}
	Let $\hat{T}$ denote a decorated $3$-regular tree in which there is a distinguished infinite ray.	
	This fixed infinite ray marks a direction in the graph, namely upwards is when we are moving towards to the end represented by the ray.
	It is easy to see that $\Aut(\hat{T})$ acts transitively on the vertices, hence we can build a Poisson zoo on $\hat{T}$ from some $\Aut(\hat{T})$-rotation invariant measure $\nu$ on the animals of $\latticeanimals_{\origin}$ via Definition \ref{def:poisson_zoo}.
	
	We choose $\nu$ to be 
	\begin{equation}
		\nu (H) := \frac{1}{2^r} \cdot \prob \left( \mathcal{L} = r \right) \cdot \ind \left[ H \text{ is a downward path of length $r$ started from $\origin$} \right],
	\end{equation} 
	where $\mathcal{L}$ is such a random variable that $\prob ( \mathcal{L} = r ) \asymp r^{-3}$ for all $r \in \N_{>0}$.
	Since the automorphisms of $\Aut( \hat{T} )$ fix the distinguished infinite ray, $\nu$ will indeed be $\Aut(\hat{T})$-rotation invariant. Moreover, we are in the interesting regime, as $m_1 = \E[\mathcal{L}] < + \infty$ and $m_2 = \E [ \mathcal{L}^2 ] = + \infty$ hold.
	
	However, using the notation of Lemma \ref{lemma:size_biasing_for_Poisson_zoo}, for the expected total size of animals ever visiting some vertex $x \in V(\hat{T})$ (counted with multiplicity), we only have 
	\[
		\E \left[ \totalsize^{ \animals_x^R } \right] \overset{(*)}{=}
		\lambda \sum_{r = 1}^R  \sum_{k = 0}^{r-1} 2^{r-k} \cdot \frac{r}{2^r} \cdot \prob \left( \mathcal{L} = r \right)
		\asymp \lambda \sum_{r = 1}^R  \frac{1}{r^2}  \sum_{k = 0}^{r-1} \frac{1}{2^k} \leq \lambda \sum_{r = 1}^\infty \frac{2}{r^2} < \infty\,,
	\]
which is not controlled by the truncated second moment anymore, and hence can not be chosen arbitrarily large.
Here in $(*)$ we used \eqref{eq:general_Campbell_expectation_formula} and the fact that the restriction to only downward paths implies that there is only one unique way to get to $x$ from any vertex above.	 
\end{remark}

The next lemma tells us that the expected number of animals hitting some vertex, given that there is at least one of such animal,  can be bounded from above by $\lambda \cdot m_1$. Moreover, this holds for any restricted zoo, as well. Hence, under the natural assumption that $m_1 < \infty$, we will obtain that neglecting multiplicities affects occupation measures only by bounded factors.

\begin{lemma}[Upper bound on the multiplicity]
\label{lemma:upper_bound_on_expected_number_of_occupation}
Under the assumptions of 	Lemma~\ref{lemma:size_biasing_for_Poisson_zoo},
let $\animals = \sum_{i \in I} \delta_{(x,H)_i} \sim \mathcal{Q}_{\nu}^\lambda$.
Given any $A \in \zooalgebra$, let us denote the restricted zoo by $\bar{ \animals } = \animals \, \ind \left[ A \right]$ and the corresponding total occupation measure by $ \bar{\LTmeasure}(x) := \LTmeasure^{ \bar{ \animals } }(x)$.
Then we have
\begin{equation}
	\label{eq:upper_bound_on_expected_number_of_occupation}
	\sup_{x \in V(\graph)} \E \left. \left[ \bar{\LTmeasure}(x) \, \right| \, \bar{\LTmeasure}(x) > 0 \right] \leq \lambda \cdot m_1 + 1.
\end{equation}
If $\bar{ \animals }$ and $x$ are such that $\bar{\LTmeasure}(x)=0$ a.s., we consider the conditional expectation to be 0.

\end{lemma}

\begin{proof}
By \eqref{PPP_prop:law} of Theorem \ref{thm:basic_properties_of_PPPs}, it follows that, for any $x \in V(\graph)$, the distribution of the random variable $\bar{\LTmeasure}(x)$ is Poisson with parameter $\E \left[ \bar{ \LTmeasure }(x) \right]$.
Moreover, if the restriction given by $A$ on $\animals$ is such that $\bar{\LTmeasure}(x)=0$ for all $x$, then \eqref{eq:upper_bound_on_expected_number_of_occupation} follows trivially. Hence let us assume that this is not the case: there exists at least one $x \in V(\graph)$ such that $\bar{\LTmeasure}(x)$ is non-trivial.

From basic probability theory and real analysis $(*)$, for any such $x \in V(\graph)$ we can write the following upper bound:
\begin{equation}
	\label{eq:upper_bound_on_expected_number_of_occupations_given_there_is_one}
	\E \left. \left[ \bar{\LTmeasure}(x) \, \right| \, \bar{\LTmeasure}(x) > 0  \right] =
	\frac{\E \left[ \bar{\LTmeasure}(x) \right]}{ \prob \left( \bar{\LTmeasure}(x) > 0 \right)} =
	\frac{\E \left[ \bar{\LTmeasure}(x) \right]}{ 1- \exp \left\{ - \E \left[ \bar{\LTmeasure}(x) \right] \right\} } \overset{(*)}{\leq} 
	\E \left[ \bar{\LTmeasure}(x) \right] + 1.
\end{equation}
At the same time, the expectation appearing on the right-hand side of \eqref{eq:upper_bound_on_expected_number_of_occupations_given_there_is_one} can be bounded using Lemma \ref{lemma:size_biasing_for_Poisson_zoo} as
\begin{equation}
	\label{eq:MTP_argument_for_expected_number_of_occupations}
	\E \left[ \bar{\LTmeasure}(x) \right] \overset{(**)}{\leq} 
	\E \left[ \LTmeasure^{ \animals }(x) \right] \overset{\eqref{eq:size_biasing_for_Poisson_zoo}}{=} \lambda \cdot m_1,
\end{equation}
where $(**)$ holds since we left the restriction on the animals.
Putting \eqref{eq:MTP_argument_for_expected_number_of_occupations} back into \eqref{eq:upper_bound_on_expected_number_of_occupations_given_there_is_one} we arrived at an upper bound that does not depend on $x$, hence yields the desired result.
\end{proof}

The previous lemma will only be used via the following corollary:

\begin{corollary}[Neglecting multiplicities]
\label{coro:expected_size_without_and_with_multiplicity}
Under the assumptions of 	Lemma~\ref{lemma:size_biasing_for_Poisson_zoo},
let $\animals = \sum_{i \in I} \delta_{(x,H)_i} \sim \mathcal{Q}_{\nu}^\lambda$.
Given any $A \in \zooalgebra$, let us denote the restricted zoo by $\bar{ \animals } = \animals \, \ind \left[ A \right]$ and the corresponding total occupation measure, as in~\eqref{def:point_measure_restricted_to_hit_B_but_not_the_closure_of_A_total_local_size}, by
\begin{equation*}
	\bar{ \totalsize } := 
	\totalsize^{ \bar{ \animals } } :=
	\sum_{x \in V(\graph)} \bar{\LTmeasure}(x) :=
	\sum_{x \in V(\graph)} \LTmeasure^{ \bar{ \animals } }(x).
\end{equation*}
Then we have
\begin{equation}
	\label{eq:expected_size_without_and_with_multiplicity}
	\E \left[ \left| \trace \left( \bar{\animals} \right) \right| \right] \geq 
	\frac{1}{\lambda \cdot m_1 + 1} \cdot \E \left[ \bar{\totalsize} \right].
\end{equation}
\end{corollary}

\begin{proof}
From \eqref{def:trace_of_animals_and_zoo} it follows that we have
\begin{equation}
	\label{obs:size_of_fat}
	\left| \trace \left( \bar{\animals} \right) \right| = \sum_{x \in V(\graph)} \ind \left[ \bar{ \LTmeasure }(x) > 0 \right].
\end{equation}
Hence the law of total expectation and Lemma \ref{lemma:upper_bound_on_expected_number_of_occupation} yield
\begin{align*}
	\E \left[ \bar{\totalsize} \right] &= 
	\sum_{x \in V(\graph)} \E \left[ \bar{\LTmeasure} (x) \right] =
	\sum_{x \in V(\graph)} \E \left[ \bar{\LTmeasure} (x)  \, \left| \, \bar{\LTmeasure} (x) > 0 \right. \right] \cdot \prob \left( \bar{\LTmeasure} (x) > 0 \right) \\ 
	& \leq
	\sup_{x \in V(\graph)} \E \left. \left[ \bar{ \LTmeasure }(x) \, \right| \, \bar{\LTmeasure}(x) > 0 \right] \cdot \E \left[ \sum_{x \in V(\graph)} \ind \left[ \bar{ \LTmeasure }(x) > 0 \right] \right] \overset{\eqref{eq:upper_bound_on_expected_number_of_occupation}}{\leq}
	(\lambda \cdot m_1 + 1) \cdot \E \left[ \left| \trace \left( \bar {\animals} \right) \right| \right],
\end{align*}
from which our statement follows after a rearrangement.
\end{proof}

\subsection{The fat of general animals on free products}\label{ss.fatfree}

Taking into account the parity problem of a free product, Figure~\ref{f.free2}, it is possible that the expected size of the fat can only be large when it comes from the direction of one of the components of the free product.
This favoured direction may depend on not only the chosen rotation invariant measure $\nu_{\origin}$, but also on the cutoff $R$, resulting in the awkward split formulation of the next two statements.

\begin{lemma}
\label{lemma:lower_bound_on_expected_total_size_of_animals_hitting_one_vertex_on_free_product}
Consider a nonamenable graph $\graph$ that can be obtained as a free product $\graph_1 \star \graph_2$ of two transitive graphs, and let $\grp \subgrp \Aut(\graph)$ be such that it acts transitively and unimodularly on $\graph$.
Let $\nu_{\origin}$ be a $\grp$-rotation invariant measure on $\latticeanimals_{\origin}$, then $\nu$ the measure on $\zoospace$ induced from $\nu_{\origin}$ via \eqref{def:induced_measure_on_zoospace}, and $\animals = \sum_{i \in I} \delta_{(x,H)_i} \sim \mathcal{Q}_{\nu}^\lambda$.
Given any $R \in \N_{>0}$ and $x \in V(\graph)$, for the restriction of the process introduced in Definition \ref{def:point_measure_restricted_to_hit_B_but_not_the_closure_of_A}, depending on $\nu_{\origin}$ and maybe $R$, at least one of 
\begin{equation}
	\label{eq:lower_bound_on_expected_total_size_of_animals_hitting_one_vertex_on_free_product}
	\E \left[ \check{ \totalsize }_{\{x\}, \, \mathcal{N}_1(x) }^R \right] \geq \tfrac{\Cheeger(\graph)}{2} \cdot \lambda \cdot m_2^R
	\qquad \text{or} \qquad
	\E \left[ \check{ \totalsize }_{\{x\}, \, \mathcal{N}_2(x)}^R \right] \geq \tfrac{\Cheeger(\graph)}{2} \cdot \lambda \cdot m_2^R
\end{equation}
holds, where $\mathcal{N}_i(x)$ is the $G_i$-neighbourhood of $x$, as defined in~\eqref{def:free_product_term_neighbourhood}.	
\end{lemma}

\begin{proof}
Although the proof follows from the same ideas as of Lemma \ref{lemma:size_biasing_for_Poisson_zoo}, let us provide a detailed explanation.
Due to \eqref{def:the_distinguished_automorphism_taking_origin_to_x}, we can define the functions:
\begin{equation}
	f_H^R(y,z) := |H| \cdot \ind \big[ z \in \extboundary \trace \left( y, \varphi_y(H) \right) \big] \cdot \ind \left[ |H| \leq R \right],
	\qquad
	y,z \in V(\graph), \; H \in \latticeanimals_{\origin}.
\end{equation}
In words, $f_H^R(y,z)$ is a transport function in which $y$ sends weight $|H|$ to $z$ if the latter is an element of the exterior boundary of $\varphi_{y}(H)$ for $H \in \latticeanimals_{\origin}$ of size at most $R$.
As before, $f_H^R$ is not necessarily a diagonally invariant function.
However, if we take the sum of these functions over all $H \in \latticeanimals_{\origin}$, then the resulting function 
will be invariant:
\begin{equation}
	\label{def:MTP_function_for_animals_of_free_product}
	F^R(y,z) := \sum_{H \in \latticeanimals_{\origin}} f_H^R(y,z); 
	\qquad
	y, z \in V(\graph).
\end{equation} 

Now, if we denote by $\graph^x$ the connected component of $\graph$ that contains $x$ after erasing all the edges $(x, z)$, $z  \in \mathcal{N}_2(x)$, then the expectations of \eqref{eq:lower_bound_on_expected_total_size_of_animals_hitting_one_vertex_on_free_product} can be rewritten as
\begin{equation*}
	\label{eq:expected_weight_coming_from_graph1_component}
	\E \left[ \check{ \totalsize }_{\{x\}, \, \mathcal{N}_1(x)}^R \right] \overset{\eqref{eq:general_Campbell_expectation_formula}}{=}
	\lambda \cdot \E_\nu \left[ \sum_{y \in V(\graph^x)} F^R(y,x) \right]
\end{equation*}
and
\begin{equation*}
	\label{eq:expected_weight_coming_from_graph2_component}
	\E \left[ \check{ \totalsize }_{\{x\}, \, \mathcal{N}_2(x) }^R \right] \overset{\eqref{eq:general_Campbell_expectation_formula}}{=}
	\lambda \cdot \E_\nu \left[ \sum_{y \in V(\graph) \setminus V(\graph^x)} F^R(y,x) \right].
\end{equation*}
Indeed, since the vertex $x$ is a cutpoint of the graph (see Claim~\ref{claim:properties_of_free_product_of_graphs}), the exterior boundary of a trace of an animal can contain $x$ but avoid all the vertices of $\mathcal{N}_2(x)$ only if its root is in the subgraph $\graph^x$ (or more precisely in $V(\graph^x) \setminus \{ x \}$), and can avoid $\mathcal{N}_1(x)$ only if the root is in $V(\graph) \setminus V(\graph^x)$.
Then, obviously,
\begin{equation}
	\label{eq:decomposition_of_the_whole_incoming_expected_sum_into_incoming_weights_from_different_directions}
	\E_\nu \left[ \sum_{y \in V(\graph)} F^R(y,x) \right] = 
	\E_\nu \left[ \sum_{y \in V(\graph^x)} F^R(y,x) \right] + \E_\nu \left[ \sum_{y \in V(\graph) \setminus V(\graph^x)} F^R(y,x) \right],
\end{equation}
whose left-hand side can be lower bounded, using the mass transport principle at step $(*)$, as
\begin{align*}
	\E_\nu \left[ \sum_{y \in V(\graph)} F^R(y,x) \right] & \overset{(*)}{=}
	\E_\nu \left[ \sum_{y \in V(\graph)} F^R(x,y) \right] \overset{\eqref{def:MTP_function_for_animals_of_free_product}}{=}
	\sum_{H \in \latticeanimals_{\origin}, \, |H| \leq R} \left| \extboundary H \right| \cdot |H| \cdot \nu_{\origin}(H) \\ & \overset{(**)}{\geq}
	\Cheeger(\graph) \cdot \sum_{H \in \latticeanimals_{\origin}, \, |H| \leq R} |H|^2 \cdot \nu_{\origin}(H) =
	\Cheeger(\graph) \cdot m_2^R.
\end{align*}
Here $(**)$ follows from the assumed nonamenability of the graph $\graph$.
Consequently, to finish the proof one only needs to note that depending on the distribution $\nu_{\origin}$ and maybe on the given $R$, either the first  or the second term on the right-hand side of 		\eqref{eq:decomposition_of_the_whole_incoming_expected_sum_into_incoming_weights_from_different_directions} is at least as large as the other, hence at least half of the sum.	Combining the above inequalities gives the result.
\end{proof}

What we really need is the size of the trace, i.e., the union of the fattening animals without multiplicities.
This is an immediate consequence of Lemma \ref{lemma:lower_bound_on_expected_total_size_of_animals_hitting_one_vertex_on_free_product} and Corollary \ref{coro:expected_size_without_and_with_multiplicity}.

\begin{corollary}
\label{coro:lower_bound_on_expected_size_of_fat_of_animals_through_one_vertex_on_free_product}
Let us consider the nonamenable graph $\graph$ that can be obtained as a free product $\graph_1 \star \graph_2$ of two transitive graphs, and let $\grp \subgrp \Aut(\graph)$ act transitively and unimodularly on $\graph$.
Let $\nu_{\origin}$ be a $\grp$-rotation invariant measure on $\latticeanimals_{\origin}$, then $\nu$ the measure on $\zoospace$ induced from $\nu_{\origin}$ via \eqref{def:induced_measure_on_zoospace}, and $\animals = \sum_{i \in I} \delta_{(x,H)_i} \sim \mathcal{Q}_{\nu}^\lambda$.
Given any $R \in \N_{>0}$ and $x \in V(\graph)$, for the restriction of the process introduced in Definition \ref{def:point_measure_restricted_to_hit_B_but_not_the_closure_of_A}, depending on $\nu_{\origin}$ and maybe $R$, at least one of 
\begin{equation}
	\label{eq:lower_bound_on_expected_size_of_fat_of_animals_through_one_vertex_on_free_product}
	\E \left[ \left| \trace \left( \check{ \animals }_{\{x\}, \, \mathcal{N}_1(x) }^R \right) \right| \right] \geq
	\tfrac{\Cheeger(\graph) \cdot \lambda}{2( \lambda \cdot m_1 + 1 )} \cdot m_2^R
	\quad \text{or} \quad
	\E \left[ \left| \trace \left( \check{ \animals }_{\{x\}, \, \mathcal{N}_2(x)}^R \right) \right| \right] \geq \tfrac{\Cheeger(\graph) \cdot \lambda}{2( \lambda \cdot m_1 + 1 )} \cdot m_2^R
\end{equation}
holds. \hfill $\square$
\end{corollary}

\subsection{The fat of worms on nonamenable graphs}
\label{subsection:worms_on_nonamenable_graphs_fattening}

Recall the random length worms model from Definition~\ref{def:random_length_worms_as_PPP}, and the strategy outlined in  Subsection~\ref{subsection:random_length_worms_model_introduction}. Compared to general animals on free products, one complication is that we do not have the abundance of cut vertices that made it easy (except for the parity issue) to find animals bouncing back from the exterior boundary of the current cluster.  But worms have an advantage here: they can navigate more freely on the vertices. Thus, for any vertex $x$ in the exterior boundary of a so far explored set $A$, provided that it is reasonably accessible for random walks started outside of $A$, the expected number of worms bouncing back from $x$ will be large; this will be Lemma~\ref{lemma:lower_bound_on_the_expected_total_size_of_worms_hitting_one_vertex}. Moreover, due to random walk capacity being linear in $|A|$, see Lemma~\ref{lemma:capacity_bound}, the number of ``reasonably accessible'' vertices is  linear. More precisely, we will use Corollary~\ref{coro:lower_bound_on_partial_square_term_capacity}, together with Corollary~\ref{coro:expected_size_without_and_with_multiplicity} on neglecting multiplicities, to establish a good lower bound on the expected size of the fat, Corollary~\ref{coro:lower_bound_on_expected_size_of_fat_worms_on_nonamenable_graphs}. 

Compared to free products, here comes the second complication: on a general nonamenable graph, the worms bouncing back from different exterior vertices can get tangled in the newly explored set, hence it is unclear how to define a branching process that the exploration dominates. Instead, we will need to work with the entire fat together, and a large expectation will not suffice any more --- we will also need a variance upper bound, to be given in Lemma~\ref{lemma:upper_bound_on_the_second_moment_of_the_occupied_sets_size_for_worms}.

Here comes our first moment lower bound for a single exterior vertex.

\begin{lemma}
\label{lemma:lower_bound_on_the_expected_total_size_of_worms_hitting_one_vertex}
Let us consider a nonamenable graph $\graph$ with a subgroup of its automorphisms $\grp \subgrp \Aut(\graph)$ that acts transitively on $\graph$. Let $\eta \, : \, \N_{>0} \rightarrow [0,1]$ be a probability mass function and $\animals = \sum_{i \in I} \delta_{(x,H)_i} \sim \mathcal{P}_{\eta}^\lambda$.
Given some $R \in \N_{>0}$, a finite connected set $A \fsubset V(\graph)$ and a vertex from its exterior boundary, $x \in \extboundary A$, let us consider
the total occupation measure $\check{ \totalsize }_{A, \{ x\}}^R$ corresponding to the restricted process introduced in \eqref{eq:point_measure_restricted_to_hit_B_but_not_the_closure_of_A}.
For any $\veps > 0$, with the constant $c(\veps) > 0$ from Corollary~\ref{coro:number_of_steps_and_size_of_trace_of_SRW_are_comparable_on_nonamenable_graphs}, we have
\begin{equation}
	\label{eq:lower_bound_on_the_expected_total_size_of_worms_hitting_one_vertex}
	\E \big[ \check{ \totalsize }_{A, \{ x \}}^R \big] \geq c(\veps) \cdot
	\lambda \cdot 
	\E \left[ \mathcal{L}^2 \, \ind \left[ \mathcal{L} \leq R \right] \right] \cdot  
	P_x \left( T_{\overline{A}}^+ = +\infty \right) \cdot 
	\left( P_x \left( T_{\overline{A}}^+ = +\infty \right) - \veps \right),
\end{equation}
where $\mathcal{L}$ denotes a random variable that has law $\eta$.
\end{lemma}

\begin{proof}
To prove the lemma, we will want to work with the trajectories of  the random walks instead of their traces only. The first step into this direction is the use of Lemma~\ref{lemma:general_Campbell_formulas} as
\begin{align}
	\label{eq:expected_total_size_of_fat_worms_after_Campbell_formula}
	\E \big[  \check{ \totalsize }_{A, \{ x \}}^R \big] \overset{\eqref{eq:general_Campbell_expectation_formula}}{=}
	\lambda \cdot \sum_{z \in V(\graph)} \sum_{H \in \latticeanimals_z} |H| \cdot \ind \left[
	\parbox{7.5em}{
		\small{
			$ |H| \leq R,\  x \in H,\\ 
			H \cap \overline{A} = \{ x \}$}
	} \right] \cdot P_z \left( \trace (X[0, \mathcal{L})) = H \right),
\end{align}
where recall from Definition \ref{def:random_length_worms_as_PPP} that $\mathcal{L}$ is not only a random variable with law $\eta$, but it is independent of the random walk.

Now, because of this independence and the fact that a random walk taking at most $R$ steps can only generate a set of volume at most $R$, the probability on the right hand side of \eqref{eq:expected_total_size_of_fat_worms_after_Campbell_formula} can be lower bounded by
\begin{equation*}
	P_z \left( \trace (X[0, \mathcal{L}) ) = H \right) = 
	\sum_{\ell = 1}^{\infty} P_z \left( \trace(X[0, \ell)) = H \right) \cdot \eta(\ell) \geq
	\sum_{\ell = 1}^R P_z \left( \trace( X[0, \ell) ) = H  \right) \cdot \eta(\ell),
\end{equation*}
hence the expected total occupation measure can be further bounded from below by
\begin{equation}
	\label{eq:expected_total_size_of_fat_worms_after_restricting_number_of_step}
	\E \big[ \check{ \totalsize }_{A, \{ x \}}^R \big] \overset{\eqref{eq:general_Campbell_expectation_formula}}{\geq}
	\lambda \cdot \sum_{\ell = 1}^R \, \eta(\ell) \sum_{z \in V(\graph)} \sum_{H \in \latticeanimals_z} |H| \cdot \ind \left[
	\parbox{5em}{
		\footnotesize{
			$x \in H,\\ 
			H \cap \overline{A} = \{ x \}$}
	} \right] \cdot P_z \left( \trace (X[0, \ell)) = H \right).
\end{equation}
We also used here that the terms are non-negative, hence the sums can be freely reordered.

Since the innermost sum of \eqref{eq:expected_total_size_of_fat_worms_after_restricting_number_of_step} is just the expectation of the size of the trace of a random walk, under some extra constraints on its behaviour, the desired result will come once we can give a good enough lower bound on the sum	
\begin{align}
	\label{eq:expectation_of_size_of_fat_worms_with_restrictions}
	& \sum_{z \in V(\graph)} E_z \left[
	\left| \trace (X[0, \ell)) \right| \cdot \ind \left[  
	\parbox{10.5em}{
		$x \in \trace( X[0,\ell) );\\ 
		\trace( X[0, \ell) ) \cap \overline{A} = \{ x \}$
	}
	\right]
	\right]  \\ & \quad \overset{(*)}{\geq}
	\label{eq:expectation_of_size_of_fat_worms_that_bounce_back}
	\sum_{z \in V(\graph)} \sum_{t = 0}^{\ell -1} E_z \left[
	\left| \trace (X[0, \ell)) \right| \cdot \ind \left[  
	\parbox{15em}{
		$T_x(X) = t; \, \trace( X[0,t) ) \cap \overline{A} = \emptyset; \\ 
		\trace( X(t, \ell) ) \cap \overline{A} = \emptyset$
	} \right]
	\right]
\end{align}
Here, the inequality $(*)$ holds since we further assumed that the walker is such that it hits $x$ at some step $t$ (before $\ell$), but then never visits the set $\overline{A}$ again, throwing away the trajectories that visit $x$ more than once. The expectation in \eqref{eq:expectation_of_size_of_fat_worms_that_bounce_back} can be bounded further from below:
\begin{align}
	& E_z \left[
	\left| \trace (X[0, \ell)) \right| \cdot \ind \left[  
	\parbox{15em}{
		$T_x(X) = t; \, \trace( X[0,t) ) \cap \overline{A} = \emptyset; \nonumber \\ 
		\trace( X(t, \ell) ) \cap \overline{A} = \emptyset$
	} \right]
	\right]  \\ & \quad \overset{(**)}{\geq}
	E_z \left[
	\left| \trace (X[0, t)) \right| \cdot \ind \left[  
	\parbox{15em}{
		$T_x(X) = t; \, \trace( X[0,t) ) \cap \overline{A} = \emptyset; \nonumber \\ 
		\trace( X(t, \ell) ) \cap \overline{A} = \emptyset$
	} \right]
	\right]  \\ & \quad \overset{(\bullet)}{\geq}
	\label{eq:expectaiton_of_size_of_fat_worm_halfs_that_avoid_closure_of_A}
	P_x \left( T_{\overline{A}}^+ = + \infty \right) \cdot 
	E_z \left[
	\left| \trace (X[0, t)) \right| \cdot \ind \left[  
	\parbox{15em}{
		$T_x(X) = t; \, \trace( X[0,t) ) \cap \overline{A} = \emptyset$
	} \right]
	\right],
\end{align}
using in $(**)$ that $t \leq \ell - 1$, and the Markov property with $\{ T_{\overline{A}}^+(X) \geq \ell - t \} \supseteq \{ T_{\overline{A}}^+(X) = + \infty \}$ in $(\bullet)$.
Substituting this back into \eqref{eq:expectation_of_size_of_fat_worms_that_bounce_back} and using the reversibility of the simple random walk in the expectation of \eqref{eq:expectaiton_of_size_of_fat_worm_halfs_that_avoid_closure_of_A}, we arrive at the following lower bound on \eqref{eq:expectation_of_size_of_fat_worms_with_restrictions}:
\begin{align}
	& \sum_{z \in V(\graph)} E_z \left[
	\left| \trace (X[0, \ell)) \right| \cdot \ind \left[  
	\parbox{10.4em}{
		$x \in \trace( X[0,\ell) );\\ 
		\trace( X[0, \ell) ) \cap \overline{A} = \{ x \}$
	}
	\right]
	\right] \nonumber \\ 
	& \quad \geq
	P_x\left( T_{\overline{A}}^+ = + \infty \right) \cdot \sum_{t = 0}^{\ell-1} \sum_{z \in V(\graph)} E_x \left[ 	\left| \trace (X[0, t)) \right| \cdot \ind \left[ 
	\parbox{10em}{
		$\trace( X[0, t)) \cap \overline{A} = \emptyset; \\ X(t) = z$}
	\right] \right] \nonumber \\ 
	& \quad \overset{(\bullet \bullet)}{=} 
	\label{eq:expected_size_of_half_fat_worm_conditioned_on_escaping_A}
	P_x\left( T_{\overline{A}}^+ = + \infty \right) \cdot \sum_{t = 0}^{\ell-1} E_x \left[ 	\left| \trace (X[0, t)) \right| \cdot \ind \left[ \trace( X[0, t) ) \cap \overline{A} = \emptyset \right] \right],
\end{align}
where $(\bullet \bullet)$ follows from the law of total probability. 
The expectation on the right hand side can be further bounded from below using the constant $c(\veps)$ from Corollary \ref{coro:number_of_steps_and_size_of_trace_of_SRW_are_comparable_on_nonamenable_graphs} with any $\veps > 0$:
\begin{align}
	& E_x \left[ \left| \trace (X[0, t)) \right| \cdot \ind \left[ \trace( X[0, t) ) \cap \overline{A} = \emptyset \right] \right] \nonumber \\ 
	& \quad \geq
	E_x \left[ 	\left| \trace (X[0, t)) \right| \cdot \ind[ \dist_{\graph}(x, X(t)) \geq c(\veps) \cdot t ] \cdot \ind \left[ \trace( X[0, t) ) \cap \overline{A} = \emptyset \right] \right] \nonumber \\ 
	& \quad \overset{(\diamond)}{\geq}
	c(\veps) \cdot t \cdot E_x \left[ \ind[ \dist_{\graph}(x, X(t)) \geq c(\veps) \cdot t ] \cdot \ind \left[ \trace( X[0, t) ) \cap \overline{A} = \emptyset \right] \right] \nonumber \\ 
	& \quad \geq 
	c(\veps) \cdot t \cdot P_x \left( \dist_{\graph}(x, X(t)) \geq c(\veps) \cdot t, \, T_{\overline{A}}^+(X) = + \infty \right) \nonumber \\ 
	& \quad \overset{\eqref{eq:number_of_steps_and_size_of_trace_of_SRW_are_comparable_on_nonamenable_graphs}}{\geq}
	\label{eq:expectation_of_size_of_half_worm_escaping_A_dealt_with_using_positive_speed}
	c(\veps) \cdot t \cdot \left( P_x \left( T_{\overline{A}}^+ = + \infty \right) - \veps \right).
\end{align}
Here $(\diamond)$ is due to the observation that the number of sites visited up to some step by a walk is always lower bounded by the distance between its starting point and its actual position.

If we substitue the lower bound \eqref{eq:expectation_of_size_of_half_worm_escaping_A_dealt_with_using_positive_speed} back into \eqref{eq:expected_size_of_half_fat_worm_conditioned_on_escaping_A} and that back into \eqref{eq:expected_total_size_of_fat_worms_after_restricting_number_of_step}, we obtain
\begin{align*}
	\E \big[ \check{ \totalsize }_{A, \{ x \}}^R \big] & \geq 
	c(\veps) \cdot \lambda \cdot P_x \left( T_{\overline{A}}^+ = + \infty \right) 
	\left( P_x \left( T_{\overline{A}}^+ = + \infty \right) - \veps \right)
	\cdot \sum_{\ell = 1}^{R} \, \mu(\ell) \cdot \sum_{t = 0}^{\ell - 1} \, t \\ & \geq 
	c(\veps) \cdot \lambda \cdot P_x \left( T_{\overline{A}}^+ = + \infty \right) 
	\cdot \left( P_x \left( T_{\overline{A}}^+ = +\infty \right) - \veps \right)
	\cdot \E \left[ \mathcal{L}^2 \ind \left[ \mathcal{L} \leq R \right] \right],
\end{align*}
which is exactly what we wanted to prove.
\end{proof}

We will now sum \eqref{eq:lower_bound_on_the_expected_total_size_of_worms_hitting_one_vertex} over the vertices $x$ in some $B \subseteq \extboundary A$ using Corollary~\ref{coro:lower_bound_on_partial_square_term_capacity}, then, assuming unimodularity, will get rid of the multiplicities using Corollary~\ref{coro:expected_size_without_and_with_multiplicity}.

\begin{corollary}
\label{coro:lower_bound_on_expected_size_of_fat_worms_on_nonamenable_graphs}	
Consider a nonamenable graph $\graph$ with a subgroup of its automorphisms $\grp \subgrp \Aut(\graph)$ that acts transitively and unimodularly on $\graph$. Let $\eta \, : \, \N_{>0} \rightarrow [0,1]$ be a probability mass function and $\animals = \sum_{i \in I} \delta_{(x,H)_i} \sim \mathcal{P}_{\eta}^\lambda$.
Assume also that we are given some $R \in \N_{>0}$, a finite connected set $A \fsubset V(\graph)$ and a subset of its exterior boundary $B \subseteq \extboundary A$.
There exists a constant $c(\graph) > 0$ depending only on the graph such that, for the restricted process $\check{ \animals }_{A,B}^R$ defined in Definition \ref{def:point_measure_restricted_to_hit_B_but_not_the_closure_of_A}, we have
\begin{equation}
	\label{eq:lower_bound_on_expected_size_of_fat_worms_on_nonamenable_graphs}
	\E \left[ \left| \trace( \check{ \animals }_{A,B}^R ) \right| \right] \geq
	c(\graph) \cdot \tfrac{ \lambda \cdot \E \left[ \mathcal{L}^2 \ind \left[ \mathcal{L} \leq R \right] \right] }{ \lambda \cdot m_1 + 1 } \cdot 
	\left[  
	\tfrac{ (1-\spectral(\graph))^2 \cdot (1 + \Cheeger (\graph))}{2} \cdot |A| - \left( | \extboundary A | - |B| \right)
	\right],
\end{equation}
where $\mathcal{L}$ denotes a random variable that has law $\eta$.
\end{corollary}

\begin{proof}
By Corollary \ref{coro:expected_size_without_and_with_multiplicity}, it is enough to prove an appropriate lower bound on the total restricted occupation measure $\check{ \totalsize }_{A,B}^R$ introduced in~\eqref{def:point_measure_restricted_to_hit_B_but_not_the_closure_of_A_total_local_size}. This can be further bounded by the decomposition
\begin{equation}
	\label{eq:lower_bound_on_totalsize_using_decomposition}
	\check{ \totalsize }_{A,B}^R \geq 
	\sum_{x \in B} \check{ \totalsize }_{A, \{ x \}}^R.
\end{equation}
Indeed, on the right hand side we divided the worms of the restricted zoo into separate (and independent) bundles hitting only one of the vertices of $B$, but not intersecting $\overline{A}$ otherwise. 	
Evidently, this is just a lower bound since all the worms that hit $B$ at multiple vertices are thrown away. So, Lemma~\ref{lemma:lower_bound_on_the_expected_total_size_of_worms_hitting_one_vertex} yields that for every $\veps > 0$ there exists some $c(\veps) > 0$ such that
\begin{equation}\label{eq:lower_bound_on_expected_size_of_fat_worms_on_nonamenable_graphs_with_epsilon}
	\begin{aligned}
		&
		\E \left[ \left| \trace \left( \check{ \animals }_{A,B}^R \right) \right| \right] \overset{\eqref{eq:expected_size_without_and_with_multiplicity}}{\geq}
		\frac{1}{\lambda \cdot m_1 + 1} \cdot \E \left[ \check{ \totalsize }_{A,B}^R \right] \overset{\eqref{eq:lower_bound_on_totalsize_using_decomposition}}{\geq}
		\frac{1}{\lambda \cdot m_1 + 1} \cdot \sum_{x \in B} \E \left[ \check{ \totalsize }_{A,\{x\}}^R \right] \\ 
		& \qquad \overset{\eqref{eq:lower_bound_on_the_expected_total_size_of_worms_hitting_one_vertex}}{\geq}
		c(\veps) \cdot \frac{ \lambda \cdot \E \left[ \mathcal{L}^2 \ind \left[ \mathcal{L} \leq R \right] \right] }{ \lambda \cdot m_1 + 1 } \cdot \sum_{x \in B} P_x \left( T_{\overline{A}}^+ = + \infty \right) \cdot \left( P_x \left( T_{\overline{A}}^+ = + \infty \right) - \veps \right).	
	\end{aligned}
\end{equation}
However, since 
\begin{equation*}
	\sum_{x \in B} P_x \left( T_{\overline{A}}^+ = + \infty \right) \leq \capacity \left( \overline{A} \right) \overset{\eqref{eq:capacity_bound}}{\leq}
	\left| \overline{A} \right|
\end{equation*}
holds, together with Corollary \ref{coro:lower_bound_on_partial_square_term_capacity} we have
\begin{equation} \label{eq:lower_bound_on_partial_square_term_capacity_with_epsilon}
	\begin{aligned}
		& \sum_{x \in B} P_x \left( T_{\overline{A}}^+ = +\infty \right) \cdot \left( P_x \left( T_{\overline{A}}^+ = +\infty \right) - \veps \right) \\ 
		& \qquad  \overset{\eqref{eq:lower_bound_on_partial_square_term_capacity}}{\geq}
		(1 - \spectral(\graph))^2 \cdot \left| \overline{A} \right| - ( | \extboundary A| - |B|) - \veps \left| \overline{A} \right| \\ 
		& \qquad \hskip 0.23 cm \geq
		\left( (1 - \spectral(\graph))^2 - \veps \right) \cdot (1 + \Cheeger(\graph))\cdot \left| A \right| - ( | \extboundary A| - |B|).
	\end{aligned}
\end{equation}
Substituting  \eqref{eq:lower_bound_on_partial_square_term_capacity_with_epsilon} back into \eqref{eq:lower_bound_on_expected_size_of_fat_worms_on_nonamenable_graphs_with_epsilon} and choosing $\veps = (1 - \spectral(\graph))^2 \slash 2$ we arrive to the desired lower bound.
\end{proof}

As it is evident from the proof, the appearance of the strange expression on the right-hand side of \eqref{eq:lower_bound_on_expected_size_of_fat_worms_on_nonamenable_graphs} is due to the fact that a vertex $x \in B \subseteq \extboundary A$ is weighted by the corresponding escape probability, which depends on the shape of $A$.
Fortunately, in the exploration we will achieve that the fattening set is larger than a big multiplier of the already explored set, so that the relevant subset of the exterior boundary will always be so large that the exact shape will not matter.

Finally, here is our upper bound on the second moment of the size of the trace. It does not use unimodularity and works for any animal measure, not just worms.

\begin{lemma}[Second moment upper bound for general zoo]
\label{lemma:upper_bound_on_the_second_moment_of_the_occupied_sets_size_for_worms}
Let us consider a graph $\graph$ with a transitive  subgroup of its automorphisms $\grp \subgrp \Aut(\graph)$.
Let $\nu_{\origin}$ be a $\grp$-rotation invariant measure on $\latticeanimals_{\origin}$, then $\nu$ the measure on $\zoospace$ induced from $\nu_{\origin}$ via \eqref{def:induced_measure_on_zoospace}, and $\animals = \sum_{i \in I} \delta_{(x,H)_i} \sim \mathcal{Q}_{\nu}^\lambda$.
Given some $R \in \N$ and an event $A^R \in \zooalgebra$ for which $A^R \subseteq \zoospace^R(V(\graph))$ holds, let us consider the restricted zoo $\bar{\animals} := \animals \ind \left[ A^R \right]$ introduced in~\eqref{eq:set_hitting_animals_of_volume_at_most_R}. Then we have
\begin{equation}
	\label{eq:upper_bound_on_the_second_moment_of_the_occupied_sets_size_for_worms}
	\E \left[ \left| \trace \left( \bar{ \animals } \right) \right|^2 \right] \leq 
	\left( \left| \ball(\origin, R) \right| + 1 \right) \cdot \E\left[ \left| \trace \left( \bar{ \animals } \right) \right| \right] + 
	\E^2 \left[ \left| \trace \left( \bar{ \animals } \right) \right| \right],
\end{equation}
from which it obviously follows that 
\begin{equation}
	\label{eq:upper_bound_on_the_variance_of_the_occupied_sets_size_for_worms}
	\Var \left( \left| \trace \left( \bar{\animals} \right) \right| \right) \leq 
	\left( \left| \ball(\origin, R) \right| + 1 \right) \cdot \E\left[ \left| \trace \left( \bar{ \animals } \right) \right| \right].
\end{equation}
\end{lemma}

\begin{proof}
To prove~\eqref{eq:upper_bound_on_the_second_moment_of_the_occupied_sets_size_for_worms}, recall the notation $\bar{\LTmeasure}(\cdot)$ and $\bar{\totalsize}$ from~\eqref{def:point_measure_restricted_to_hit_B_but_not_the_closure_of_A_total_local_size} corresponding to our present $\bar{ \animals }$, and note that 	
\begin{equation}
	\left| \trace \left( \bar{\animals} \right) \right| = \sum_{x \in V(\graph)} \ind \left[ \bar{ \LTmeasure }(x) > 0 \right],
\end{equation}
and hence
\begin{equation*}
	\left| \trace \left( \bar{\animals} \right) \right|^2 = 
	\sum_{x \in V(\graph)} \ind \left[ \bar{ \LTmeasure }(x) > 0 \right] + 
	\sum_{x \in V(\graph)} \sum_{y \in V(\graph) \setminus \{ x \}} \ind \left[ \bar{ \LTmeasure }(x) > 0, \, \bar{ \LTmeasure }(y) > 0 \right].
\end{equation*}
As a consequence, we obtain
\begin{equation}
	\label{eq:upper_bound_on_second_moment}
	\E \left[ \left| \trace \left( \bar{ \animals } \right) \right|^2 \right] \leq
	\E \left[ \left| \trace \left( \bar{ \animals } \right) \right| \right] +
	\sum_{x \in V(\graph)} \sum_{y \in V(\graph)} \prob \left( \bar{ \LTmeasure }(x) > 0, \, \bar{ \LTmeasure }(y) > 0 \right),
\end{equation}
so, to prove the result, we will now need to concentrate on the second term.

Let us introduce some more notation.
Recalling \eqref{eq:set_hitting_animals} and \eqref{def:restricted_zoo}, for any given vertex $x \in V(\graph)$ let us consider the following decomposition of $\bar{\animals}$:
\begin{align}
	\bar{ \animals } = 
	\bar{ \animals }^x + \bar{ \animals }^{ \neg x } :=
	\bar{ \animals } \ind \left[ \zoospace( \{ x \} ) \right] +
	\bar{ \animals } \ind \left[ \zoospace \setminus \zoospace( \{ x \} ) \right],
\end{align}
and the corresponding decomposition of the occupation measure 	
\begin{equation}
	\bar{ \LTmeasure } = \bar{ \LTmeasure }^{x} + \bar{ \LTmeasure }^{ \neg x }.
\end{equation}
In words, we separate the animals depending on if they are hitting the vertex $x$ or not.
We know from \eqref{PPP_prop:independence} of Theorem \ref{thm:basic_properties_of_PPPs} that these terms are independent.

Now, for a given $x \in V(\graph)$, we can further decompose the inner summation in~\eqref{eq:upper_bound_on_second_moment} as
\begin{align}
	& \sum_{y \in V(\graph)} \prob \left( \bar{ \LTmeasure }(x) > 0, \, \bar{ \LTmeasure }(y) > 0 \right) =
	\sum_{y \in V(\graph)} \prob \left( \bar{ \LTmeasure }^{x} (x) > 0, \, \bar{ \LTmeasure }^{x}(y) + \bar{ \LTmeasure }^{ \neg x }(y) > 0 \right) \nonumber \\ 
	& \qquad \overset{(*)}{\leq}
	\sum_{y \in V(\graph)} \prob \left( \bar{ \LTmeasure }^{x} (x) > 0, \, \bar{ \LTmeasure }^{ x }(y) > 0 \right) +
	\sum_{y \in V(\graph)} \prob \left( \bar{ \LTmeasure }^{x} (x) > 0, \, \bar{ \LTmeasure }^{ \neg x }(y) > 0 \right) \nonumber \\ 
	& \qquad \overset{(**)}{=}
	\sum_{y \in V(\graph)} \prob \left( \bar{ \LTmeasure }^{x} (x) > 0, \, \bar{ \LTmeasure }^{ x }(y) > 0 \right) +
	\prob \left( \bar{ \LTmeasure }^{x} (x) > 0 \right) \cdot \sum_{y \in V(\graph)} \prob \left( \bar{ \LTmeasure }^{ \neg x }(y) > 0 \right) \nonumber \\ 
	& \qquad \leq
	\label{eq:upper_bound_on_simultaneous_hitting_by_the_whole_set}
	\sum_{y \in V(\graph)} \prob \left( \bar{ \LTmeasure }^{x} (x) > 0, \, \bar{ \LTmeasure }^{ x }(y) > 0 \right) +
	\prob \left( \bar{ \LTmeasure } (x) > 0 \right) \cdot \sum_{y \in V(\graph)} \prob \left( \bar{ \LTmeasure } (y) > 0 \right).
\end{align}
Here, $(*)$ follows from the union bound, meanwhile $(**)$ is due to \eqref{PPP_prop:independence} of Theorem \ref{thm:basic_properties_of_PPPs} (as the two occupation measures are restricted to disjoint sets).

The second term on the right hand side of  \eqref{eq:upper_bound_on_simultaneous_hitting_by_the_whole_set}, after summing over $x\in V(\graph)$, is just $\E^2 \left[ \left| \trace \left( \bar{ \animals } \right) \right| \right]$, which is the last term in \eqref{eq:upper_bound_on_the_second_moment_of_the_occupied_sets_size_for_worms}. The first term in \eqref{eq:upper_bound_on_simultaneous_hitting_by_the_whole_set} can be rewritten using the law of total probability as follows:
\begin{align}
	\label{eq:upper_bound_on_simultaneous_hitting_restricted_to_x}
	\sum_{y \in V(\graph)} \prob \left( \bar{ \LTmeasure }^{x} (x) > 0, \, \bar{ \LTmeasure }^{ x }(y) > 0 \right) &=
	\prob \left( \bar{ \LTmeasure }^{x} (x) > 0 \right) \,
	\sum_{y \in V(\graph)} \prob \left. \left( \bar{ \LTmeasure }^{ x }(y) > 0 \, \right| \, \bar{ \LTmeasure }^{x} (x) > 0 \right),
\end{align}
and
\begin{align}
	\label{eq:silly_upper_bound_on_sum_of_conditional_hitting}
	\sum_{y \in V} \prob \left. \left( \bar{ \mu }^x(y) > 0 \, \right| \, \bar{ \mu }^x(x) > 0 \right) \overset{(*)}{=}
	\sum_{y \in \ball(x, R)} \prob \left. \left( \bar{ \mu }^x(y) > 0 \, \right| \, \bar{ \mu }^x(x) > 0 \right) \overset{(**)}{\leq}
	\left| \ball(\origin, R) \right|.
\end{align}
In $(*)$ we used that since the restriction is given by $A^R \subseteq \zoospace^R(V(\graph))$ and hitting $x$, only finitely many terms can be positive and these non-zero terms must be in a ball of radius $R$ around $x$, and have $(**)$ since the graph is transitive.

Now, putting \eqref{eq:silly_upper_bound_on_sum_of_conditional_hitting} and \eqref{eq:upper_bound_on_simultaneous_hitting_restricted_to_x} back into \eqref{eq:upper_bound_on_simultaneous_hitting_by_the_whole_set}, then summing over $x \in V(\graph)$ as in \eqref{eq:upper_bound_on_second_moment}, we obtain the desired upper bound~\eqref{eq:upper_bound_on_the_second_moment_of_the_occupied_sets_size_for_worms}.
\end{proof}

\section{The exploration processes}
\label{section:exploration_processes}

In this section, building on the results of Section~\ref{section:bounds_on_fat}, we will use two different exploration processes to prove Theorems~\ref{t.free} and~\ref{t.worms} (more precisely, Theorem~\ref{thm:random_length_worms_model_on_nonamenable_graphs}). The argument for worms on any nonamenable unimodular transitive graph will be more general than the branching process argument for free products, and, in principle, it could also be used for that purpose. However, handling the parity issue (Figure~\ref{f.free2}) within that more general exploration process would have made the argument much harder to follow, hence we have decided to keep both approaches, instead of trying to unify the proofs. 

\subsection{A supercritical branching process on free products}\label{ss.explofree}

In this subsection, we will work under the circumstances of Theorem~\ref{t.free}: we are given a nonamenable free product $\graph = \graph_1 \star \graph_2$, a transitive and unimodular subgroup $\grp \subgrp \Aut(\graph)$, a $\grp$-rotation invariant measure $\nu_{\origin}$ on $\latticeanimals_{\origin}$, the measure $\nu$ on $\zoospace$ induced from $\nu_{\origin}$ via \eqref{def:induced_measure_on_zoospace}, and the corresponding Poisson point process $\animals = \sum_{i \in I} \delta_{(x,H)_i} \sim \mathcal{Q}_{\nu}^\lambda$, with some fixed intensity $\lambda > 0$. 

Recall also the proof strategy described in Subsection~\ref{subsection:poisson_zoo_on_free_products}: using the fact that every vertex is a cutpoint of the graph, Claim~\ref{claim:properties_of_free_product_of_graphs}, we are going to explore a subset of the cluster of a fixed vertex in the form of a supercritical branching process.

\subsubsection{Some preparation}

Before we can describe the exploration branching process, we need to introduce some more notation.
They all serve the purpose of dealing with parity issue of Figure~\ref{f.free2}.

Looking back at Corollary \ref{coro:lower_bound_on_expected_size_of_fat_of_animals_through_one_vertex_on_free_product}, recall that for a fixed $\nu$ and $R$ it can happen that there is only one factor of the free product in whose direction it is useful to continue the exploration. However, we are not necessarily provided with this direction from a vertex on the exterior boundary of an arbitrary connected set. In such a bad case, we should take one extra step further in the bad factor to get to the good one. 
This means a mapping from the vertices of the exterior boundary of some set:

\begin{definition}
\label{def:mapping_into_the_good_direction}
Let $A \fsubset V(\graph)$ connected, $x \in \extboundary A$ and $y \in \mathcal{N}_1(x)$ and $z \in \mathcal{N}_2(x)$ fixed in an arbitrary way. Then
\begin{equation}
	\label{eq:mapping_into_the_good_direction}
	\Psi_1(x,A) := 
	\begin{cases}
		x, & \text{if} \quad \mathcal{N}_1(x) \cap \overline{A} = \emptyset, \\
		z, & \text{otherwise}, 
	\end{cases}
	\quad \text{and} \quad
	\Psi_2(x,A) := 
	\begin{cases}
		x, & \text{if} \quad \mathcal{N}_2(x) \cap \overline{A} = \emptyset, \\
		y, & \text{otherwise}.
	\end{cases}
\end{equation}
These can be also extended to any subset $B \subseteq \extboundary A$ of the exterior boundary as 
\begin{equation}
	\label{eq:mapping_into_the_good_direction_for_sets}
	\Psi_1(B, A) := \big\{ \Psi_1(x,A) : x \in B \big\}
	\quad \text{and} \quad
	\Psi_2(B, A) := \big\{ \Psi_2(x,A) : x \in B \big\}.
\end{equation}
\end{definition}

\noindent
It is important to note that although in the second cases of \eqref{eq:mapping_into_the_good_direction} the vertex $\Psi_i(x,A)$ is chosen arbitrarily, by the structure of the free product we always have $|\Psi_i(B,A)| = |B|$.

Later in the explicit description of the exploration, we would also like to get back to the vertex $x \in \extboundary A$ from $\Psi_i(x,A)$. This inverse will be denoted by $\Psi_i^{-1}(\cdot,A)$.

The statement of Corollary \ref{coro:lower_bound_on_expected_size_of_fat_of_animals_through_one_vertex_on_free_product} does not say anything about the status of the vertex whose neighbourhood was considered, which would leave holes in the explored set.
So, to arrive at a connected set after using the corollary, we need to have some animal visiting that vertex.
A weak but still satisfactory way to achieve this is to consider only animals that are rooted here:

\begin{definition}
\label{def:gluing_of_fat}
Let $A \fsubset V(\graph)$ connected, $x \in \extboundary A$, and $\animals = \sum_{i \in I} \delta_{(x,H)_i}$ a point measure. Then, for $i = 1,2$,
\begin{equation}
	\label{eq:gluing_of_fat}
	\Phi_i(x, A, \animals) := 
	\begin{cases}
		\Psi_i(x,A), & \text{if} \quad \forall \, y \in \{ x, \Psi_i(x,A) \}, \, \exists \, H \in \latticeanimals_y: \, \animals((y,H)) > 0; \\
		\emptyset, & \text{otherwise}.
	\end{cases}
\end{equation}
Again, this can be extended to any $B \subseteq \extboundary A$ as
\begin{equation}
	\label{eq:gluing_of_fat_for_sets}
	\Phi_i(B,A,\animals) := \big\{ \Phi_i(x,A,\animals) : {x \in B} \big\}.
\end{equation}
\end{definition}

Obviously, the previous notion is not that useful when the point process $\animals$ is exactly the same as the one we will use for the fattening, since that would introduce some extra nontrivial dependencies.
However, using the colouring property from Theorem \ref{thm:basic_properties_of_PPPs} we can use independent processes for the different purposes. For $\animals = \sum_{i \in I} \delta_{(x,H)_i}$, take i.i.d.~random variables $\{ \xi_i \}_{i \in I}$ with $\xi_i \sim \Bernoulli(1/2)$. Then, by Theorem \ref{thm:basic_properties_of_PPPs}~\eqref{PPP_prop:colouring}, in the decomposition
\begin{equation}
\label{def:sprinkling_and_branching_parts_of_animals}
\animals= {}^s \animals + {}^b \animals \quad\text{with}\quad
{}^s \animals := \sum_{i \in I} \xi_i \delta_{(x,H)_i} , \quad 
{}^b \animals := \sum_{i \in I} (1 - \xi_i) \delta_{(x,H)_i},
\end{equation}
we have that ${}^s \animals$ and ${}^b \animals$ are independent, with law $\mathcal{Q}_{\nu}^{\lambda/2}$.
Let us note here that the notation comes from the fact that ${}^b \animals$ will be used in the exploration (and hence the branching), meanwhile ${}^s \animals$ has the role of filling the holes left by the exploration (i.e., sprinkling).

\subsubsection{The exploration and the branching}

We now have everything to describe the exploration and prove that it will survive forever with positive probability. As we would like to use branching process domination for this, we will use the appropriate terminology throughout.

Recall from Corollary \ref{coro:lower_bound_on_expected_size_of_fat_of_animals_through_one_vertex_on_free_product} that for given $\nu$ and $R$ there is always a factor of the free product $\graph$ in whose direction the expected size of the set of all animals approaching any vertex can be lower bounded by a multiple of the truncated second moment $m_2^R$. Although we will take $R$ large enough only later, we can assume without loss of generality that for that $R$ this good direction will be given by $\graph_1$.

As the initial ancestor, we have the origin $\origin$. In the first step, we explore the random set
\begin{equation}
\label{def:first_step_of_exploration_branching}
E_1 := A_1 \cup B_1 \cup \big\{ \Psi_1^{-1}(x,A_1) : x \in B_1 \big\},
\end{equation}
where the subsets are defined as
\begin{equation}
\label{def:first_step_of_exploration_branching_A1_B1}
A_1 := \trace \left( {}^b \animals \ind \left[ \zoospace^R ( \{ \origin \} ) \right] \right)
\quad \text{and} \quad
B_1 := \Phi_1 \left( \extboundary A_1, A_1, {}^s \animals  \right).
\end{equation}
That is, $A_1$ is the set covered by all the animals of volume at most $R$ of ${}^b \animals$ that ever visit $\origin$, while $B_1$ is a collection of vertices outside of $A_1$ but at most distance $2$ from it, which have all their $\graph_1$ edges going further away, and moreover, they have at least one animal growing from every vertex on the path going to them from $A_1$.
Obviously, adding the set of $\Psi_1^{-1}(x,A_1)$'s for all $x \in B_1$ is needed to make the explored set connected.

As we reason below, the vertices of $B_1$ can be considered as the children of the origin.
Thus, it is worth examining its expected size, controlled by the truncated second moment.

\begin{claim}
\label{claim:lower_bound_on_reproduction_rate_branching_process_first_generation}
Using the notation introduced above, we have
\begin{equation}
	\label{eq:lower_bound_on_reproduction_rate_branching_process_first_generation}
	\E \left[ \left| B_1 \right| \right] \geq
	\frac{ \left(1-e^{-\lambda/2}\right)^2 \cdot \Cheeger(\graph) \cdot \lambda}{ \lambda \cdot m_1^R + 2 } \cdot m_2^R.
\end{equation}
\end{claim}

\begin{proof}
Using Corollary \ref{coro:expected_size_without_and_with_multiplicity} on neglecting multiplicities and 
Lemma~\ref{lemma:size_biasing_for_Poisson_zoo} about size-biasing, 
\begin{equation}\label{eq:lower_bound_on_expected_size_A1}
	\begin{aligned}		
		\E \left[ \left| A_1 \right| \right] & \overset{\eqref{def:first_step_of_exploration_branching_A1_B1}}{=}
		\E  \left| \trace \left( {}^b \animals \ind \left[ \zoospace^R( \{ \origin \} ) \right] \right) \right|
		\\ & \overset{\eqref{eq:expected_size_without_and_with_multiplicity}}{\geq} 
		\frac{1}{(\lambda \slash 2) \cdot m^R_1 + 1} \cdot
		\E \left[ \totalsize^{ {}^b \animals \ind \left[ \zoospace^R( \{ \origin \} ) \right] } \right]
		\overset{\eqref{eq:size_biasing_for_Poisson_zoo}}{=}
		\frac{2 \cdot \lambda \cdot m_2^R}{2\cdot (\lambda \cdot m^R_1 + 2)},
	\end{aligned}
\end{equation}
noting that ${}^b \animals \sim \mathcal{Q}_{\nu}^{\lambda/2}$. Then,
\begin{align*}
	\E  \left[ \left| B_1 \right| \right] & \overset{\eqref{def:first_step_of_exploration_branching_A1_B1}}{=}
	\E  \left[ \left| \Phi_1 \left( \extboundary A_1, A_1, {}^s \animals \right) \right| \right] \overset{\eqref{eq:gluing_of_fat_for_sets}}{=} 
	\E \left[ \sum_{x \in \extboundary A_1} \Phi_1 \left( x, A_1, {}^s \animals \right) \right] \\ & \overset{\eqref{eq:gluing_of_fat}}{=}
	\E \left[ 
	\sum_{x \in \extboundary A_1} \prob \left. \left( 
	\parbox{21.5em}{
		$\exists \, H \in \latticeanimals_x \, : \, {}^s \animals( (x, H) ) > 0; \\ 
		\exists \, H' \in \latticeanimals_{\Psi_1(x, A_1)} \, : \, {}^s \animals( (\Psi_1(x, A_1), H') ) > 0$
	}
	\, \right| \,
	A_1							
	\right)
	\right] \\ & \overset{(*)}{\geq}
	\E \left[ \sum_{x \in \extboundary A_1} 
	\prob \big( \exists \, H \in \mathcal{H}_x \, : \, {}^s \animals ((x,H)) > 0 \big)^2
	\right] \overset{(**)}{=}
	\left(1-e^{-\lambda/2}\right)^2 \E \left[ \left| \extboundary A_1 \right| \right] \\ & \overset{\eqref{def:Cheeger_constant}}{\geq}
	\left(1-e^{-\lambda/2}\right)^2 \cdot \Cheeger(\graph) \cdot \E \left[ \left| A_1 \right| \right] \overset{\eqref{eq:lower_bound_on_expected_size_A1}}{\geq}
	\frac{ \left(1-e^{-\lambda/2}\right)^2 \cdot \Cheeger(\graph) \cdot \lambda}{\lambda \cdot m_1^R + 2 } \cdot m_2^R.
\end{align*}
Here, in $(*)$ we used the independence of ${}^s \animals$ from ${}^b \animals$ and the invariance of the law of the former under the action of $\grp$, meanwhile $(**)$ is due to the fact that ${}^s \animals \sim \mathcal{Q}_{\nu}^{
	\lambda/2}$.
	\end{proof}
	
	For any vertex $x \in V(\graph)$, let us denote the connected subgraph of $\graph$ by $\graph^x$ that is the component containing $x$ after erasing all the edges $\{ z,x \}$, $z \in \mathcal{N}_2(x)$. 
	Due to the structure of the free product, the subgraphs $\graph^x$, $x \in B_1$, are disjoint.
	Moreover, the way we defined $E_1$, we also know that no animals of ${}^b \animals$ rooted in $\graph^x$ have been used yet.
	Hence, by \eqref{PPP_prop:independence} of Theorem \ref{thm:basic_properties_of_PPPs}, we can continue the exploration through the vertices of $B_1$ independently of each other, so one can indeed think of the vertices of $B_1$ as the children of $\origin$.
	
	In order to define a general step in the exploration process, let us introduce the following notions for a vertex $x \in V(\graph)$:
	\begin{equation}
\label{def:reproduction_in_the_branching}
\alpha_x := \trace \left( {}^b \check{ \animals }_{x, \mathcal{N}_1(x)}^R \right);
\quad
\hat{\beta}_x := \extboundary \alpha_x \setminus \ball(x, 1);
\quad
\beta_x := \Phi_1 \left( \hat{ \beta }_x, \alpha_x, {}^s \animals \right).
\end{equation}
In words: $\alpha_x$ is the set of all animals that hit at least one vertex of $\mathcal{N}_1(x)$, but not $x$; $\hat{ \beta }_x$ is the exterior boundary except the vertices of $\mathcal{N}_1(x)$ that we have already considered; and $\beta_x$ is the mapping of the considered exterior vertices into good vertices. 
Note here that is indeed enough to exclude the vertices of $\ball(x,1)$ in $\hat{\beta}_x$ as the corresponding vertices of $\beta_x$ will always open up not yet explored territories due to Claim \ref{claim:properties_of_free_product_of_graphs}.
Moreover, one immediate consequence of this construction and Corollary \ref{coro:lower_bound_on_expected_size_of_fat_of_animals_through_one_vertex_on_free_product} is the following.

\begin{claim}
\label{claim:lower_bound_on_reproduction_rate_branching_process}
Under the assumptions as above, for any $x \in V(\graph)$ we have
\begin{equation}
	\label{eq:lower_bound_on_reproduction_rate_branching_process}
	\E \left[ \left| \beta_x \right| \right] \geq 
	\left(1-e^{-\lambda/2}\right)^2  \cdot \left( \frac{ \lambda \cdot \Cheeger(\graph)^2 \cdot m_2^R}{ 2 \cdot ( \lambda \cdot m_1^R + 2) } - (\degree + 1)\right).
\end{equation}
\end{claim}

\begin{proof}
The proof is similar to that of Claim \ref{claim:lower_bound_on_reproduction_rate_branching_process_first_generation}:
\begin{align*}
	\E  \left[ \left| \beta_x \right| \right] & \overset{(*)}{=}
	\E \left[ \sum_{y \in \hat{ \beta }_x} \prob \left. \left( 
	\parbox{21em}{
		$\exists H \in \latticeanimals_y \, : \, {}^s \animals( (y, H) ) > 0; \\ 
		\exists H' \in \latticeanimals_{\Psi_1(y,\alpha_x)} \, : \, {}^s \animals( (\Psi_1(y,\alpha_x), H') ) > 0$
	}
	\, \right| \, 
	\alpha_x, \, \hat{ \beta }_x
	\right) \right] \\ & \overset{(**)}{\geq}
	\E \left[ \sum_{y \in \hat{ \beta }_x} 
	\prob^2 \left( \exists \, H \in \latticeanimals_y \, : \; {}^s \animals \left( (y,H) \right) > 0 \right)
	\right] \overset{(\bullet)}{=}
	\left(1-e^{-\lambda/2}\right)^2 \cdot \E\left[ \left| \hat{ \beta }_x \right| \right] \\ & \overset{\eqref{def:reproduction_in_the_branching}}{=}
	\left(1-e^{-\lambda/2}\right)^2  \cdot \E \left[ \left| \extboundary \alpha_x \setminus \ball(x,1)  \right| \right] \\ & \overset{\eqref{def:Cheeger_constant}}{\geq}
	\left(1-e^{-\lambda/2}\right)^2  \cdot \big( \Cheeger(\graph) \cdot \E \left[ \left| \alpha_x \right| \right] - (\degree + 1) \big) \\ & \overset{(\bullet \bullet)}{\geq}
	\left(1-e^{-\lambda/2}\right)^2  \cdot \left( \Cheeger(\graph) \cdot \frac{ \lambda \cdot \Cheeger(\graph) \cdot m_2^R}{ 2 \cdot ( \lambda \cdot m_1^R + 2) } - (\degree + 1)\right).
\end{align*}

Here in $(*)$ we used the definitions of $\beta_x$ from 
\eqref{def:reproduction_in_the_branching} and $\Phi_1$ from \eqref{eq:gluing_of_fat},
$(**)$ is due to the independence of ${}^s \animals$ and ${}^b \animals$ and the invariance of the former under the action of $\grp$,
$(\bullet)$ holds since ${}^s \animals \sim \mathcal{Q}_{\nu}^{\lambda/2}$, and
$(\bullet \bullet)$ follows from \eqref{eq:lower_bound_on_expected_size_of_fat_of_animals_through_one_vertex_on_free_product} noting that ${}^b \animals \sim \mathcal{Q}_{\nu}^{\lambda/2}$.
\end{proof}

Now we are ready to define the explored set, and hence the exploration process for a general $n > 1$ as
\begin{equation}
\label{def:general_step_of_exploration_branching}
E_n := E_{n-1} \cup A_n \cup B_n \cup \big\{ \Psi_1^{-1}(x,A_n) : x \in B_n \big\},
\end{equation}
where
\begin{equation}
\label{def:exploration_branching_An_Bn_sets}
A_n := \bigcup_{x \in B_{n-1}}  \alpha_x 
\qquad \text{and} \qquad
B_n :=\bigcup_{x \in B_{n-1}}  \beta_x .
\end{equation}
According to the structure of the free product and our construction,  the set $E_n$ is connected.
Furthermore, the unions of \eqref{def:exploration_branching_An_Bn_sets} are disjoint and hence one can think of the vertices in $\beta_x$ as the children of the vertex $x \in B_{n-1}$, in line with the previous terminology.

Obviously, the sequence $\{ E_n \}_{n \geq 1}$ of explored sets is monotone increasing, hence it is meaningful to define $E_{\infty} := \lim_{n \to \infty} E_n$, which, for any $R$, is a subset of the cluster of the origin in the true zoo.
Consequently, to prove that this cluster is indeed infinite, we only need to show that $|E_{\infty}| = \infty$ holds. As the following statement shows, choosing a large enough $R$ we have such indefinite run with positive probability.

\begin{proposition}
\label{prop:the_exploration_runs_indefinitely_on_free_product}
Using the notation introduced above, if we choose $R$ to be such that
\begin{equation}
	\begin{aligned}
		\label{cond:large_enough_R_for_free_products}
		m_2^R 
		&>
		\left( \left(1-e^{-\lambda/2} \right)^{-2} + (d+1) \right) \frac{ 2 \cdot ( \lambda \cdot m_1 + 2)  }{  \lambda \cdot \Cheeger(\graph)^2 } \\
		&= \frac{16 + o(1)}{\Cheeger(\graph)^2} \cdot \frac{1}{\lambda^3}, \quad \textrm{ as }\lambda\to 0,
	\end{aligned}
\end{equation}

holds, then the exploration process runs indefinitely with positive probability:
\begin{equation}
	\label{eq:the_exploration_runs_indefinitely_on_free_product}
	\prob \left( | E_{\infty} | = +\infty \right) > 0.
\end{equation}
\end{proposition}

\begin{proof}
Without loss of generality we can still assume that for the chosen $R$ that satisfies \eqref{cond:large_enough_R_for_free_products} the first inequality of \eqref{eq:lower_bound_on_expected_size_of_fat_of_animals_through_one_vertex_on_free_product} in Corollary \ref{coro:lower_bound_on_expected_size_of_fat_of_animals_through_one_vertex_on_free_product} holds.
Now, from \eqref{def:first_step_of_exploration_branching} and \eqref{def:general_step_of_exploration_branching} it follows that $E_n \supseteq B_n$ holds for all $n \geq 1$.
Consequently, we have \eqref{eq:the_exploration_runs_indefinitely_on_free_product} once we prove that 
\begin{equation}
	\label{eq:branching_process_explodes}
	\prob \left( \lim_{n \to \infty} |B_n| = + \infty \right) > 0.
\end{equation}
As we argued above, the sequence $\{ |B_n| \}_{n \geq 0}$, where $B_0 := \{ \origin \}$, defines a branching process that, apart from the first generation, follows a Galton-Watson process.
Hence, it follows from the theory of Galton-Watson branching processes (see, for example, \cite[Section 5.1]{LyonsPeres2016}), that \eqref{eq:branching_process_explodes} holds if the expected number of successors is strictly greater than $1$.
And, due to assumption \eqref{cond:large_enough_R_for_free_products}, Claims~\ref{claim:lower_bound_on_reproduction_rate_branching_process_first_generation} and~\ref{claim:lower_bound_on_reproduction_rate_branching_process}, and $m_1^R \leq m_1$ for any $R$, we have that
\begin{equation}
	\E \left[ \left| B_1 \right| \right] \geq 
	\E \left[ \left| \beta_x \right| \right] \geq
	\left(1-e^{-\lambda/2}\right)^2  \cdot \left( \frac{ \lambda \cdot \Cheeger(\graph)^2 \cdot m_2^R}{ 2 \cdot ( \lambda \cdot m_1 + 2) } - (\degree + 1)\right) > 1
\end{equation}
holds for any $x \in V(\graph)$, and so we are done.
\end{proof}

\begin{proof}[Proof of Theorem~\ref{t.free}]
From the assumption $m_2(\nu_{\origin}) = \infty$ and the Monotone Convergence Theorem, it follows that we can choose $R$ large enough so that~\eqref{cond:large_enough_R_for_free_products} holds. Thus Proposition~\ref{prop:the_exploration_runs_indefinitely_on_free_product} says that the connected component of $\origin$ is infinite with positive probability.
This, by the ergodicity \eqref{obs:poisson_zoo_is_ergodic}, also means the almost sure existence of an infinite cluster.
\end{proof}

\subsection{Worms on nonamenable graphs}\label{ss.exploworm}

We now turn our attention to worms on general nonamenable graphs. We will actually work under slightly more general assumptions (see Proposition~\ref{p.growth} and the paragraph before), and then prove Theorem~\ref{thm:random_length_worms_model_on_nonamenable_graphs} by simply substituting the results of Subsection~\ref{subsection:worms_on_nonamenable_graphs_fattening} into that framework.

More precisely, we will assume that we are given a nonamenable graph $\graph$, a $\grp \subgrp \Aut(\graph)$ subgroup of its automorphisms that acts transitively on $V(\graph)$, a $\grp$-rotation invariant measure $\nu_{\origin}$ on $\latticeanimals_{\origin}$, the measure $\nu$ on $\zoospace$ induced from $\nu_{\origin}$ via \eqref{def:induced_measure_on_zoospace} and the corresponding Poisson point process $\animals = \sum_{i \in I} \delta_{(x,H)_i} \sim \mathcal{Q}_{\nu}^\lambda$. Recall also the notion of restriction from Definition~\ref{def:point_measure_restricted_to_hit_B_but_not_the_closure_of_A}.

Subsubsection~\ref{sss.explo} will construct in detail the exploration process sketched in Subsection~\ref{subsection:random_length_worms_model_introduction}. Then, Subsubsection~\ref{sss.expon} will prove that exponential growth of the exploration process happens forever with positive probability.

\subsubsection{The exploration process}\label{sss.explo}

The formal definition of the exploration is as follows. For the initial step, consider the following random sets:
\begin{equation}
\label{def:exploration_process_initial_step}
E_0 := \trace \left( \animals \ind \left[ \zoospace( \{  \origin \} ) \right] \right); \qquad
B_0 := \extboundary E_0.	
\end{equation}
That is, $E_0$ is the set covered by the animals ever visiting $\origin$, and $B_0$ is its exterior boundary.
One can think of $E_0$ as the set of all vertices explored in the $0$'th step and $B_0$ as the basis for the next step, in a sense that the next exploration step will be done through its vertices.

After this initialization, the general step can be described as follows:
\begin{equation}
\label{def:exploration_process_general_step}
C_{n+1} := \trace \left( \check{\animals}_{E_{n}, B_{n}}^R \right); \qquad
E_{n+1} := E_{n} \cup C_{n+1}; \qquad
B_{n+1} := \extboundary E_{n+1} \setminus \extboundary E_{n}.
\end{equation}
In words, for the fat $C_{n+1}$ we take the set covered by all the animals of volume at most $R$ that visit $B_{n}$ but avoid the closure $\overline{E}_{n}$ otherwise. Taking the union of this fat $C_{n+1}$ with the previously explored set $E_{n}$ defines the  next explored set $E_{n+1}$. Then the next exploration will happen through $B_{n+1}$: all those vertices of the exterior boundary of $E_{n+1}$ that are new. See Figure~\ref{f.growth}. This subtraction of the exterior boundary of the previous step is needed to guarantee that no ``direction'' was used more than once: if we could not fatten through some vertices in a step, we will not be able to do so in the future steps, either.

\begin{figure}[tb]
\SetLabels
(0.75*0.72)\textcolor{red}{$E_{n}$}\\
(0.8*0.42)\textcolor{red}{$B_{n}$}\\
(0.45*0.38)\textcolor{cyan}{$C_{n+1}$}\\
(0.67*0.18)\textcolor{blue}{$E_{n+1}$}\\
(0.83*0.03)\textcolor{blue}{$B_{n+1}$}\\
\endSetLabels
\centerline{
	\AffixLabels{
		\includegraphics[width=0.35\textwidth]{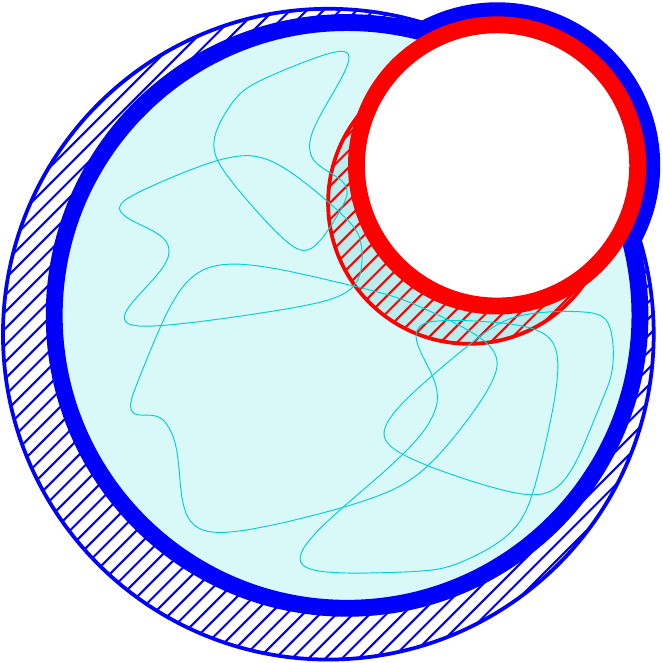}
	}
}
\caption{The fat $C_{n+1}$ is the growth through $B_{n}$, yielding $E_{n+1}$, and then $B_{n+1}$ is the new part of the exterior boundary of $E_{n+1}$.}
\label{f.growth}
\end{figure}

An important consequence of this choice of $B_n$ can be described once we define the filtration corresponding to the exploration process:
\begin{equation}
\label{def:filtration_of_the_exploration}
\Falg_n := \sigma \left( \{ E_k \}_{k = 0}^n, \, \{ B_k \}_{k = 0}^n, \, \{ C_k \}_{k = 0}^n \right),
\end{equation}
where $C_0 := \emptyset$. In words, for $n \geq 0$ the $\sigma$-algebra $\Falg_n$ contains all the information we obtained about and from the exploration process up to the $n$'th step. That is, if we have $\Falg_n$ then we know everything we need to generate the set $C_{n+1}$ and hence $E_{n+1}$ and $B_{n+1}$.

\begin{claim}
\label{claim:independence_of_the_next_step_of_the_exploration_from_the_past}
For any $n \in \N$, given $E_n$ and $B_n$ as in \eqref{def:exploration_process_general_step}, the restricted process $\check{ \animals }_{E_n, B_n }^R$ and hence the set $C_{n+1}$ is independent of $\Falg_n$.
\end{claim}

\begin{proof}
As it was mentioned above, the way we constructed $B_n$ from $E_n$ and $E_{n-1}$ it follows that it consists of vertices that were not exposed before.	
Consequently, the process is restricted to a disjoint subset of animals and due to \eqref{PPP_prop:independence} of Theorem \ref{thm:basic_properties_of_PPPs} it is independent from the previously exposed animals.
\end{proof}

Since $\{E_n\}$ is increasing, one can consider the countable union $E_{\infty} := \lim_{n \to \infty} E_n = \bigcup_{n \in \N} E_n$, and ask if it is infinite or not. If it is infinite, then the cluster of $\origin$ is infinite, as well. But we can and will ask for more:  $| E_{n+1} | > a \cdot | E_{n} |$ for all $n\ge 0$, with some large enough $a>1$. This will be useful, since if we have this for $n-1$, then
\begin{equation}
\label{eq:lower_bound_on_Bn_by_En_assuming_growth}
| \extboundary E_n | - |B_n | \overset{\eqref{def:exploration_process_general_step}}{\leq} 
| \extboundary E_{n-1} | \leq 
\degree \cdot | E_{n-1} | \leq
\frac{\degree}{a} \cdot |E_n|,
\end{equation}
where $\degree$ is the degree in $\graph$. This means that the old part of $\extboundary E_n$ is small, hence we can expect some large growth $C_{n+1}$, and then $|E_{n+1}| > a \cdot |E_{n}|$ will hold again with a large probability, maintaining an inductive argument.

\subsubsection{Exponential growth}\label{sss.expon}

A key idea is to use the following general lemma, which, once stated, is easy to prove. 

\begin{lemma}[Concentrated exponential growth implies survival]\label{l.growth}
Consider an $\N$-valued stochastic process $(Z_n)_{n\ge 0}$, with increments $Y_{n+1}=Z_{n+1}-Z_n \ge 0$. Define the events, for some $b>0$ and $a>1$,
\begin{equation}\label{e.cE}
	\cE_0 := \left\{ Z_0 \ge b \right\} \qquad \text{and} \qquad \cE_{n} := \left\{ Z_n \ge a \cdot  Z_{n-1}  \right\}, \, n \geq 1.
\end{equation}
Let $\Falg_n$ be the sigma-algebra generated by $\{Z_0,Z_1,\dots, Z_n\}$, and $\Falg_n^{\cE}$ be its restriction on the event $\cE_0\cap\dots\cap \cE_n$.  Assume that $\prob ( \cE_0 ) > 0$ and that, for every $n\ge 0$,
\begin{equation}\label{e.EY}
	\E[Y_{n+1} \,|\, \Falg_n^\cE ] \ge 2a \cdot Z_n
\end{equation}
and
\begin{equation}\label{e.VarY}
	\Var[Y_{n+1} \,|\, \Falg_n^\cE ] \leq C \cdot \E[Y_{n+1} \,|\, \Falg_n^\cE ],
\end{equation}
where we assume that $C<b/2$. Then 
\begin{equation}\label{e.ZZZ}
	\prob\big(Z_n > 0 \text{ for all }n\ge 0\big) \ge \prob\left( \bigcap_{n=0}^\infty \cE_n \right) > 0\,.
\end{equation}
\end{lemma}

\begin{proof} First note that \eqref{e.EY} and~\eqref{e.cE} imply, by induction, that 
\begin{equation}\label{e.EYYY}
	\E[Y_{n+1} \,|\, \Falg_n^\cE ] \ge 2 b \cdot a^{n+1}.
\end{equation}
Then, using Chebyshev's inequality in step $(*)$ below:	
\begin{align*}
	\prob \left( \cE_{n+1}^c \, \left| \, \Falg_n^{\cE} \right. \right) 
	&= \prob \left(Z_{n+1} \leq a \cdot Z_n \, \left| \, \Falg_n^{\cE} \right. \right)\\
	& \leq
	\prob \left( Y_{n+1} \leq a \cdot Z_n \, \left| \, \Falg_n^{\cE} \right.  \right) \\ 
	&  =
	\prob \Big( Y_{n+1} - \E \left[ Y_{n+1} \, \left| \, \Falg_n^{\cE} \right. \right] \leq a \cdot Z_n - \E \left[ Y_{n+1} \, \left| \, \Falg_n^{\cE} \right. \right] \, \left| \, \Falg_n^{\cE} \Big) \right.  \\ 
	& \overset{\eqref{e.EY}}{\leq}
	\prob \Big( \left| Y_{n+1} - \E \left[ Y_{n+1} \, \left| \, \Falg_n^{\cE} \right. \right] \right| \geq  \E \left[ Y_{n+1} \, \left| \, \Falg_n^{\cE} \right. \right]/2 \, \left| \, \Falg_n^{\cE} \Big) \right.  \\ 
	& \overset{(*)}{\leq}
	\frac{ 4 \cdot \Var \left[ Y_{n+1} \, \left| \, \Falg_n^{\cE} \right. \right] }{ \E \left[ Y_{n+1} \, \left| \, \Falg_n^{\cE} \right. \right]^2} 
	\overset{\eqref{e.VarY}}{\leq}
	\frac{ 4 C }{  \E \left[ Y_{n+1} \, \left| \, \Falg_n^{\cE} \right. \right]} 
	\overset{\eqref{e.EYYY}}{\leq} \frac{4C}{2b \cdot a^{n+1}}.
\end{align*}
By averaging over the possible values of $Z_0,Z_1,\dots,Z_n$ within the event $\cE_0\cap\dots\cap \cE_n$,
\[
	\prob \left( \cE_{n+1}^c \, \left| \, \cE_0\cap\dots\cap \cE_n   \right. \right) \leq \frac{2C}{b \cdot a^{n+1}}.
\]	
Therefore,
\[
	\prob\left( \bigcap_{n=0}^\infty \cE_n \right) = \prob(\cE_0) \cdot \prod_{n=0}^\infty \prob \big( \cE_{n+1} \,\big|\, \cE_0 \cap\dots\cap \cE_n \big) \ge \prob(\cE_0) \cdot \prod_{n=0}^\infty \left( 1- \frac{2C}{b \cdot a^{n+1}} \right).
\]
Since the exponentially decaying terms (recall $a>1$) are summable in $n$, the infinite product is positive provided that each factor is positive, which holds by the assumption that $C<b/2$.
\end{proof}

The input from Subsection~\ref{subsection:worms_on_nonamenable_graphs_fattening} on the growth of the exploration process~\eqref{def:exploration_process_general_step} is quite similar to the assumptions in Lemma~\ref{l.growth}: Corollary~\ref{coro:lower_bound_on_expected_size_of_fat_worms_on_nonamenable_graphs} is a first moment lower bound on the increments, similar to~\eqref{e.EY}, while Lemma~\ref{lemma:upper_bound_on_the_second_moment_of_the_occupied_sets_size_for_worms} is a variance upper bound, similar to~\eqref{e.VarY}. The following proposition makes the actual transition from reasonably general Poisson zoos to the lemma. The existence of the function $\vartheta(R)$ in the first moment lower bound on the proposition, for the case of worms, will basically be provided by $m_2^R$; however, on some graphs, some animal measures, one may imagine using other functions, such as a truncated lower moment.

\begin{proposition}
\label{p.growth}
Consider a Poisson zoo on an infinite transitive graph $\graph$ with intensity $\lambda>0$ and a lattice animal measure $\nu_o$ that has unbounded support. Assume further that there exist constants $0 < c_1(\graph,\lambda), c_2(\graph), C(R) < \infty$ and a function $\vartheta \, : \, \N \rightarrow \R_{\ge 0}$ with $\lim_{R \to \infty} \vartheta(R) = \infty$ such that for any $R \in \N$ and for any finite connected subset $A \fsubset V(\graph)$ of vertices and any subset $B \subseteq \extboundary A$ of its exterior boundary for the restricted process of Definition~\ref{def:point_measure_restricted_to_hit_B_but_not_the_closure_of_A}, we have
\begin{equation}
	\label{assumption:expected_size_fat_set}
	\E \left[ \left| \trace \left( \check{ \animals }_{A,B}^R \right) \right| \right] \geq 
	c_1(\graph,\lambda)  \cdot \vartheta(R) \cdot \big[ c_2(\graph) \cdot |A| - \left( | \extboundary A | - |B|  \right) \big]
\end{equation}
and 
\begin{equation}
	\label{assumption:variance_size_fat_set}
	\Var \left( \left| \trace \left( \check{ \animals }_{A,B}^R \right) \right| \right) \leq
	C(R) \cdot \E \left[ \left| \trace \left( \check{ \animals }_{A,B}^R \right) \right| \right].
\end{equation}

Then the exploration process defined in \eqref{def:exploration_process_general_step} survives forever with positive probability:
\[
	\prob \big( |E_n|>0 \text{ for all }n\ge 0 \big) > 0.
\]
\end{proposition}

\begin{proof}	
Naturally, we want to use Lemma~\ref{l.growth} with $Z_n=|E_n|$ and $Y_n=|C_n|$ from the exploration process \eqref{def:exploration_process_general_step}. However, we will need to choose the constant $a$, which governs the exponential growth, with some care. Namely, assumption~\eqref{assumption:expected_size_fat_set} tells us that
\begin{align*}
	\E\big[|C_{n+1}| \,\big|\, \Falg_n^\cE \big] &\ge c_1(\graph,\lambda) \cdot \vartheta(R) \cdot \big[ c_2(\graph) \cdot |E_n| - \left( | \extboundary E_n | - |B_n|  \right) \big] \\
	& \overset{\eqref{eq:lower_bound_on_Bn_by_En_assuming_growth}}{\geq} c_1(\graph,\lambda) \cdot \vartheta(R) \cdot  \left( c_2(\graph) - \frac{\degree}{a} \right)  \cdot |E_n|.
\end{align*}
So, in order to satisfy \eqref{e.EY}, choose
\begin{equation}
	\label{eq:condition_on_parameters:a,R}
	a := \frac{\degree}{2 c_2(\graph)} \vee 1\,,
	\qquad \text{and then} \qquad
	\vartheta(R)  > \frac{4a}{ c_1(\graph,\lambda) \cdot c_2(\graph)}. 
\end{equation}
This is possible, since $\lim_{R \to \infty} \vartheta(R) = \infty$.

For~\eqref{e.VarY}, we have $C=C(R)$ from assumption~\eqref{assumption:variance_size_fat_set}. The last bit we need is to be able to choose $b>2 \cdot C(R)$ such that $\prob[|E_0|>b]>0$ still holds. This is obviously ensured by $\nu_o$ having an unbounded support. 
\end{proof}

\begin{remark}
One may find the right-hand side of~\eqref{assumption:expected_size_fat_set} a little weird --- possibly, a lower bound of the form $c_3(\graph,\lambda) \cdot \vartheta(R) \cdot |B|$ would be more natural. For instance, we would get something like this from  Corollary~\ref{coro:lower_bound_on_expected_size_of_fat_of_animals_through_one_vertex_on_free_product} for general animal measures on free products. However, in a nonamenable graph, for any $B \subseteq \extboundary A$ we have
\begin{equation}
	\label{eq:silly_expression_for_the_size_of_B}
	|B| = |\extboundary A| - \left( |\extboundary A| - |B| \right) \geq \Cheeger(\graph) \cdot |A| - \left( | \extboundary A | - |B| \right),
\end{equation}	
hence the current assumption \eqref{assumption:expected_size_fat_set} actually covers this other form, as well.
\end{remark}

\begin{remark}\label{r.BoR}
Also note that, according to Lemma~\ref{lemma:upper_bound_on_the_second_moment_of_the_occupied_sets_size_for_worms}, the choice $c(R) := |\ball(\origin, R)|+1$ will always work in~\eqref{assumption:variance_size_fat_set}, hence this assumption is redundant.
We chose to keep it to emphasize that some control is needed on the variance of the size of the explored set.
\end{remark}

At last, here is the proof of Theorem~\ref{thm:random_length_worms_model_on_nonamenable_graphs}, hence Theorem~\ref{t.worms}:

\begin{proof}[Proof of Theorem \ref{thm:random_length_worms_model_on_nonamenable_graphs}]
This reduces to a simple substitution into Proposition~\ref{p.growth}. For the constants in~\eqref{assumption:expected_size_fat_set}, given the constants in Corollary~\ref{coro:lower_bound_on_expected_size_of_fat_worms_on_nonamenable_graphs}, choose 
\begin{equation*}
	c_1(\graph,\lambda):=\frac{ c(\graph) \cdot \lambda }{ \lambda \cdot m_1 + 1 }, 
	\qquad
	c_2(\graph) := \frac{(1 - \spectral(\graph))^2 (1 + \Cheeger(\graph))}{2},
	\qquad
	\vartheta(R) := \E\left[ \mathcal{L}^2 \ind \left[ \mathcal{L} \leq R \right] \right].
\end{equation*}
The last one works because $\lim_{R \to \infty} \vartheta(R) = \E \left[ \mathcal{L}^2 \right] = \infty$. Regarding~\eqref{assumption:variance_size_fat_set}, the choice $C(R) := |\ball(\origin, R)|+1$ works by Lemma~\ref{lemma:upper_bound_on_the_second_moment_of_the_occupied_sets_size_for_worms}.

Thus Proposition~\ref{p.growth} tells us that the cluster of $\origin$ is infinite with positive probability. By ergodicity~\eqref{obs:poisson_zoo_is_ergodic}, we get an infinite cluster almost surely.
\end{proof}

\section{An example of immediate uniqueness}\label{s.immuniq}

By the direct product $G\times H$ of two graphs we mean that 
\[
	V(G\times H):=\big\{ (x,u) : x\in V(G),\ u\in V(H)\big\},
\]
and the adjacency relation defining $E(G\times H)$ is
\[
	(x,u) \sim (y,v) \qquad \textrm{if{f}} \qquad x=y \textrm{ and }u\sim v, \textrm{ or }x\sim y \textrm{ and }u=v\,.
\]

\begin{proof}[Proof of Proposition~\ref{p.immuniq}]
Let us denote the origin of $G=\tree_\degree \times \Z^5$ by $o=(\rho,0)$. Every vertex $(x,u) \in V(\tree_\degree \times \Z^5)$ is contained in a subgraph $Z_x:= \{x\} \times \Z^5$. Let $\tilde\nu_0$ be a random length worm measure on $\Z^5$ with $\tilde\nu_0(\{0\})=p>0$ and $\E_{\tilde\nu_0} |H|<\infty$, but $\lambda_c(\tilde\nu)=0$, provided by Ráth and Rokob \cite{RathRokob2022}. Then let $\nu_o$ be the pushforward of $\tilde\nu_0$ by the obvious root-preserving graph isomorphism from $\Z^5$ to $Z_\rho$. The Poisson zoo $\mathcal{Q}_{\nu}^\lambda$ we are considering is then the one generated by $\nu_o$ and the transitive and unimodular automorphism group $\grp=\Aut(\tree_\degree)\times \Z^5$, with the group $\Z^5$ acting on its standard Cayley graph $\Z^5$ by translations.

Because of the singleton probability $p>0$ above, the Poisson zoo on $\Z^5$ for any $\lambda>0$ is insertion tolerant, and $\Z^5$ is amenable, hence the Burton-Keane argument \cite[Theorem 7.9]{LyonsPeres2016} tells us that there is a unique infinite cluster for any $\lambda>0$. So, almost surely, each $Z_x$ copy of $\Z^5$ has a unique infinite cluster $C_x$. Let $\theta=\theta(\lambda):=\prob \big( (x,u) \in C_x \big)>0$, independent of $u$ or $x$. 

For any $x,y\in V(\tree_\degree)$, consider
\[
	\mathcal{U}_{x,y}:=\big\{ u\in\Z^5 : \text{both } (x,u) \in C_x \text{ and }(y,u) \in C_y\big\}\,.
\]
Since, for any $u\in \Z^5$, we have $\prob\left( u\in \mathcal{U}_{x,y}\right) = \theta^2$, Fatou's lemma gives that 
\[
	\prob \big( |\mathcal{U}_{x,y} | = \infty \big) \geq \theta^2 > 0,
\]
and then, by the ergodicity of the Poisson zoo, this probability is 1. Now, there is a finite path $\gamma$ in $\tree_\degree$ between $x$ and $y$, and each vertex of $\gamma \times \{ u \}$, for any $u\in\Z^5$, is covered by a singleton in $\mathcal{Q}_{\nu}^\lambda$ with probability at least $p\lambda$, independently, so the entire path is open with probability at least $(p\lambda)^{|\gamma|}$. Thus, among the infinitely many $u \in \mathcal{U}_{x,y}$, there will be ones for which $\gamma \times \{ u \}$ is fully open, meaning that $C_x$ and $C_y$ are contained in the same cluster of $G$.

Therefore, for any $\lambda>0$, there is a(n obviously unique) infinite open cluster on $G$ such that its intersection with each $Z_x$, $x\in\tree_\degree$, contains the unique infinite open cluster within $Z_x$. In principle, there may exist other infinite clusters on $G$, as well, glued together from finite clusters within some of $Z_x$'s. However, the indistinguishability of infinite clusters, shown for insertion tolerant invariant percolations in \cite{LySch}, with a simpler proof in \cite{Damis} that also includes any Poisson zoo, excludes this possibility. That is, $\lambda_u(G,\nu)=0$.
\end{proof}



\begin{thebibliography}{99}

\bibitem{Abert:ICM}
Ab\'{e}rt, M.: 
On a curious problem and what it lead to. 
{\it In: Proc. Int. Cong. Math. 2022, Vol.~5}. pp. 3374--3387. 
European Mathematical Society, 2023.
\MR{4680366}

	
\bibitem{AW}
Ab\'{e}rt, M. and Weiss, B.:
Bernoulli actions are weakly contained in any free action.
\textit{Ergodic Theory Dynam. Systems} 
\textbf{33}, (2013), 323--333.
\MR{3035287}
	
\bibitem{BaB}
Babson, E. and Benjamini, I.:
Cut sets and normed cohomology, with applications to percolation.
\textit{Proc. Amer. Math. Soc.}
\textbf{127}, (1999), 589--597.
\MR{1622785}

\bibitem{BandyopadhyaySteifTimar2010}
Bandyopadhyay, A., Steif, J.~E., and Tim\'{a}r, \'{A}.:
On the cluster size distribution for percolation on some general graphs.
\textit{Rev. Mat. Iberoam.}
\textbf{26}, (2010), 529--550.
\MR{2677006}

\bibitem{BenjaminiJonassonSchrammTykesson2009}
Benjamini, I., Jonasson, J., Schramm, O. and Tykesson, J.:
Visibility to infinity in the hyperbolic plane, despite obstacles.
\textit{ALEA Lat. Am. J. Probab. Math. Stat.}
\textbf{6}, (2009), 323--342.
\MR{2546629}

\bibitem{BLPS:GAFA}
Benjamini, I., Lyons, R., Peres, Y. and Schramm, O.:
Group-invariant percolation on graphs.
\textit{Geom. Funct. Anal.}
\textbf{9}, (1999), 29--66. 
\MR{1675890}

\bibitem{BenjaminiNachmiasPeres2011}
Benjamini, I, Nachmias, A. and Peres, Y.:
Is the critical percolation probability local? 
\textit{Probab. Theory Related Fields}
\textbf{149}, (2011), 261–-269.
\MR{2773031}

\bibitem{BenjaminiSchramm1996}
Benjamini, I. and Schramm, O.:
Percolation beyond $\Z^d$, many questions and a few answers.
\textit{Electron. Comm. Probab.}
\textbf{1}, (1996), 71--82.
\MR{1423907}

\bibitem{BorbenyiRathRokob}
Borb\'{e}nyi, M, R\'{a}th, B and Rokob, S.:
Random interlacement is a factor of i.i.d.
\textit{Electron. J. Probab.}
\textbf{28}, (2023), 1--45.
\MR{4580894}

\bibitem{Bowen2019}
Bowen , L.:
Finitary random interlacements and the Gaboriau--Lyons problem.
\textit{Geom. Funct. Anal.}
\textbf{29}, (2019), 659-–689.
\MR{3962876}

\bibitem{BromanTykesson2016}
Broman, E. I. and Tykesson, J.:
Connectedness of Poisson cylinders in Euclidean space.
\textit{Ann. Inst. Henri Poincar\'{e} Probab. Stat.}
\textbf{52}, (2016), 102--126.
\MR{3449296}

\bibitem{BromanTykesson2015}
Broman, E. I. and Tykesson, J.:
Poisson cylinders in hyperbolic space.
\textit{Electron. J. Probab.}
\textbf{20}, (2015), 1--25.
\MR{3335832}

\bibitem{CPZ22}
Cai, Z, Procaccia, E. B. and Zhang, Y.:
Continuity and uniqueness of percolation critical parameters in finitary random interlacements.
\textit{Electron. J. Probab.}
\textbf{27}, (2022), 1--46.
\MR{4460272}

\bibitem{Chang2021}
Chang, Y.:
Bernoulli hyper-edge percolation on $\Z^d$.
\ARXIV{2101.06082}

\bibitem{ChangShapoznikov2016}
Chang, Y. and Sapozhnikov, A.:
Phase transition in loop percolation.
\textit{Probab. Theory Related Fields}
\textbf{164}, (2016), 979--1025.
\MR{3477785}

\bibitem{ColettiMirandeGrynberg2020}
Coletti, C. F., Mirana, D. and Grynberg S. P.:
Boolean percolation on doubling graphs.
\textit{J. Stat. Phys.}
\textbf{178}, (2020), 814--831.
\MR{4059962}

\bibitem{CsMP}
Cs\'oka, E., Mester, P., and Pete, G.:
Quantitative indistinguishability and sparse and dense clusters in factor of IID percolations.
\ARXIV{2512.07740}

\bibitem{DAGKRW}
D'Achille, M., Greb\'{\i}k, J., Khezeli, A., Recke, K., and Wilkens, A.: 
Vanishing uniqueness thresholds in Voronoi percolation on products. 
\ARXIV{2511.23317}

\bibitem{DeTa}
Dembin, B. and Tassion, V.:
Almost sharp sharpness for Poisson Boolean percolation.
{\it Ann. Probab.} 
{\bf 54} (2026), 1772--1802.
\MR{5092946}

\bibitem{DeMu}
Dewan, V. and Muirhead, S.:
Mean-field bounds for Poisson-Boolean percolation.
\textit{Electron. J. Probab.}
\textbf{28}, (2023), 1--24.
\MR{4551565}

\bibitem{DPR}
Drewitz, A., Pr\'evost, A., and Rodriguez, P.-F.:
Geometry of Gaussian free field sign clusters and random interlacements. 
\textit{Probab. Theory Related Fields},
(2024), 96 pages.
\MR{4914069}

\bibitem{DrewitzRathSapozhnikov2014}
Drewitz, A., R\'{a}th, B. and Sapozhnikov, A.:
An introduction to Random Interlacements.
SpringerBriefs in Mathematics.
\emph{Springer}, Cham,
2014, x+120 pp.
\MR{3308116}

\bibitem{Duke}
Duminil-Copin, H., Goswami, S., Raoufi, A., Severo, F., and Yadin, A.:
Existence of phase transition for percolation using the Gaussian free field. 
\textit{Duke Math. J.}
\textbf{169}, (2020), 3539--3563.
\MR{4181032}

\bibitem{DukeEquality}
Dumini-Copin, H, Goswami, S., Rodriguez, P.-F., and Severo, F.:
Equality of critical parameters for percolation of Gaussian free field level sets.
\textit{Duke Math. J.}
\textbf{172}, (2023), 839--913.
\MR{4568695}

\bibitem{EST}
Easo, P., Severo, F., and Tassion, V.: 
Counting minimal cutsets and $p_c<1$.
{\it Forum of Math. Pi}, 
{\bf 13} (2025), e23
\MR{4972123}

\bibitem{Damis}
El Alami, D., Pete, G., and Tim\'ar, \'A.:
Indistinguishability for recurrent clusters.
\ARXIV{2512.18435}

\bibitem{FrosstromGantertSteif2024}
Forsstr\"{o}m, M.~P., Gantert, N. and Steif, J.~E.:
Poisson representable processes.
{\it Probab. Theory Related Fields}, 
{\bf 195} (2026), 741--780.
\MR{5009006}
 
\bibitem{Mikolaj}
Fr\c{a}czyk, M.:
Infinitesimal containment and sparse factors of iid.
\ARXIV{2512.09301}

\bibitem{FMW}
Fr\c{a}czyk, M., Mellick, S. and Wilkens, A.:
Poisson-Voronoi tessellations and fixed price in higher rank.
{\it Ann. Math.}, to appear.
\ARXIV{2307.01194} 

\bibitem{Gabo:free}
Gaboriau, D.:
Co\^ut des relations d’équivalence et des groupes. 
\textit{Inv. math.}
\textbf{139}, (2000), 41--98.
\MR{1728876}

\bibitem{Gabo:ICM} 
Gaboriau, D.:
Orbit equivalence and measured group theory.
\textit{Proc.\ Int.\ Congr.\ Math. Vol.~III,}
\emph{Hindustan Book Agency}, New Delhi,
(2010), 1501--1527.
\MR{2827853}

\bibitem{GaboriauLyons}
Gaboriau, D. and Lyons, R.:
A measurable-group-theoretic solution to von Neumann’s problem.
\textit{Invent. Math.} 
\textbf{177}, (2009), 533--540.
\MR{2534099}

\bibitem{Gouere2008}
Gou\'er\'e, J.--B.:
Subcritical regimes in the Poisson Boolean model of continuum percolation.
\textit{Ann. Probab.}
\textbf{36}, (2008), 1209--1220.
\MR{2435847}

\bibitem{convexgrain}
Gracar, P., Korfhage, M. and M\"{o}rters, P.:
Robustness in the Poisson Boolean model with convex grains.
\ARXIV{2410.13366}

\bibitem{GrebReck}
Greb\'{\i}k, J. and Recke, K.:
Poisson-Voronoi percolation in higher rank.
\ARXIV{2504.02435}

\bibitem{Haggstrom}
H\"{a}ggstr\"{o}m, O.:
Infinite clusters in dependent automorphism invariant percolation on trees.
\textit{Ann. Probab.}
\textbf{25}, (1997), 1423--1436.
\MR{1457624}

\bibitem{socialfrogs}
Hermon, J., Morris, B., Qin, C., and Sly, A.:
The social network model on infinite graphs. 
{\it Ann. Appl. Probab.} 
{\bf 30} (2020), 902--935.
\MR{4108126}

\bibitem{HuPe}
Hutchcroft, T. and Pete, G.:
Kazhdan groups have cost 1.
\textit{Invent. Math.}
\textbf{221}, (2020), 873--891.
\MR{4132958}

\bibitem{Krakow}
Jardón-Sánchez, H., Mellick, S., Poulin, A., and Wróbel, K.:
Exactness and the topology of the space of invariant random equivalence relations.
{\it Proc. London Math. Soc.} (2026), to appear. \ARXIV{2502.11622}

\bibitem{Khezeli}
Khezeli, A.:
Products of infinite countable groups have fixed price one. 
\ARXIV{2509.08325}

\bibitem{Kingman1993}
Kingman, J. F. C.:
{\it Poisson processes.}
Oxford Studies in Probability, 3.
\emph{The Clarendon Press, Oxford University Press}, New York,
1993. viii + 104 pp.
\MR{1207584}

\bibitem{Levitt} 
Levitt, G.:
On the cost of generating an equivalence relation.
\textit{Ergodic Theory Dynam. Systems} 
\textbf{15}, (1995), 1173--1181.
\MR{1366313}

\bibitem{LovaszBook}
Lov\'asz, L.
{\it Large networks and graph limits.}
Amer.\ Math.\ Soc., Providence, RI, 2012.
\MR{3012035}

\bibitem{Lupu}
Lupu, T.:
From loop clusters and random interlacements to the free field.
\textit{Ann. Probab.}
\textbf{44}, (2016), 2117--2146.
\MR{3502602}

\bibitem{Lyons} 
Lyons, R.:
Factors of IID on trees.
\textit{Combin. Probab. Comput.} 
\textbf{26}, (2017), 285--300.
\MR{3603969}

\bibitem{LyonsPeres2016}
Lyons, R. and Peres, Y.:
Probability on trees and networks.
Cambridge Series in Statistical and Probabilistic Mathematics, 42.
\emph{Cambridge University Press}, New York,
2016. xv+699 pp.
\MR{3616205}

\bibitem{LySch}
Lyons, R. and Schramm, O.:
Indistinguishability of percolation clusters.
\textit{Ann. Probab.}
\textbf{27}, (1999), 1809--1836.
\MR{1742889}

\bibitem{MeesterRoy1996}
Meester, R. and Roy, R.:
Continuum Percolation,
Cambridge Tracts in Mathematics, Volume 119.
\emph{Cambridge University Press}, Cambridge,
1996. x+231 pp.
\MR{1409145}

\bibitem{SapoMu}
Mu, Y. and Sapozhnikov, A.:
Uniqueness of the infinite connected component for the vacant set of random interlacements on amenable transient graphs.
\textit{Electron. Commun. Probab.}
\textbf{28}, (2023), 1--9.
\MR{4684068}

\bibitem{PSN}
Pak, I. and Smirnova-Nagnibeda, T.:
On non-uniqueness of percolation on nonamenable Cayley graphs.
\textit{C. R. Acad. Sci. Paris Sér. I Math.}
\textbf{330}, (2000), 495--500.
\MR{1756965}

\bibitem{PGG}
Pete, G.:
Probability and geometry on groups: Lecture notes for a graduate course.
Book in preparation, available at
\href{https://math.bme.hu/~gabor/PGG.pdf}{https://math.bme.hu/~gabor/PGG.pdf}.

\bibitem{physnetw}
Pete, G., Tim\'{a}r, \'{A}., Stef\'{a}nsson, S. \"{O}, Bonamassa, I. and P\'{o}sfai, M.
Physical networks as network-of-networks.
\textit{Nat. Commun.}
\textbf{15}, (2024), article number 4882.
\DOI{10.1038/s41467-024-49227-8}

\bibitem{PYZ}
Procaccia, E. B., Ye, J. and Zhang, Y.:
Percolation for the finitary random interlacements.
\textit{ALEA, Lat. Am. J. Probab. Math. Stat.}
\textbf{18}, (2021), 265--287.
\MR{4198877}

\bibitem{RathRokob2022}
R\'ath, B. and Rokob, S.:
Percolation of worms.
\textit{Stochastic Process. Appl.}
\textbf{152}, (2022), 233--288.
\MR{4453159}

\bibitem{Resnick2008}
Resnick, S. I.:
Extreme values, regular variation and point processes.
Reprint of the 1987 original. Springer Series on Operations Research and Financial Engineering.
\textit{Springer}, New York.
2008. xiv+320 pp.
\MR{2364939}

\bibitem{Sznitman2010}
Sznitman, A.-S.:
Vacant set of random interlacements and percolation.
\textit{Ann. of Math.}
\textbf{171}, (2010), 2039--2087.
\MR{2680403}

\bibitem{Teixeira2009}
Teixeira, A.:
Interlacement percolation on transient weighted graphs.
\textit{Electron. J. Probab.}
\textbf{14}, (2009), 1604--1627.
\MR{2525105}

\bibitem{TeixeiraUngaretti2017}
Teixeira, A. and Ungaretti, D.:
Ellipses percolation.
\textit{J. Stat. Phys.}
\textbf{168}, (2017), 369--393.
\MR{3667365}

\bibitem{TeixeiraTykesson2013}
Teixeira, A. and Tykesson, J.:
Random interlacements and amenability.
\textit{Ann. Appl. Probab.}
\textbf{23}, (2013), 923--956.
\MR{3076674}

\bibitem{Thom}
Thom, A.:
A remark about the spectral radius. 
\textit{Int. Math. Res. Not. IMRM}
\textbf{2015}, (2015), 2856--2864.
\MR{3352259}

\bibitem{Timar:cut}
Tim\'ar, \'A.:
Cutsets in infinite graphs.
\textit{Combin. Probab. Comput.}
\textbf{16}, (2007), 159--166.
\MR{2286517}

\bibitem{Tykesson2007}
Tykesson, J.:
The number of unbounded components in the Poisson Boolean model of continuum percolation in hyperbolic space.
\textit{Electron. J. Probab.}
\textbf{12}, (2007), 1379--1401.
\MR{2354162}

\bibitem{Tykesson2009}
Tykesson, J.:
Continuum percolation at and above the uniqueness threshold on homogeneous spaces.
\textit{J. Theoret. Probab.}
\textbf{22}, (2009), 402--417.
\MR{2501327}

\bibitem{TykessonCalka2013}
Tykesson, J. and Calka, P.:
Asymptotics of visibility in the hyperbolic plane.
\textit{Adv. in Appl. Probab.}
\textbf{45}, (2013), 332--350.
\MR{3102454}

\bibitem{TykessonWindisch2012}
Tykesson, J. and Windisch, D.:
Percolation in the vacant set of Poisson cylinders.
\textit{Probab. Theory Related Fields}
\textbf{154}, (2012), 165--191.
\MR{2981421}

\bibitem{Woess2000}
Woess, W.:
Random walks on infinite graphs and groups.
Cambridge Tracts in Mathematics, 138.
\emph{Cambridge University Press}, Cambridge,
2000. xii+334 pp.
\MR{1743100}

\end{thebibliography}



\end{document}

\section{Introduction}\label{s.intro}

\subsection{Results and questions}\label{ss.results}

The {\bf Poisson zoo} is a very simple correlated {site percolation} model on infinite transitive graphs, introduced in this generality by Ráth and Rokob \cite{RathRokob2022}. Intuitively, we drop random rooted connected sets, called {\bf lattice animals}, at each vertex of the graph, independently, with intensity $\lambda>0$. One could think of a random collection of balls with random radii, or of random walk trajectories with some random finite lengths. In this paper, we will establish quite general scenarios when we can take $\lambda$ arbitrarily small, making the overall coverage density arbitrarily small, but still, large animals occur often enough to find each other and produce infinite connected components. This behaviour is very different from i.i.d.~Bernoulli percolation, and is interesting for several reasons.

More precisely, we start with a locally finite infinite graph $\graph=(V,E)$ and a subgroup of graph automorphisms $\grp\leq \Aut(\graph)$  that acts transitively on $V$. We also have a probability measure $\nu=\nu_o$ on rooted connected sets $(o,H)$ with $o\in H \subseteq V$, called {lattice animals}, where the measure is invariant under root-preserving automorphisms in $\grp$. For any $x\in V$, we can take any graph automorphism $\varphi_{x} \in \grp$ taking $o$ to $x$, and translate $\nu_o$ to $\nu_x := \nu_o \circ \varphi_{x}^{-1}$. It is easy to see that, because of the invariance of $\nu_o$ under root-preserving automorphisms, $\nu_x$ does not depend on the choice of $\varphi_{x}$. (See Subsection~\ref{subsection:poisson_zoo_on_transitive_graphs} for more details.) We fix $\lambda>0$. For each $x\in V$, we  sample an independent $\POI(\lambda)$ variable $N_x^\lambda$, then $N_x^\lambda$ independent animals from $\nu_x$, denoted by $\{H^x_i\}_{i=1}^{N_x}$. We then consider the subset of $V$ covered by all the animals,
\begin{equation}\label{e.zoo}
S_\nu^\lambda := \bigcup_{x\in V} \bigcup_{i=1}^{N_x} H^x_i\,,
\end{equation}
and the main question is whether
\begin{equation}\label{e.lambdac}
\lambda_c = \lambda_c(\graph,\nu) := \inf\left\{ \lambda > 0 :  \P\big(S_\nu^\lambda \textrm{ has some infinite connected component}\big)>0  \right\}
\end{equation}
is strictly between 0 and $\infty$. In the simplest example, Bernoulli site percolation, where animals are just single vertices, $0< \lambda_c$ holds for any bounded degree graph \cite[Theorem 6.47]{LyonsPeres2016}, while the much more difficult direction, $\lambda_c < \infty$, is known for all Cayley graphs that are not finite extensions of $\Z$ \cite{Duke} and for all uniformly transient graphs, even without transitivity \cite{EST}.

To state our results, we need one more technical but crucial assumption: $\graph$ should be not only transitive, but also, the action of $\grp$ on it should be {unimodular}; see Definition~\ref{def:unimodularity} below. For instance, any finitely generated Cayley graph $\graph$ of any group $\grp$, with the group acting on $\graph$, is fine. Unimodularity ensures the following basic facts; see Lemmas~\ref{lemma:two_silly_lemmas} and~\ref{lemma:size_biasing_for_Poisson_zoo}:
\begin{itemize}
\item If $\E_{\nu} |H| = \infty$, then, for any $\lambda > 0$, the entire graph is covered by lattice animals.
\item If  $\E_{\nu} |H| < \infty$, then the number of animals covering any given vertex $x$ has distribution $\POI\big( \lambda \E_{\nu} |H| \big)$. In particular, for small $\lambda$, the density of covered vertices is small.
\item If $\E_{\nu} |H|^2<\infty$, there is a percolation phase transition: $\lambda_c(\graph,\nu) \in (0,\infty)$. The main reason is that the expected total size of animals covering a given vertex is $\lambda\, \E_{\nu} |H|^2$. Note the size-biasing effect: the second moment of the animal size governs the first moment in any cluster exploration.
\end{itemize}

On $\R^d$, if the animals are {\it balls} of random radii, this is called the {\bf Poisson Boolean model}, and Gou\'er\'e \cite{Gouere2008} proved that $\E_{\nu} |H| < \infty$ (now $|H|$ denoting Lebesgue measure) already implies the existence of a non-trivial $\lambda_c \in (0,\infty)$. This was extended by an easy coupling argument to graph metric balls in $\Z^d$ in \cite{RathRokob2022}, and by copying the proof of \cite{Gouere2008} to transitive graphs of polynomial volume growth in \cite{ColettiMirandeGrynberg2020}. The easy general facts displayed above were noticed by R\'ath and Rokob \cite{RathRokob2022}, who proved that, for {\bf worms}, where the animals are random walk trajectories of a random length, for $d\ge 5$, already  $\E_{\nu}|H|^{2-o(1)}=\infty$ suffices for $\lambda_c=0$, where the $o(1)$ power is just a poly-$\log\log$ factor. It has remained open if $\E_{\nu}|H|^{2}=\infty$ for worms is actually sufficient for $\lambda_c=0$; see \cite[Question 1.11]{RathRokob2022}. In any case, that work has shown that balls and worms are almost on the two extremes of the possible behaviours on $\Z^d$, $d\ge 5$.

In the present paper, we close the gap for worms on any nonamenable unimodular transitive graph, and show that the difference between worms and balls and any other lattice animal measure disappears in many (possibly all?) nonamenable examples.

\begin{theorem}[Random length worms on any nonamenable group]\label{t.worms}
If  $\graph$ is \emph{any} nonamenable unimodular transitive graph, and the lattice animals are simple random walk trajectories of random length, satisfying $\E_\nu |H|^2=\infty$, then the Poisson zoo has $\lambda_c=0$.
\end{theorem}

\vskip -0.2 cm
\begin{figure}[htbp]
\SetLabels
(0.7*0.35)\textcolor{red}{$E_{n-1}$}\\
(1*0.35)\textcolor{blue}{$E_{n}$}\\
\endSetLabels
\centerline{
\AffixLabels{
\includegraphics[width=2.8 in]{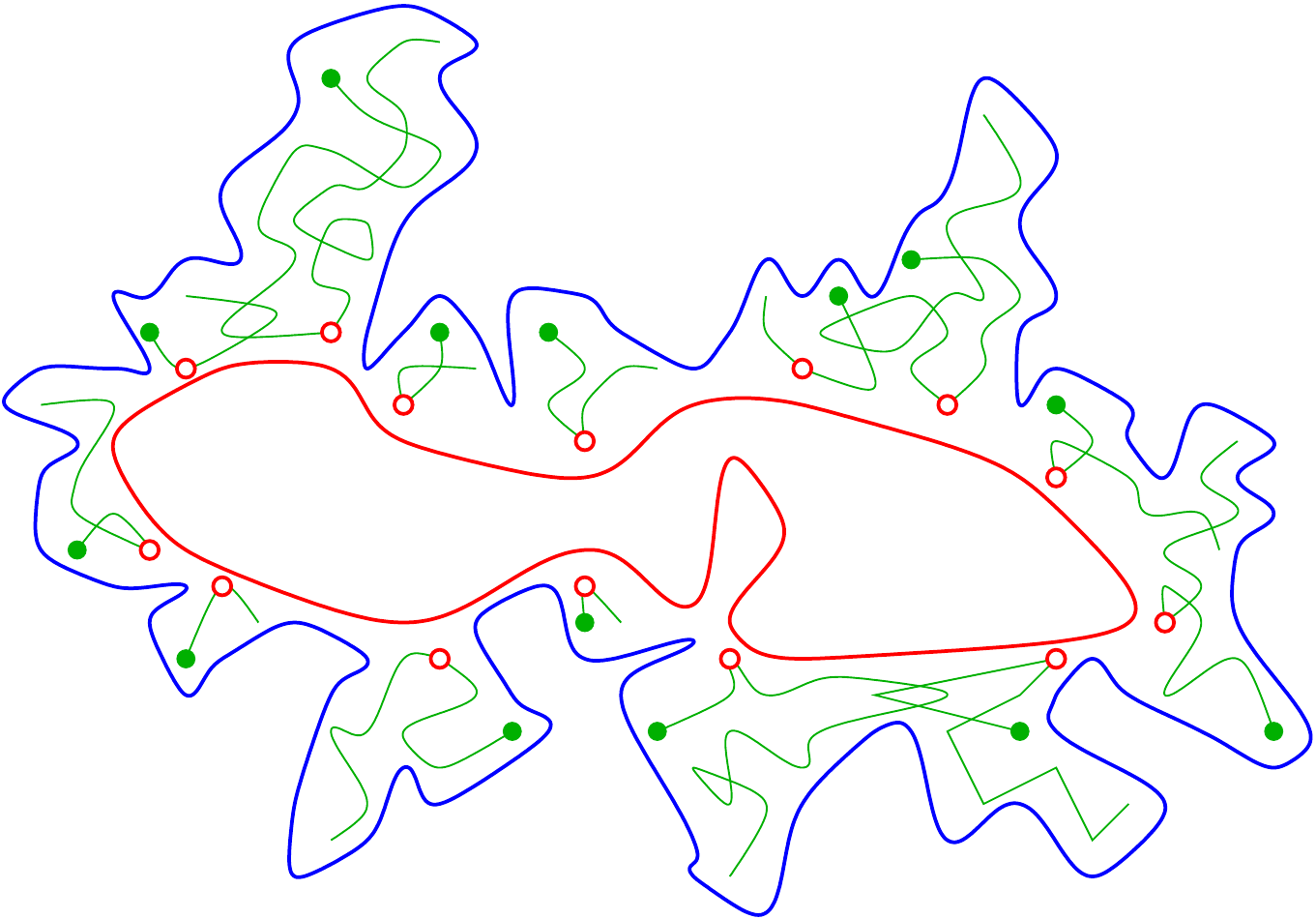}
}}
\caption{Exploring a cluster: at any given stage $E_{n-1}$, there are many exposed vertices on its boundary where new worms (the green trajectories) can touch it, creating a much larger $E_n$.}
\label{f.lincap}
\end{figure}

The key advantage of nonamenable transitive graphs compared to $\Z^d$, high $d$, is that the  {\bf random walk capacity} of {\it any} finite set $S\subset V(\graph)$ is \emph{linear} in the volume $|S|$; see Lemma~\ref{lemma:capacity_bound}. This makes it possible to build an exploration process, with enough new worms touching the already explored cluster from the outside for the exploration never to stop; see Figure~\ref{f.lincap}. The supercriticality of the exploration, for an arbitrary $\lambda>0$, will be fuelled by $\E_\nu |H|^2=\infty$: after size-biasing, this corresponds to the exploration step having an infinite mean. The difficulty in the argument is that the group is completely general (besides being nonamenable), we do not have any structural information, hence neither the dynamic renormalization approach of \cite{RathRokob2022} makes any sense, nor the exploration could be done in a usual branching process manner. Instead, we will make sure that the exploration keeps going forever with positive probability by using the trick that a successful induction may work by requiring more than what is minimally necessary: not only will the exploration keep going, but it will grow fast. The argument will rely on first and second moment estimates for the growth in each step of the exploration.

\begin{theorem}[General animals on free products]\label{t.free}
If $\graph$ is a nonamenable  unimodular transitive graph that is obtained as a {free product} $\graph_1 \star \graph_2$, then the Poisson zoo with \emph{any} lattice animal measure $\E_\nu |H|^2=\infty$ satisfies $\lambda_c=0$.
\end{theorem}

For the definition of the free product of transitive graphs, see Definition~\ref{def:free_product}, and for two examples, Figure~\ref{f.free}. The simplest example is a regular tree, and our result is new and interesting already there. The proof is again via an exploration process, with the main idea being that, even without an exact understanding of how new pieces can touch an already explored part of the cluster (which was provided by the notion of  random walk capacity in the case of worms), the free product structure provides enough independent tries in disjoint areas of the graph for a branching process argument. One subtlety here is that the two factors in the free product can behave very differently, and the infinite second moment of $\nu$ may correspond to growth only in one of the factors. We will handle this asymmetry by sprinkling.

\vskip 0.2 cm
\begin{figure}[htbp]
\SetLabels
(-0.1*0.2){$\Z_3 \star \Z_4$}\\
(1.05*0.2){$\Z \star \Z$}\\
\endSetLabels
\centerline{
\AffixLabels{
\includegraphics[width=3.3 in]{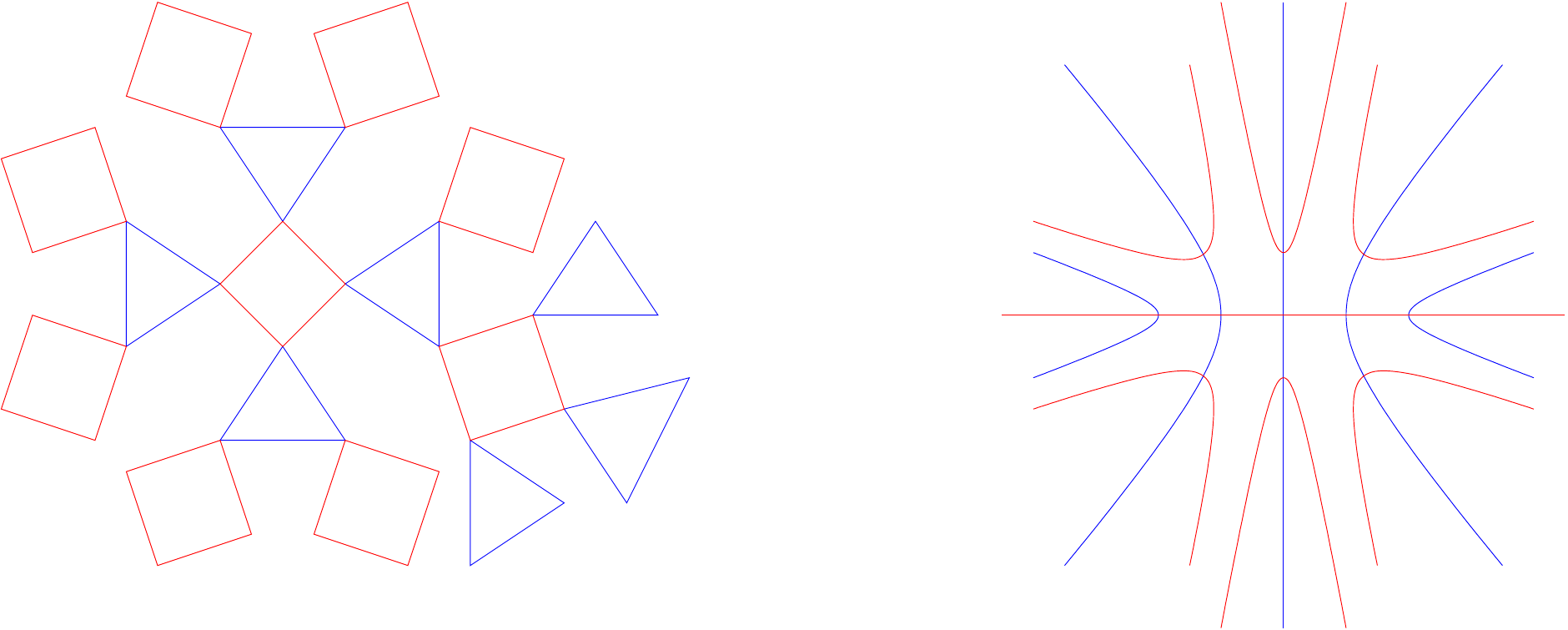}
}}
\caption{Two examples of free products of graphs.}
\label{f.free}
\end{figure}

The simplest joint generalization of our two theorems would be the following.

\begin{question}[General animals on general nonamenable groups]\label{q.general}
If  $\graph$ is \emph{any} nonamenable unimodular transitive graph, is  $\lambda_c=0$ for \emph{every} lattice animal measure with  $\E_\nu |H|^2=\infty$? 
\end{question}

For Bernoulli and many other percolation processes on nonamenable groups, besides the appearance of infinite clusters, there is typically a second important phase transition, which we can also define for the Poisson zoo:
$$
\lambda_u=\lambda_u(\graph,\nu) := \inf\left\{ \lambda > 0 :  \P\big(S_\nu^\lambda \textrm{ has a unique infinite cluster}\big)>0  \right\}.
$$
In the case of Bernoulli percolation (site or bond, does not matter), a well-known conjecture of Benjamini and Schramm \cite{BenjaminiSchramm1996} is that any nonamenable Cayley graph $\graph$ has $0< p_c(\graph) < p_u(\graph) \leq 1$; it has been proved in \cite{PSN,Thom} that any finitely generated nonamenable group $\grp$ has a finite generating set $S$ such that the Cayley graph $\graph=\Cay(\grp,S)$ satisfies this. Furthermore, one-ended Cayley graphs $\graph$ (meaning that removing any finite subset of the vertices results in a single infinite connected component) should always satisfy $p_u(\graph) < 1$, proved for finitely \emph{presented} ones \cite{BaB, Timar:cut}. 

For Bernoulli percolation and, more generally, any insertion tolerant invariant percolation, it was proved by Lyons and Schramm \cite{LySch} that infinite clusters, whenever they exist, are {\bf  indistinguishable} by $\grp$-invariant properties (e.g., either all infinite clusters are recurrent, or all of them are transient), and, as a consequence of this, uniqueness monotonicity holds: for every $p > p_u$, there is a unique infinite cluster. The Poisson zoo is not quite insertion tolerant, but is very close to it, and indistinguishability of infinite clusters is indeed proved in \cite{Damis}, also implying uniqueness for any $\lambda > \lambda_u(
\nu)$. 

Regarding the existence of an intermediate regime of non-unique infinite clusters, it was proved by Bowen \cite{Bowen2019}, motivated by the so-called dynamical von Neumann-Day problem \cite{GaboriauLyons}, that for any nonamenable Cayley graph and any $\veps>0$ there exists a worm measure $\nu$ with intensity $\lambda$ (called finitary interlacements there) such that the density is $\lambda \, \E_\nu |H| < \veps$ and there are infinitely many infinite clusters. 

However, there exist nonamenable Poisson zoos where there are no three phases, just one. A comment from Tom Hutchcroft, together with the results of \cite{RathRokob2022} and the indistinguishability of infinite clusters, easily imply the following, interesting for $d\ge 3$:

\begin{proposition}\label{p.immuniq}
There is a Poisson zoo $\nu$ on $G=\tree_\degree\times \Z^5$ with $\E_\nu|H|<\infty$ and $\lambda_u(G,\nu)=0$.
\end{proposition}

How widespread is this phenomenon of immediate uniqueness? This question brings us to a small detour, the next subsection,  concluding in Question~\ref{q.unique}.

\subsection{A naive motivation: groups with fixed price 1}\label{ss.cost}

Besides Question~\ref{q.general}, another motivation for this work (even though we will not make here real progress in this direction) comes from the theory of {\bf measurable cost} for groups, Cayley graphs, and probability measure preserving (p.m.p.)~group actions, introduced in \cite{Levitt} and studied by Gaboriau in several influential papers; see \cite{Gabo:free, Gabo:ICM}. It is an ergodic theoretic or probabilistic version of the rank of a group (which is the minimal number of generators required). Let us give here just the following two definitions, for Cayley graphs $\graph=\Cay(\grp,S)$ given by a finite generating set $S$ of $\grp$:
\begin{equation}\label{e.cost}
\cost(\grp,S) := \frac12 \inf \left\{ \E_\mu[\deg(o)] \,\left|\; 
\parbox{19em}{$\mu$ { is a }$\grp$-invariant probability measure on\\
connected spanning subgraphs of $\Cay(\Gamma,S)$} \right. \right\};
\end{equation}
\begin{equation}\label{e.costFIID}
\cost^*(\grp,S) := \frac12 \inf \left\{ \E_\mu[\deg(o)] \,\left|\; 
\parbox{19em}{$\mu$ is a factor of i.i.d.~(FIID) measure on\\
connected spanning subgraphs of $\Cay(\Gamma,S)$} \right. \right\}.
\end{equation}
Here, {\bf factor of i.i.d.~(FIID)} means that the measure $\mu$ is a $\grp$-equivariant measurable function $\psi$ of i.i.d.~variables $\xi=\{ \xi_v \sim \mathsf{Unif}[0,1] : v\in V(\graph) \}$ or $\{\xi_e \sim \mathsf{Unif}[0,1] : e\in E(\graph)\}$:
$$\psi(\xi) \sim \mu, \hskip 2 cm  \gamma \cdot \psi(\xi) = \psi(\gamma\cdot \xi), \quad  \forall \gamma \in \grp,$$ 
with the usual action of $\grp$ on a configuration $\xi$, namely, $(\gamma\cdot\xi)(e) = \xi(\gamma^{-1}(e))$ for every $e\in E$, or similarly for $v \in V$. Equivalently (but somewhat loosely), one has a measurable (i.e., locally approximable) coding map that decides the status of each edge $e$ (or vertex $v$) from the i.i.d.~variables viewed from $e$ (or $v$). See, e.g.,~\cite{Lyons} for an introduction to FIID processes. An example is taking i.i.d.~$\Bernoulli(p)$ bond percolation, then deleting every finite cluster. Another one is the Wired Minimal Spanning Forest, where, given $\{\xi_e \sim \mathsf{Unif}[0,1] : e\in E\}$, we delete the edge with the largest label in every cycle, including ``generalized cycles'' going through infinity; see, e.g., \cite[Section 11]{LyonsPeres2016}. An example of a FIID {\it site} percolation is the set $S^\lambda_\nu$ covered by our Poisson zoo:~(\ref{e.zoo}) is clearly a FIID  construction. 

The reason for singling out FIID measures in~(\ref{e.costFIID}) is that a finer version of cost can be defined for p.m.p.~group actions, so that the cost of a group is the infimum cost among all its free actions, while a theorem of Ab\'ert and Weiss \cite{AW} implies that FIID actions have the supremum cost. The {\bf fixed price} question of Gaboriau \cite{Gabo:free} is if all free actions of any given group actually have the same cost, and if it is also independent of the generating set: for any generating set $S$, we might have $\cost(\grp,S) = \cost^*(\grp,S)=\cost(\grp,\grp)$. 

It is well-known that infinite amenable Cayley graphs have fixed price 1; this follows, e.g., from \cite[Theorem 5.3]{BLPS:GAFA}. It was proved in \cite{Gabo:free} that free groups on $d$ free generators, with any finite generating set $S$, have fixed price $d$. But there also exist nonamenable Cayley graphs with fixed price 1: for instance,  $\cost^*(\grp_1 \times \grp_2,S)=1$ if either group is infinite. It is quite open exactly which groups have cost 1 or fixed price 1; another well-known question of Gaboriau is whether they are characterized by their {\bf first $\ell^2$-Betti number} being 0, which means having no non-constant harmonic functions with finite Dirichlet energy \cite[Section 10.8]{LyonsPeres2016}. As an important step in this direction, Hutchcroft and Pete \cite{HuPe} proved that infinite Kazhdan (T) groups have cost 1, but their construction is not at all FIID, hence fixed price remains open. In another recent breakthrough \cite{FMW}, using Ideal Poisson-Voronoi Tessellations, it is proved that lattices in higher rank Lie groups have fixed price 1. 

It was observed in \cite[Proposition 2.1]{HuPe} that if $\Cay(\grp,S)$ has, for any $\veps>0$, an invariant site or bond percolation with density at most $\veps$ and a unique infinite cluster, then $\cost(\grp,S)=1$. Moreover, if this {\bf arbitrarily sparse unique infinite cluster} can be given in a FIID way, then $\cost^*(\grp,S)=1$. Since the Poisson zoo is always a FIID measure, and, as our above results show, it often produces infinite clusters already at arbitrarily small densities, sometimes even a unique one, it makes sense to ask the following, aiming to generalize Proposition~\ref{p.immuniq}:

\begin{question}[FIID sparse unique cluster]\label{q.unique}\ 
\begin{enumerate}
\item[{\bf (a)}]  Give interesting examples of nonamenable Cayley graphs $\graph$ with FIID percolations with a sparse unique infinite cluster. 
\item[{\bf (b)}] In particular, does every Cayley graph with first $\ell^2$-Betti number being 0 admits, for \emph{every} $\veps>0$,  some lattice animal measure $\nu$ with $\E_\nu |H| < \infty$ such that $\lambda_u(\graph,\nu) \, \E_\nu |H| < \veps$? Or even a single lattice animal measure with $\lambda_u(\graph,\nu)=0$?
\end{enumerate}
\end{question}

Having FIID sparse infinite clusters, without the requirement of uniqueness, is not difficult:
\begin{itemize}
\item Consider i.i.d.~{Bernoulli percolation} on a Cayley graph at $p=p_c+\veps$. Assuming $\theta(p_c)=0$, proved for nonamenable Cayley graphs in \cite{BLPS:GAFA}, the density of infinite clusters is small. Delete all finite clusters.
\item Consider the {Wired Minimal Spanning Forest}, mentioned above. This is infinitely many one-ended trees for any nonamenable Cayley graph \cite[Theorem 11.12]{LyonsPeres2016}. Prune the leaves, repeatedly, $N$ times. For large enough $N$, the overall density will be small.
\item Take random interlacements \cite{Sznitman2010, DrewitzRathSapozhnikov2014}, the canonical Poisson point process of bi-infinite simple random walk trajectories, at low intensity. The set of covered points is a FIID process \cite{BorbenyiRathRokob}, for non-trivial reasons: as opposed to worms, a trajectory here does not have a starting point where the random choices for the entire trajectory could be made locally.
\end{itemize}
However, on nonamenable Cayley graphs, these constructions inherently give infinitely many infinite clusters: for Bernoulli percolation, this is the $p_c < p_u$ conjecture; for interlacements, this is proved in \cite{TeixeiraTykesson2013}. The Poisson zoo has a chance to yield unique sparse infinite clusters on many nonamenable groups, Proposition~\ref{p.immuniq} hopefully being only a first example.

Very recently, the preprint \cite{GrebReck} has strengthened \cite{FMW} to produce FIID sparse unique infinite clusters on lattices in higher rank Lie groups.

\subsection{Further related works}\label{ss.litera}

The introduction of the general Poisson zoo in \cite{RathRokob2022} was inspired by several Poisson soup percolation models, as follows. 

The first one is the Poisson Boolean model with random radii in $\R^d$, see, e.g., \cite{MeesterRoy1996, Gouere2008, DeMu, DeTa} and the many references therein, and with a fixed radius in hyperbolic space \cite{Tykesson2007, Tykesson2009}, where results similar to nonamenable Bernoulli percolation have been shown, such as $0<\lambda_c<\lambda_u<\infty$. Percolation results in spaces beyond $\R^d$ seem to stop here; there are some results about visibility of the ideal boundary in the vacant set \cite{BenjaminiJonassonSchrammTykesson2009, TykessonCalka2013}.

The second main family of models is random walk loop soups \cite{Lupu,ChangShapoznikov2016}, especially interesting for their relationships to level sets of the Gaussian Free Field, and the related models of finitary interlacements (a special case of the worms model), studied in \cite{Bowen2019, PYZ, CPZ22}.

Earlier results on permanent supercriticality, meaning $\lambda_c=0$, were \cite[Theorem 1.2]{TeixeiraUngaretti2017} on {ellipses percolation} in $\R^2$, and \cite[Remark 1.1 and Claim 1.2]{Chang2021} on Bernoulli hyper-edge percolation on $\Z^d, \, d \geq 3$.  In the recent preprint \cite{convexgrain}, for the Poisson Boolean model in $\R^d$ with quite general random convex shapes, criteria for $\lambda_c=0$ (called ``robustness'' there) were given. 

One may also consider Poisson soups of infinite objects. The {\bf random interlacements process} on $\Z^d$, already mentioned above, was introduced by Sznitman \cite{Sznitman2010}, as the local distributional limit of simple random walk on a $d$-dimensional torus $(\Z / N \Z)^d, \, d \geq 3$, where we run the walk up to times comparable to the volume of the torus, and let $N \to \infty$. The limit can be described as the trace of a Poisson point process on the space of labeled bi-infinite random walk  trajectories modulo time-shift, and it is a canonical object related to capacity. This description also works on any transient edge-weighted graph \cite{Teixeira2009}. A percolation phase transition that makes sense in this setting is whether all these infinite trajectories form a single infinite component, and it was proved in \cite{TeixeiraTykesson2013} that this is the case for all positive intensities if{f} the underlying graph is amenable. The harder direction here is the existence of a non-uniqueness phase in the nonamenable case, proved by dominating the interlacements set by a certain branching random walk. Another phase transition to study is percolation in the vacant set, a recent breakthrough on $\Z^d$ being \cite{DukeEquality}. To our nonamenable setting, more relevant is the proof of the existence of a supercritical phase for percolation in the vacant set in \cite[Theorem 5.1]{Teixeira2009}, using a branching process idea, somewhat similarly to our Theorem~\ref{t.free}. It was shown recently in \cite{SapoMu} that transient amenable transitive graphs always have at most one vacant infinite component. Using a coupling with the Gaussian Free Field, it is proved in \cite{DPR} that certain graphs, including all transient Cayley graphs of polynomial growth, have a supercritical phase for vacant percolation. To our knowledge, it is not known for general transient transitive graphs that there is a supercritical phase, nor that nonamenable graphs have a non-uniqueness phase. 
 
Another natural Poisson soup of infinite objects is the {\bf Poisson cylinder model}, defined in \cite{TykessonWindisch2012} as a Poisson point process on the space of lines in $\R^d$, where a multiplicative factor of the intensity measure determines the density of the lines and where every line in the process is taken as the axis of a bi-infinite cylinder of radius $1$. Phase transitions for vacant percolation were proved here. Connectedness of the occupied set for all positive intensities was proved in \cite{BromanTykesson2016}. One can also generalize the model to hyperbolic spaces $\mathbb{H}^d$ by changing the lines of $\R^d$ to geodesic, and ask if the permanent connectedness is a result of $\R^d$ being a ``small'' space. Indeed, it was proved in \cite{BromanTykesson2015} that in any $\mathbb{H}^d$, $d \geq 2$, there exists some critical intensity below which there are infinitely many infinite clusters, but above which there is a unique one. Let us note that the proof of the existence of a non-unique phase again uses a branching process construction. 

The abundance of independence in a Poisson point process simplifies many arguments, but one may wonder how much we lose by considering such processes instead of more traditional $\{0,1\}$-valued invariant percolations. In a recent paper \cite{FrosstromGantertSteif2024}, the authors call a $\{ 0,1 \}$-valued process on some (discrete) set {\bf Poisson representable} if it can be constructed as a union of subsets coming from a Poisson process on the collection of subsets, and prove, e.g., that all positively associated Markov chains on $\{0,1\}^\Z$ are Poisson representable, while the Ising model on $\Z^d$, $d\ge 2$, and on the complete graph, for certain temperatures, are not.

We close this subsection with the question that we find the most attractive one in the amenable setting. It is not addressed in~\cite{convexgrain}, being a borderline case there.

\begin{question}[Random rays in the plane]\label{q.Z2R2}
Let the lattice animal measure $\nu_0$ be a straight line segment of random length, started at $o$, in a uniform random direction. On $\Z^2$, this means one of the four lattice directions with probability $1/4$ each, but one could also look at a Poisson point process on $\R^2$ as starting points, directions given by $\mathsf{Unif}[0,2\pi)$ variables.  Does $\E_{\nu}|H|^2=\infty$ suffice for $\lambda_c(\nu)=0$?
\end{question}

A growth process for physical networks built on random rays in $\Z^2$ was defined in \cite{physnetw}, as a member of a large family of models. It was argued non-rigorously and by simulations that this model has a mean-field degree exponent, vaguely suggesting that also for Question~\ref{q.Z2R2} we could observe the ``mean-field'' behaviour, meaning that the finiteness of the second moment would decide about the phase transition.

\subsection{Structure of the paper}
\label{subsection:structure_of_the_paper}

The rest of the paper is organized as follows. 

In Section~\ref{section:setup}, we collect all the necessary definitions and notions needed either for the precise statements or the proofs of our main theorems. We define amenability, unimodularity, and the free product of transitive graphs in Subsection~\ref{ss.graphs}. We provide some basic definitions and lemmas on random walk capacity in Subsection~\ref{ss.rw}. The basics of the Poisson zoo, after the rushed Introduction, are given carefully in Subsection~\ref{subsection:poisson_zoo_on_transitive_graphs}. The proof strategies for Theorem~\ref{t.free} (general animals on free products) and Theorem~\ref{t.worms} (worms on general nonamenable unimodular graphs) are explained in Subsections~\ref{subsection:poisson_zoo_on_free_products} and~ \ref{subsection:random_length_worms_model_introduction}, respectively. The basic properties of Poisson point processes that will be used in the arguments are summarized in Subsection~\ref{subsection:point_processes_of_rooted_animals}.

In Section \ref{section:bounds_on_fat} we will study the first and second moments of the ``fat'', which is the newly found occupied set in each step of our exploration process, in either of the two settings. In Section~\ref{section:exploration_processes} we analyze the growth of our exploration processes, proving the main two theorems. Finally, in Section~\ref{s.immuniq}, we prove Proposition~\ref{p.immuniq} on immediate uniqueness for a zoo on $\tree_\degree \times \Z^5$.

\subsection{Acknowledgments}

We are grateful to Tom Hutchcroft for inspiring Proposition~\ref{p.immuniq}, and to Bal\'azs R\'ath for several useful conversations. Three people so far have independently suggested that the Poisson zoo should be renamed the Aquarium; we do not thank them \smiley{}.

This work was supported by the ERC Synergy Grant No. 810115--DYNASNET and by the Hungarian National Research, Development and Innovation Office, OTKA grant K143468.

\section{Preliminaries}
\label{section:setup}

Most of the definitions and theorems introduced here can be found in the union of \cite{LyonsPeres2016, Pete2023, Woess2000}. 

Some general notation that we will use are $\N := \{ 0, 1, \ldots  \}$ and $\N_{> 0} := \{ 1, 2, \ldots \}$ and their obvious modifications for $\R$. If we are given two real numbers $a, b \in \R$, then $a \wedge b := \min \{ a, b \}$ and $a \vee b := \max \{ a, b \}$.

\subsection{Graph theory definitions}\label{ss.graphs}

We will work on undirected, locally finite, connected, infinite graphs $\graph=(V,E)$, sometimes with a distinguished vertex $\origin \in V(\graph)$, referred to as the origin. We will use the shorthand $x \edge y$ to denote that the vertices $x, y \in V(\graph)$ are neighbours: $\{ x,y \} \in E(\graph)$.

For finite subsets we will use the notation $A \fsubset V(\graph)$. If $A \subseteq V(\graph)$, then the interior and exterior vertex boundaries of $A$ are 
\begin{equation}
	\label{def:vertex_boundaries}
	\intboundary A := \{ x \in A \, : \, \exists \, y \in A^c, \, x \edge y \}, \quad \text{ and } \quad
	\extboundary A := \{ x \in A^c \, : \, \exists \, y \in A, \, x \edge y \}.
\end{equation}
By the discrete closure of the set $A$ we mean $\overline{A} := A \cup \extboundary A$.

Distances with respect to the usual graph metric will be denoted by $\dist_{\graph}(\cdot, \cdot)$. Balls are denoted by 
\begin{equation}
	\ball(x,R) := \{ y \in V(\graph) \, : \, \dist_{\graph}(x,y) \leq R \}.
\end{equation}

An automorphism of $\graph$ is a bijection $V(\graph)$ to itself that preserves adjacency. The group of all automorphisms of $\graph$ is denoted by $\Aut(\graph)$. 
The graph is called transitive if $\Aut(\graph)$ acts transitively on $V(\graph)$. In such a case, the degrees of the vertices are constant, denoted by $\degree := \degree(\graph)$. Importantly to our work, edge transitivity does not follow; see, e.g., the first example on Figure~\ref{f.free}. The most important source of transitive graphs are Cayley graphs $\graph=\Cay(\grp,S)$ with some symmetric generating set $S = S^{-1} \subset \grp$, where $E(\graph):=\{ (x,x s) : x \in \grp, s\in S \}$. Here, $\grp$ itself acts transitively by $\gamma\cdot x := \gamma x$.

Recall the definition of the exterior vertex boundary from \eqref{def:vertex_boundaries}. By the Cheeger constant of the (infinite) graph $\graph$ we mean
\begin{equation}
	\label{def:Cheeger_constant}
	\Cheeger(\graph) :=
	\inf \left\{
		\frac{| \extboundary A| }{|A|} \, : \,
		A \subset V(\graph), \, 
		0 < |A| < \infty	
	\right\}.
\end{equation}
Intuitively, $\Cheeger(\graph)$ is an indicator of ``bottlenecks'' in the graph. The same notion could be defined using a different notion of boundary (interior vertex boundary or edge boundary), but since our graphs will always be locally finite and transitive, the difference between these notions is within constant multipliers. And this is what will matter: we call a graph $\graph$ {\bf amenable} if $\Cheeger(\graph) = 0$, and nonamenable if $\Cheeger(\graph) > 0$. 

The most notable example of amenable graphs are the Euclidean lattices $\Z^d$, $d \geq 1$.
On the other hand, it is easy to prove \cite[Exercise 6.1]{LyonsPeres2016} that for the $\degree$-regular tree $\tree_{\degree}$ we have $\Cheeger( \tree_{\degree} ) = \degree - 2 $, and hence it is nonamenable for $\degree > 2$.

The importance of the following mass conservation principle in percolation theory was noted in \cite{BLPS:GAFA}, following \cite{Haggstrom}. It will be crucial also for us.

\begin{definition}[Unimodularity]
	\label{def:unimodularity}
	Let $\grp$ be a transitive subgroup of automorphism of the graph $\graph$. If we have the {\bf mass transport principle (MTP)} 
	\begin{equation}
		\label{eq:MTP}
		\sum_{y \in V(\graph)} f(y,x) =
		\sum_{y \in V(\graph)} f(x,y)
	\end{equation}
	for any diagonally $\grp$-invariant $f  :  V(\graph) \times V(\graph) \longrightarrow [0, \infty]$ function and for any $x \in V(\graph)$, then we say that the action of $\grp$ is {\bf unimodular}.	If the action of $\Aut(\graph)$ is unimodular on $\graph$, then we simply say that $\graph$ is unimodular.
\end{definition}

Every Cayley graph is unimodular, since, for every diagonally $\grp$-invariant function $f$, 
$$\sum_{y \in \grp} f(y,x)=\sum_{y\in \grp} f(xy^{-1}y,xy^{-1}x)=\sum_{z\in\grp} f(x,z),$$ 
as $y\mapsto z:=xy^{-1}x$ is a bijection. Moreover, every amenable transitive graph is unimodular, as well (see, e.g., \cite[Proposition 8.13]{LyonsPeres2016}).

However, in the nonamenable case, there are non-unimodular transitive graphs. The simplest such graph is Trofimov's grandparent graph, obtained from a $\degree$-regular tree by distinguishing an end (an infinite ray starting at any fixed vertex), thinking of the neighbour of each vertex $x$ towards the end as the parent of $x$, then connecting every vertex to its (unique) grandparent; see \cite[Example 7.1 and Section 8.2]{LyonsPeres2016}. The automorphism group of this graph is just the subgroup of $\Aut(\tree_\degree)$ that fixes that distinguished end, which does not satisfy MTP: if $f(x,y)$ is 1 when $y$ is the parent of $x$, otherwise 0, then the outgoing mass is 1, while the incoming is $\degree$.

\medskip

Next, we give a detailed description of the free product of transitive graphs, following \cite[Chapter II. 9.C]{Woess2000}. Before the definition, let us note that we only deal with the free product of two graphs here, since that will be enough for our purposes. However, one can easily extend the definition for countably many terms (as it is done in \cite{Woess2000}).

\begin{definition}
	\label{def:free_product}
	Suppose that we are given two non-trivial connected transitive graphs $\graph_i = (V(\graph_i), E(\graph_i))$, $i = 1,2$, with disjoint vertex sets and some distinguished vertices $\origin_i$, $i = 1,2$. The {\bf free product}
	\begin{equation}
		\label{eq:free_product}
		\graph := (V(\graph), E(\graph)) :=
		\graph_1 \star \graph_2
	\end{equation}
	is constructed in the following way.
	The vertices of $\graph$ are all the finite {words}:
	\begin{equation}
		\label{def:vertices_of_free_product}
		V(\graph) :=  \left\{
		x_1 \ldots x_n \, \left| \, 
		\parbox{17em}{
			$n \in \N; \, x_j \in V(\graph_1) \cup V(\graph_2) \setminus \{ \origin_1, \origin_2 \}; \\
			x_k \in V(\graph_i) \implies x_{k+1} \notin V(\graph_i), \, i=1,2$
		}		
		\right. \right\} 
		\cup
		\left\{ \origin \right\}
	\end{equation}
	where $\origin$ denotes the {empty word}, the distinguished vertex of $\graph$, obtained by identifying each $\origin_i$, $i = 1,2$ with $\origin$.
	If we are given words $x = x_1 \ldots x_k \in V(\graph)$ and $y = y_1 \ldots y_m \in V(\graph)$ with $x_k \in V(\graph_i)$ and $y_1 \notin V(\graph_i)$, then $xy$ stands for the concatenation as words. In addition, for all $x \in V(\graph)$ we define $x \origin = \origin x = x$.
	Furthermore, adjacency in $\graph$ is given as follows: for $i = 1,2$, if $\{ x,y \} \in E(\graph_i)$, then $\{ zx, zy \} \in E(\graph)$ for all $z = z_1 \ldots z_{\ell} \in V(\graph)$, with $z_{\ell} \notin V(\graph_i)$.
\end{definition}

In words, we take copies of $\graph_i$, $i=1,2$ and glue them together by identifying $\origin_1$ and $\origin_2$ into a single vertex $\origin$. 
Then, inductively, at each vertex $v$ of $\graph_i$ attach a copy of $\graph_j$, $j \neq i$, in such a way that $v$ is identified with $\origin_j$ from the the new copy of $\graph_j$.
So, from any vertex of the product one can either go in the direction of the component $\graph_1$ or $\graph_2$, and these two subgraphs only meet in the given vertex.
Consequently, it is also meaningful to introduce the $\graph_i$-neighbourhood of a vertex $x \in V(\graph)$ as
\begin{equation}
	\label{def:free_product_term_neighbourhood}
	\mathcal{N}_i(x) := \left\{
	y \in V(\graph) \, : \,
	x \edge y \text{ in } \graph_i 
	\right\}
\end{equation}
for $i = 1,2,$.
That is, $\mathcal{N}_i(x)$ is the subset of those neighbours of $x$ that can be reached using a $\graph_i$-edge.

One well-known example of such a free product is the $\degree$-regular tree $\tree_{\degree}$, which is the free product of a single edge and $\tree_{\degree-1}$; or, iterating this, the $\degree$-wise free product of a single edge. In particular, $\tree_4$ is the free product of $\Z$ with itself, as on the second picture of Figure~\ref{f.free}. 

The most relevant properties of a free product can be summarized as the following collection of observations.
As their proofs are obvious from the definition, they are omitted.

\begin{claim}
	\label{claim:properties_of_free_product_of_graphs}
	\begin{enumerate}
		\item[]
		\item The free product of two connected, locally finite and transitive graphs is again connected, locally finite and transitive.
		\item Every vertex of a free product is a cutpoint: removing it results in at least two disjoint infinite subgraphs.
		\item Given any cycle in a free product, all of its edges are either from the first or the second component. That is, the free product cannot introduce new cycles.
	\end{enumerate}
\end{claim}

\subsection{Random walks, spectral radius, capacity}\label{ss.rw}

Let us denote the space of all $V(\graph)$-valued infinite nearest neighbour trajectories indexed by nonnegative integers by $W := W(\graph)$. This countable space can be endowed with a natural $\sigma$-algebra $\mathcal{W}$, the one generated by the canonical coordinate maps.

\begin{definition}
	For $x \in V(\graph)$, let $P_x$ denote the law of the simple random walk $( X(n) )_{n \in \N}$ on $\graph$ which starts from the vertex $x$, i.e., $X(0) = x$. Let us also denote the corresponding expectation by $E_x$. The law $P_x$ can be considered as a probability measure on $(W, \mathcal{W})$.
\end{definition}


On any graph, simple random walk is a reversible Markov chain, with stationary measures proportional to the degrees. The chain is described by the transition probabilities
\begin{equation}
	\label{def:transition_probabilities_of_random_walks}
	p_n(x,y) := P_x(X(n) = y), \qquad x, y \in V, \, n \in \N.
\end{equation}

On a nonamenable graph, one expects that these transition probabilities decay fast with $n$. This is indeed the case, and to argue why, let us introduce the notion of spectral radius of a graph; note that it does not matter which starting vertex we choose, due to connectedness.

\begin{definition}
	\label{def:spectral_radius}
	The {\bf spectral radius} of the graph $\graph$ is the quantity
	\begin{equation}
		\label{eq:spectral_radius}
		\spectral(\graph) := 
		\limsup_{n \to \infty}
		p_n(\origin, \origin)^{1 \slash n}.
	\end{equation}
\end{definition}

\noindent
Obviously, $\spectral(\graph) < 1$ means that $p_n(\cdot, \cdot)$ decreases exponentially fast. The celebrated result of Kesten, Cheeger, Dodziuk and Mohar (see \cite[Theorem 10.3]{Woess2000} or \cite[Theorem 6.7]{LyonsPeres2016} or \cite[Theorem 7.3]{Pete2023}) tells us that this property characterizes nonamenability, defined at~(\ref{def:Cheeger_constant}):
\begin{equation}
	\label{eq:relation_of_spectral_radius_and_Cheeger_constant}
	\spectral(\graph) < 1 
	\quad \Longleftrightarrow \quad
	\Cheeger(\graph) > 0.
\end{equation}

Some random stopping times we will use extensively throughout the analysis of the random length worms model are as follows. For any $w \in W$ and $K \fsubset V(\graph)$, let
\begin{align}
	\label{nota:entrance_time}
	T_K(w) & := \inf \{ \, n \geq 0 \, : \, w(n) \in K\, \}, \quad \text{first entrance time}, \\
	\label{nota:hitting_time}
	T_K^+(w) & := \inf \{ n \geq 1 \, : \, w(n) \in K \}, \quad \text{first hitting time},
\end{align}
where we define $\inf \emptyset  = + \infty$. 
To simplify notation, the case of sets that consist of only one vertex will be denoted by $T_x(w) := T_{\{ x \}}(w)$ and $T_x^+(w) := T_{\{ x \}}^+(w)$.
With the latter in hand, let us recall that the graph $\graph$ is called {\bf recurrent} if 
\begin{equation}
	\label{def:recurrence_of_graphs}
	P_{\origin} \left( T_{\origin}^+(X) < + \infty \right) = 1,
\end{equation}
and {\bf transient} otherwise, which does not depend on the chosen vertex for a connected graph, and hence is a well-defined property of a graph.

It is an easy fact that a nonamenable graph is always transient: (\ref{eq:relation_of_spectral_radius_and_Cheeger_constant}) implies that 
$$E_o \big|\{n \geq 0 : X(n)=o\}\big| = \sum_{n\ge 0} p_n(o,o) < \infty,$$ 
hence the number of returns to $o$ is almost surely finite, hence~(\ref{def:recurrence_of_graphs}) fails. The next theorem, a more general version of which can be found in \cite[Propositon 9.3]{Pete2023} or \cite[Proposition 6.9]{LyonsPeres2016}, says much more.

\begin{theorem}
	\label{thm:positive_constant_speed_of_SRW_on_nonamenable_transitive_graphs}
	Given a nonamenable and transitive graph $\graph$, there exists a constant $c(\graph) > 0$ such that 
	\begin{equation}
		\label{eq:positive_constant_speed_of_SRW_on_nonamenable_transitive_graphs}
		P_{\origin} \left( \liminf_{t \to \infty} \frac{\dist_{\graph}(\origin, X(t))}{t} \geq c(\graph)  \right) = 1\,.
	\end{equation} 
\end{theorem}

\noindent
In words, for nonamenable (and transitive) graphs, the distance between the start and the end points of a random walk running for some long enough time is almost surely comparable to the number of steps it took. Obviously, this distance provides a lower bound on the size of the trace of the walker, hence one can derive the following corollary.

\begin{corollary}
	\label{coro:number_of_steps_and_size_of_trace_of_SRW_are_comparable_on_nonamenable_graphs}
	Given a nonamenable and transitive graph $\graph$, for any $\veps > 0$, there exists $c(\veps) > 0$ such that for any $t > 0$ we have
	\begin{equation}
		\label{eq:number_of_steps_and_size_of_trace_of_SRW_are_comparable_on_nonamenable_graphs}	
		P_{\origin} \Big(
			\big| X[0, t] \big| \geq c(\veps) \, t
		\Big) > 1 - \veps.
	\end{equation}
\end{corollary}

Given $K \fsubset V(\graph)$ and $x \in K$, let us define the {\bf equilibrium measure} $e_K(x)$ of $x$ with respect to $K$ by
\begin{equation}
	\label{def:equilibrium_measure}
	e_K(x) := P_x \left( T_K^+ = \infty \right),
\end{equation}
and the {\bf capacity} of $K$ by
\begin{equation}
	\label{def:capacity}
	\capacity(K):=\sum_{x \in K} e_K(x).
\end{equation}
Obviously, the equilibrium measure is concentrated on the interior vertex boundary of a set. Due to this and the reversibility of the simple random walk, heuristically speaking, the capacity of a set measures the visibility of the set from the point of view of a random walk coming from infinity: the larger the capacity, the more visible the set is.
As an example, it follows immediately from the definitions that the capacity of a one-element set is
\begin{equation}
	\label{eq:capacity_of_one_vertex}
	\capacity(\{ x \}) = 
	P_{x}( \hitTime_x = \infty ) = 1 \slash \Green(x, x)
\end{equation}
for every $x \in V(\graph)$, where $\Green(x,y) := \sum_{n = 0}^{\infty} p_n(x,y)$ is {\bf Green's function}. 

Although there are many properties of  capacity that can be used to examine the behaviour of random walks, in our context we will be interested only in one of them. Namely, on nonamenable graphs, the capacity of a set is comparable to its size, which follows, for example, from \cite[Eq.~(4.4)]{Teixeira2009} or \cite[Lemma 2.1]{BenjaminiNachmiasPeres2011}.

\begin{lemma}[Capacity is linear]
	\label{lemma:capacity_bound}
	For any $K \fsubset V$ we have 
	\begin{equation}
		\label{eq:capacity_bound}
		(1 - \spectral(\graph)) \cdot |K| \leq
		\capacity(K) \leq
		|K|,
	\end{equation}
	where recall that $\spectral(\graph)$ denotes the spectral radius of the graph. In particular, by \eqref{eq:relation_of_spectral_radius_and_Cheeger_constant}, on a nonamenable graph $\graph$ the capacity and the cardinality of a finite vertex set are comparable.
\end{lemma}

In the proof of Theorem~\ref{t.worms}, we will also use the following restricted version of capacity.

\begin{corollary}
	\label{coro:lower_bound_on_partial_square_term_capacity}
	Let $A \fsubset V(\graph)$ be finite and connected, $B \subseteq \extboundary A$. Then we have
	\begin{equation}
		\label{eq:lower_bound_on_partial_square_term_capacity}
		\sum_{y \in B} P_y^2 \left( T_{\overline{A}}^+ = +\infty \right) \geq
		(1 - \spectral(\graph))^2 \cdot \left| \overline{A} \right| - ( | \extboundary A| - |B|).
	\end{equation}
\end{corollary}

\begin{proof}
	Since $B \subseteq \extboundary A$, the sum can be ``decomposed'' as the subtraction
	\begin{equation*}
		\sum_{y \in B} P_y^2 \left( T_{\overline{A}}^+ = + \infty \right) =
		\sum_{y \in \extboundary A} P_y^2 \left( T_{\overline{A}}^+ = + \infty \right) - 
		\sum_{y \in \extboundary A \setminus B} P_y^2 \left( T_{\overline{A}}^+ = + \infty \right).
	\end{equation*}
 For the first term, we have
	\begin{equation*}
		\sum_{y \in \extboundary A} P_y^2 \left( T_{\overline{A}}^+ = + \infty \right) \geq 
		\left( 1 - \spectral(\graph) \right)^2 \cdot \left| \overline{A} \right|,
	\end{equation*}
	since, by the Cauchy-Schwarz inequality, we can write and then rearrange the following:
	\begin{align*}
		& |\overline{A}| \cdot 
		\sum_{y \in \extboundary A} P_y^2 \left( T_{ \overline{A}}^+ = + \infty \right) = 
		\left( \sum_{x \in \overline{A}} 1^2 \right) \cdot \left( \sum_{y \in \overline{A}} P_y^2 \left( T_{ \overline{A}}^+ = + \infty \right) \right) \geq \\ & \qquad \geq 
		\left( \sum_{x \in \overline{A}} P_y \left( T_{\overline{A}}^+ = +\infty \right) \right)^2 = 
		 \capacity \left( \overline{A} \right)^2 \overset{\eqref{lemma:capacity_bound}}{\geq} 
		(1- \spectral(\graph))^2 \cdot |\overline{A}|^2.
	\end{align*}
	Meanwhile, for the second term, one can use the trivial upper bound
	\begin{align*}
		\sum_{y \in \extboundary A \setminus B} P_y^2 \left( T_{\overline{A}}^+ = + \infty \right) \leq 
		\left| \extboundary A \right| - \left| B \right|.
	\end{align*}
Subtracting the two bounds finishes the proof.
\end{proof}

\subsection{The Poisson zoo on transitive graphs}
\label{subsection:poisson_zoo_on_transitive_graphs}

We now introduce the Poisson zoo carefully.


\begin{definition}
	\label{def:lattice_animal}
	Given any vertex $x \in V(\graph)$, let $\latticeanimals_x$ be the set of finite, connected, nonempty subsets $H$ of $V$ that contain $x$, called the lattice animals rooted at $x$.
\end{definition}

Let $\grp \subgrp \Aut(\graph)$ be a subgroup of the automorphism group of the graph, which still acts transitively on $V$.
Obviously, the action of an automorphism $\varphi \in \grp$ can be extended to any set of vertices $H \subseteq V(\graph)$ as
$\varphi(H) := \left\{ \varphi(h) \, : \, h \in H \right\}$. Moreover, if $x,y \in V(\graph)$ and $\varphi \in \grp$ is such that $\varphi(x) = y$, then $\varphi$ defines a bijection between $\latticeanimals_x$ and $\latticeanimals_y$.

Given a vertex $x \in V(\graph)$, the stabilizer of $x$ under the action of $\grp$ is the subgroup
\begin{equation}
	\label{def:stabilizer_subgroup}
	\grp_{x} := \left\{ \varphi \in \grp \, : \, \varphi (x) = x \right\} \subgrp \grp.
\end{equation}

\begin{definition}
	\label{def:grp-rotation-invariant_prob_measure}
	We say that a probability measure $\nu_{\origin}$ on $\latticeanimals_{\origin}$ is ($\grp$-)rotation invariant if, for any $\varphi \in \grp_{\origin}$ and any $H \in \latticeanimals_{\origin}$, we have
$\nu_{\origin}( \varphi(H) ) = \nu_{\origin}(H)$.
\end{definition}

A class of examples is when the measure of an animal depends only on its cardinality.

Since $\grp$ acts transitively on $\graph$, for all vertices $x \in V(\graph)$ 
\begin{equation}
	\label{def:the_distinguished_automorphism_taking_origin_to_x}
	\text{there exists at least one \linebreak automorphism $\varphi_x \in \grp$ such that $\varphi_x (\origin) = x$.}
\end{equation}
As a consequence, if we are given a $\grp$-rotation invariant measure $\nu_{\origin}$ on $\latticeanimals_{\origin}$, we can push it forward with $\varphi_x$ to obtain a $\grp$-rotation invariant measure on $\latticeanimals_x$ for any $x \in V$:
\begin{equation}
	\label{def:translated_prob_measure}
	\nu_x := \nu_{\origin} \circ \varphi_{x}^{-1}.
\end{equation}
Moreover, if $\gamma \in \grp$ is such that $\gamma (\origin) = x$, then for any $H \in \latticeanimals_{\origin}$ we have $\gamma(H) \in \latticeanimals_{x}$ and $\nu_x( \gamma(H) ) = \nu_{\origin} ( \varphi_{x}^{-1} \gamma(H) ) = \nu_{\origin} (H)$, since $\varphi_x^{-1} \gamma \in \grp_{\origin}$ and $\nu_o$ is rotation-invariant. This also means that all automorphisms $\varphi_x$  in (\ref{def:translated_prob_measure}) give the same measure $\nu_x$.

\begin{definition}
	\label{def:poisson_zoo}
	Given the $\grp$-rotation invariant measure $\nu := \nu_{\origin}$ and some intensity parameter $\lambda \in \R_{>0}$, let us consider independent Poisson random variables $N_{x,H}^{\lambda} \sim \POI( \lambda \cdot \nu_x(H) )$ indexed by $x \in V(\graph)$ and $H \in \latticeanimals_x$.
	We say that $N_{x,H}^{\lambda}$ is the number of copies of the animal $H$ (rooted at $x$).
	We call the collection of random variables
	\begin{equation}
		\label{eq:poisson_zoo}
		\mathcal{N}_{\nu}^\lambda := 
		\left\{ 
			N^\lambda_{x,H} \, : \, x \in V(\graph), H \in \latticeanimals_x 
		\right\}
	\end{equation}
	the {\bf Poisson zoo} at level $\lambda$ (with measure $\nu$) and the random set
	\begin{equation}
		\label{def:trace_of_poisson_zoo}
		\mathcal{S}_{\nu}^\lambda := 
		\bigcup_{x \in V} \bigcup_{H \in \latticeanimals_x} \bigcup_{i = 1}^{ N_{x,H}^{\lambda} } \, H
	\end{equation}
	the {\bf trace} of the Poisson zoo at level $\lambda$.
\end{definition}

This is the same definition as (\ref{e.zoo}) in the Introduction, via $N^\lambda_x := \sum_{H \in \latticeanimals_x} N_{x,H}^{\lambda}$. Note also that, since the set of possible animals is countable, there is a positive chance for the occurrence of multiple instances of the same animal with the same root. Although this seems redundant and might suggest to the reader that another distribution should be used to define the model, in the small intensity regime that we are mostly focusing on, i.e., when $\lambda$ is close to $0$, the probability of such events is negligible. Meanwhile, as we will see later in Section~\ref{section:setup}, having a Poisson distribution is a very useful tool in the analysis of the model.

We say that a subset $S \subseteq V(\graph)$ of sites percolates if $S$ has at least one infinite connected component in the induced subgraph. From the rotational invariance of $\nu$ and the factor of i.i.d.~construction (\ref{def:trace_of_poisson_zoo}), it follows immediately that
\begin{equation}
	\label{obs:poisson_zoo_is_ergodic}
	\text{ the law of $\mathcal{S}_{\nu}^\lambda$ is invariant and ergodic under the action of $\grp$.}
\end{equation}
Since the event of existence of an infinite cluster is invariant (under the action of the group $\grp$), we obtain that the value of $\prob \left( \mathcal{S}_{\nu}^\lambda \text{ percolates} \right)$ can be either $0$ or $1$.

For parameters $0<\lambda<\lambda'<\infty$, there are monotone couplings of $X\sim \POI(\lambda)$ and $X'\sim  \POI(\lambda')$ variables such that $X\leq X'$ almost surely; one comes from the obvious monotone coupling of $\mathsf{Exponential}$ variables, another from coupling $\mathsf{Binomial}$ variables and taking the limit. Thus there is also a monotone coupling of the Poisson zoos $\mathcal{N}_{\nu}^{\lambda}$ and $\mathcal{N}_{\nu}^{\lambda'}$ (with the same $\nu$) such that
\begin{equation}
	\label{obs:monotone_coupling_off_poisson_zoos}
	\mathcal{S}_{\nu}^{\lambda} \subseteq \mathcal{S}_{\nu}^{\lambda'} \quad \text{if} \quad
	0 \leq \lambda \leq \lambda'.
\end{equation}
Hence, there is a critical intensity $\lambda_c=\lambda_c(\graph,\nu)$, as defined in~(\ref{e.lambdac}), such that for all $\lambda < \lambda_c$ the trace $\mathcal{S}_{\nu}^{\lambda}$ almost surely does not percolate, while for all $\lambda>\lambda_c$ it almost surely does.

The basic question for any such percolation model concerns the non-triviality of phase transition: given a $\graph$ (and $\grp$), for which choices of $\nu$ do we have $\lambda_c \in (0,\infty)$?

Let us first argue that if $p_c := p_c^{\text{site}}(\graph) < 1$ for i.i.d.\ Bernoulli site percolation, then $\lambda_c<+\infty$ holds for any choice of a rotation invariant measure $\nu$. 
As mentioned in the Introduction, this is the case for all Cayley graphs that are not finite extensions of $\Z$ \cite{Duke} and for all uniformly transient graphs, even without transitivity \cite{EST}. 
Indeed, for any $x \in V(\graph)$, any $H \in \latticeanimals_x$ contains the vertex $x$ of $\graph$.  
Therefore, the trace $\mathcal{S}_{\nu}^\lambda$ of the Poisson zoo at level $\lambda$ stochastically dominates Bernoulli site percolation with density $p=1-e^{-\lambda}$,  which implies $\lambda_c \leq -\ln\left( 1-p_c \right) < +\infty$.

Therefore, the real question is whether the Poisson zoo model has a non-trivial subcritical phase (i.e., $\lambda_c>0$) or is it supercritical for all $\lambda >0$ (i.e., $\lambda_c=0$).
Let us now recall two general results from \cite{RathRokob2022} that give sufficient conditions for $\lambda_c=0$ and $\lambda_c>0$, respectively. 

\begin{definition}
	\label{def:moments_of_nu}
	Given $k \in \N$, let us denote the $k$'th moment of the cardinality of a random subset with law $\nu$ by
	\begin{equation}
		\label{eq:moments_of_nu}
		m_k := m_k(\nu):= \sum_{H \in \latticeanimals_{\origin}} |H|^k \cdot \nu(H).
	\end{equation}
\end{definition}

\begin{lemma}[Lemmas 1.5 and 1.6 in \cite{RathRokob2022}]
	\label{lemma:two_silly_lemmas}
	Assume that the action of $\grp \leq \Aut(\graph)$ on the infinite graph $\graph$ is transitive and unimodular (see Definition~\ref{def:unimodularity}).
	\begin{enumerate}[label=(\arabic*)]
		\item 
		\label{lemma:two_silly_lemmas:infinite_m1_implies_percolation}
		If $m_1 = \infty$, then for any $\lambda >0$ we have $\mathcal{S}_{\nu}^\lambda = V(\graph)$. This implies, in particular, that $\mathcal{S}_{\nu}^\lambda$ percolates for any $\lambda >0$.
		\item
		\label{lemma:two_silly_lemmas:finite_m2_implies_no_percolation}
		If $m_2 < \infty$, then, for any $\lambda \in (0, 1/ (\degree + 1) \cdot m_2 ))$, the set $\mathcal{S}_{\nu}^\lambda$ does not percolate, where $\degree$ denotes the valency of the graph.
	\end{enumerate}
\end{lemma}

These were proved in \cite{RathRokob2022} for the Euclidean lattice $\Z^d$, $d \geq 2$, but they easily extend to our general setting, assuming unimodularity. We omit the proofs here, especially that Lemma~\ref{lemma:size_biasing_for_Poisson_zoo} below is a version of these lemmas, explaining how unimodularity and size-biasing lead to the relevance of the first and second moments. Our results do require unimodularity in a crucial way --- see Remark~\ref{remark:unimodularity_is_needed}.

\subsection{Proof strategy for the Poisson zoo on free products }
\label{subsection:poisson_zoo_on_free_products}

A key example where we can show  the reverse of part~\ref{lemma:two_silly_lemmas:finite_m2_implies_no_percolation} of Lemma~\ref{lemma:two_silly_lemmas}, for completely general animals, are $\degree$-regular trees $\tree_{\degree}$ with $\degree \geq 3$. The proof, with some non-negligible extra work, applies more generally, to unimodular free products (see Definition~\ref{def:free_product}), giving us Theorem~\ref{t.free}. In this subsection, we sketch our proof strategy for this theorem.

The proof in  \cite{RathRokob2022} for Lemma~\ref{lemma:two_silly_lemmas}~\ref{lemma:two_silly_lemmas:finite_m2_implies_no_percolation} relies on a domination by a subcritical branching process {\it from above}. In the present paper, in order to prove percolation for every $\lambda>0$, we will embed a supercritical branching process {\it inside} the percolation cluster of a given vertex. As mentioned above, supercriticality at every $\lambda$ should (and will) be provided by $m_2(\nu)=\infty$, but how do we get a branching structure? The basic observation is that, since in a free product every vertex is a cutpoint that separates two directions, cf.\ Claim \ref{claim:properties_of_free_product_of_graphs}, one can go either in the direction of the first component of the product or the second, and the animals living on these disjoint subgraphs must be also disjoint. Thus, by the defining properties of the underlying Poisson point process, we have that these animals must be independent of each other.
Consequently, we can embed a Galton-Watson branching process into the vertices of the graph by defining the offspring of an individual to be certain vertices associated to the exterior boundary of the set consisting of the union of the traces of the animals visiting only those neighbours of the individual that are accessed from the not yet considered direction.
 However, since the automorphism subgroup under which the Poisson zoo is invariant acts transitively only on the vertices, not on the edges, it could happen that the given rotation-invariant measure $\nu$ could not produce enough animals that are visiting from that direction, and thus would not give rise to a supercritical reproduction mean; see Figure~\ref{f.free2} for this ``parity issue''. For any given occupied set explored so far, we can ensure large further expected growth only through its exterior boundary vertices in one of the directions; see Corollary \ref{coro:lower_bound_on_expected_size_of_fat_of_animals_through_one_vertex_on_free_product}.
 
\begin{figure}[htbp]
\centerline{
\includegraphics[width=2.5 in]{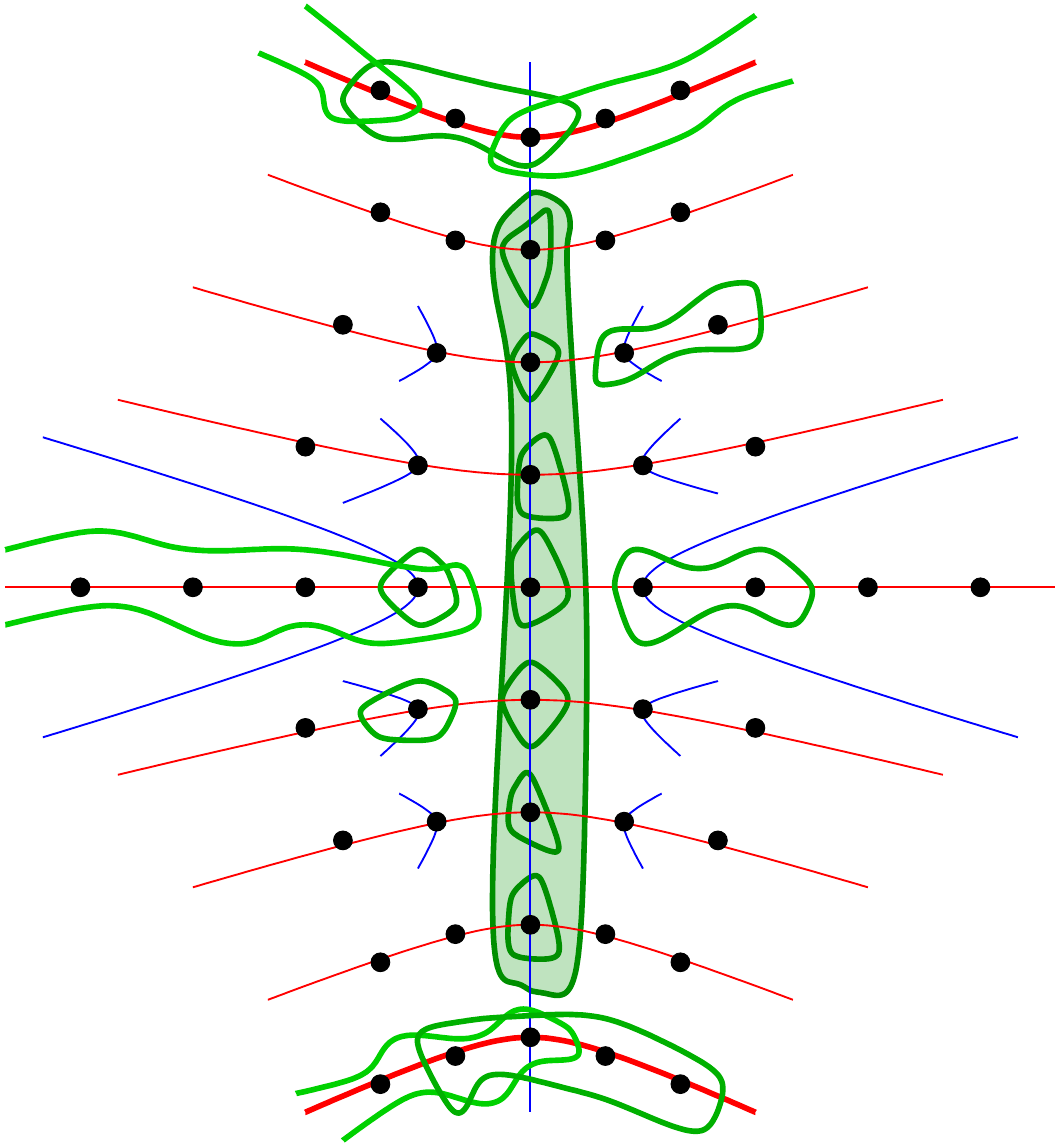}
}
\caption{{\bf A parity issue.} On the free product $\Z\star\Z$, generated by the letters $a$ and $b$, take a $\nu$ with $m_2(\nu)=\infty$ that is supported only on words in one of the generators, say $a$ (the ``horizontal'' red generator in the picture). It may happen that we have already explored a large connected set, shaded green in the picture, which has of course a large total exterior boundary, but only a small one in the $b$ direction. But only these are the exterior boundary vertices from where the $a$ direction is completely unexplored, where we can have a lot of new animals covering that vertex, exploiting the size-biasing effect and the $m_2(\nu)=\infty$ condition.}
\label{f.free2}
\end{figure}
 
Fortunately, this asymmetry caused parity issue can be remedied using a standard method of percolation theory: sprinkling.
Vaguely, if we arrived at some vertex in the exterior boundary of the set visited by animals explored so far from a direction that is not beneficial for further progress, then, instead of considering this vertex as an offspring, we take a neighbour in the right direction. On the one hand, this extra step causes a hole (just a single vertex) in the embedded branching process, but, on the other hand, it will open up territories towards the good direction. The good news is that, if we use only a positive fraction of animals for constructing the branching process, then the hole can be patched using the animals ``growing out of the hole'' from the other positive fraction, independently from everything else. Of course, this has a probability cost, worsening the reproduction mean by a positive factor, but, given the infinite mean for the offspring, we will still remain in the supercritical phase (cf.\ Claim \ref{claim:lower_bound_on_reproduction_rate_branching_process}), proving Theorem~\ref{t.free}.

\medskip

Wishing to progress further, beyond free products, there are serious issues. The tree-like inner structure of the free product guarantees that we can separate the animals into not just independent bundles, but non-intersecting ones (which then can be used as the offsprings). One cannot aim for such a disjointness of general animals in a general nonamenable graph, but may try to experiment with some relaxation of disjointness, where there is still a control over the intersections of the bundles. However, given the varieties in animal shapes and in the geometries of Cayley graphs, this seems like a really hard problem. One case where this can be done is when the animals are random walk trajectories, taking us to our next result.

%
%
%

\subsection{Restatement and proof strategy for the worms model}
\label{subsection:random_length_worms_model_introduction}

Random walk trajectories are very natural, locally growing random sets, worth studying for a variety of reasons. On nonamenable graphs, they grow linearly (Theorem~\ref{thm:positive_constant_speed_of_SRW_on_nonamenable_transitive_graphs} above), hence they may be expected to produce large clusters, possibly achieving $\lambda_c(\nu)=0$ in great generality.


\begin{definition}
	\label{def:random_length_worms_model}
	Let us consider a probability mass function $\eta: \N_{>0} \to [0,1] $ and an $\N_{>0}$-valued random variable $\mathcal{L}$ satisfying $\prob(\mathcal{L}=\ell)=\eta(\ell), \, \ell \in \N_{>0}$. 
	We call the distribution of $\mathcal{L}$ the length distribution of the worms. Let us also consider a nearest neighbour simple random walk $(X(t))_{t \in \N}$ on the vertices of $\graph$ starting from $X(0)=\origin$, independent of $\mathcal{L}$.
	Let us define the probability measure $\wormnu{\eta}$ on $\mathcal{H}_{\origin}$ by 
	\begin{equation}
		\label{def:worm_intensity_measure}
		\wormnu{\eta}(H) := P_{\origin}( \{ X(0), X(1), \ldots, X(\mathcal{L} - 1) \} = H )
	\end{equation}
	for each $H \in \mathcal{H}_{\origin}$.
	Given any $\lambda >0$, the Poisson zoo $\mathcal{N}_{\eta}^{\lambda} := \mathcal{N}_{\wormnu{\eta}}^{\lambda}$ built upon this measure via Definition \ref{def:poisson_zoo} is called the random length worms model at level $\lambda$, and the corresponding random set $\mathcal{S}_{\eta}^\lambda := \mathcal{S}_{\wormnu{\eta}}^\lambda$ is the random length worms set at level $\lambda$.
\end{definition}

Note that the law of a randomly stopped random walk on a graph $\graph$ is $\Aut(\graph)$-rotation invariant in the sense of Definition~\ref{def:grp-rotation-invariant_prob_measure}, making this zoo model meaningful.
Moreover, the moment conditions of Lemma~\ref{lemma:two_silly_lemmas} are easy to check:

\begin{lemma}[Time versus trace]\label{l.nuell}
On any nonamenable transitive graph, the lattice animal measure $\wormnu{\eta}$ corresponding to a worm length measure $\eta$, from Definition~\ref{def:random_length_worms_model}, satisfies $m_1(\wormnu{\eta})<\infty$ and $m_2(\wormnu{\eta})=\infty$ if and only if $\E[\mathcal{L}]<\infty$ and $\E[\mathcal{L}^2]=\infty$.
\end{lemma}

\begin{proof}
Obviously, $\big|X[0,\cL)\big| \leq \cL$, hence an infinite first or second moment for $\wormnu{\eta}$ implies the same for $\cL$. 

For the other directions, we can use Corollary~\ref{coro:number_of_steps_and_size_of_trace_of_SRW_are_comparable_on_nonamenable_graphs} and the random walk steps being independent of the length $\mathcal{L}$, implying, for $j=1,2$ and any $\veps\in (0,1)$, the following:
\begin{equation}
\begin{aligned}
m_j(\wormnu{\eta}) &= \sum_{\ell\ge 1} E_o \Big[ \big| X[0, \mathcal{L}) \big|^j \,\Big|\, \mathcal{L}=\ell \Big]  \, \eta(\ell)\\
& \overset{\eqref{eq:number_of_steps_and_size_of_trace_of_SRW_are_comparable_on_nonamenable_graphs}}{\geq}  (1-\veps) \, c(\veps)^j  \sum_{\ell \geq 1} \ell^j \, \eta(\ell)  \ge c_j \, \E \left[ \mathcal{L}^j \right].
\end{aligned}
\end{equation}
This finishes the proof.
\end{proof}

As a corollary, we immediately get that the next theorem is equivalent to Theorem~\ref{t.worms}.

\begin{theorem}
	\label{thm:random_length_worms_model_on_nonamenable_graphs}
	Let $\graph$ be an infinite graph, such that $\grp \subgrp \Aut(\graph)$ acts transitively and unimodularly on $\graph$.
	If $\mathcal{L} \sim \eta$ with $\E \left[ \mathcal{L}^2 \right] = +\infty$, then for any $\lambda>0$ the random length worms model $\mathcal{N}_{\eta}^\lambda$ is supercritical:
	$\prob(\, \mathcal{S}_{\eta}^\lambda \text{ percolates}\, )=1$.
\end{theorem}

As indicated above, the proof of this result uses the following process exploring the cluster of the origin.
Initially, we take the union of all those worms that ever visit the origin, giving the set $E_0$.
Assuming that $E_0$ is not empty and large enough, we can ``fatten'' this set in the next step using trajectories that 
$(1)$ have length at most $R$; and 
$(2)$ bounce back from the closure $E_0 \cup \extboundary E_0$, meaning that the worm visits the exterior boundary of $E_0$ once, but, apart from this visit, it avoids the closure entirely.
(Note here that, since we consider site percolation, it is not a problem that the edges between the set and its exterior boundary are not visited.)

If we denote the union of $E_0$ and the traces of the bounced back worms by $E_1$, then of course it can happen that $\extboundary E_1 \cap \extboundary E_0$ is not empty, meaning that no worm bounced back from the vertices of this intersection.
Consequently, to obtain independence between the iterations, in the next step we have to change property $(2)$ to the one: the worm of length at most $R$ bounces back from $\extboundary E_1 \setminus \extboundary E_0$ and never visits the closure $E_1 \cup \extboundary E_1$.
This is what we can call a {\bf fattening} of $E_1$ through $\extboundary E_1 \setminus \extboundary E_0$ (with worms of length at most $R$), and what the exploration does as long as it can: given $E_n \supseteq E_{n-1}$ it fattens $E_n$ through $\extboundary E_n \setminus \extboundary E_{n-1}$ to obtain $E_{n+1}$ as the union of $E_n$ and the traces of the fattening worms.

Of course, to achieve that this exploration runs indefinitely with positive probability (and hence that $\mathcal{S}_{\eta}^\lambda$ percolates almost surely), we also need that $E_n$ is so large compared to $E_{n-1}$ that it is still well visible by the newcoming worms even through the ``lens'' of $\extboundary E_n \setminus \extboundary E_{n-1}$.
From the point of view of random walkers, this visibility of a set is characterized by its capacity, defined in~\eqref{def:capacity}.
Fortunately, as the capacity of a set is comparable to its size on a nonamenable graph (see Lemma~\ref{lemma:capacity_bound}), the condition above reduces to a condition on the sizes, that is, it is enough to guarantee that $|E_n| > a \cdot |E_{n-1}|$ with some large enough $a \gg 1$ depending only on the graph.
As we will see in Proposition~\ref{p.growth}, this follows if $R$ is chosen large enough.

Let us emphasize here that one of the main reasons why we could prove Theorem \ref{thm:random_length_worms_model_on_nonamenable_graphs} with the outlined argument is the aforementioned control on the worms: its trace, a trajectory, can bounce back from the already explored set.
Consequently, due to the defining properties of Poisson point processes, cf.\ Theorem \ref{thm:basic_properties_of_PPPs}, we can do the fattening of the set already explored through the available subset of its exterior boundary by independent packages of fattening worms (a package being the collection of walkers bouncing back from the same vertex), which makes it possible to derive first and second moment bounds on the size of the set occupied by the newcomers, the ``fat''; see Subsection~\ref{subsection:worms_on_nonamenable_graphs_fattening} for the details.

\medskip

Again, one may wonder how far this strategy could be pushed. Worms are able to touch any given set and bounce back, but general animals do not have such capability: it is not even known if for any Cayley graph $\graph$ there is a constant $K<\infty$ such that for any two finite sets $A,B\subset V(\graph)$ there is a graph automorphism $\gamma$ that achieves $0< \dist_\graph(A,\gamma(B)) < K$. See \cite[Question 3]{BandyopadhyaySteifTimar2010} or \cite[Exercise 5.10]{Pete2023}.
This gap cannot be bridged even by sprinkling, as used in the proof of Theorem \ref{t.free}. And even if this issue of touching was solved for some animals, it is not clear how these touching animals can overlap with each other, making the existence and linearity of any notion of ``capacity'' questionable.

%

\subsection{Point processes of rooted animals}
\label{subsection:point_processes_of_rooted_animals}

The goal of this subsection is to provide an easy to use, but still general enough object, namely a Poisson point process on a general enough space, that can be helpful in the analysis of the Poisson zoo and it special cases.
Let us mention that although there is a large overlap between the construction we will describe here and the one given in \cite[Section 5]{RathRokob2022}, we will build up the notations and the concepts from zero again.
On one hand, this will make this paper more self-contained, and on the other hand here we will use rooted simple sets (animals) as the points of our base space, instead of trajectories, i.e., indexed multisets, which were used in \cite{RathRokob2022}.

\subsubsection{The universal zoo}

Let us denote the space of all rooted animals, i.e., {\bf the universal zoo} by
\begin{equation}
	\label{def:universal_zoo}
	\zoospace := \zoospace(\graph) := \left\{ (x,H) \, : \, H \in \latticeanimals_{x}, \, x \in V(\graph) \right\}.
\end{equation}
Since our graph $\graph$ is locally finite and each animal is a finite set, this set is countable.
Moreover, this space can be endowed with a natural $\sigma$-algebra $\zooalgebra$: all subsets of $\zoospace$.

Although for the different special cases of the Poisson zoo, different characteristics are relevant for the species --- e.g., the radius for the balls, or the number of steps taken for the worms --- we will only use the {\bf volume} function on $\zoospace$, i.e.,
\begin{equation}
	\label{def:volume}
	\volume((x,H)) := \left| H \right|
\end{equation}
for any $(x,H) \in \zoospace$.

Obviously, this volume can be used to decompose $\zoospace$ into disjoint subspaces:
\begin{align}
	\label{def:subzoo_of_fixed_volume_animals}
	\zoospace =	\bigcup_{k = 1}^{\infty} Z_k,
	\qquad\text{where } 
	Z_k:=	\big\{ (x, H) \; : \; x \in V(\graph), \ H \in \latticeanimals_x, \ \volume((x,H)) = k \big\}.
\end{align}

\begin{definition}
	\label{def:set_hitting_animals}
	For any $K \subset \subset V(\graph)$, let us introduce the notation
	\begin{equation}
		\label{eq:set_hitting_animals}
		\zoospace(K) := 
		\left\{
			(x, H) \in \zoospace \, : \, H \cap K \neq \emptyset
		\right\},
	\end{equation}
	the subset of animals that hit the set $K$.
	Analogously, let us also define, for any $k \in \N_{>0}$,
	\begin{equation}
		\label{eq:set_hitting_animals_of_volume_k}
		\zoospace_k(K) := 
		\left\{
			(x, H) \in \zoospace \, : \, H \cap K \neq \emptyset, \, \volume((x,H)) = k
		\right\},
	\end{equation}
	and denote the finite union of such sets, for $R \in \N$, as
	\begin{equation}
		\label{eq:set_hitting_animals_of_volume_at_most_R}
		\zoospace^R(K) :=
		\bigcup_{k = 0}^{R} Z_k(K) =
		\left\{
			(x, H) \in \zoospace \, : \, H \cap K \neq \emptyset, \, \volume((x,H)) \leq R
		\right\}.
	\end{equation}
\end{definition}

\subsubsection{Point measures and their occupation measures}

Let us denote by $\zooPP$ the {\bf space of point measures} on $\zoospace$, that is,
\begin{equation}
	\animals \in \zooPP \, \iff \,
	\animals = \sum_{i \in I} \delta_{(x,H)_i},
\end{equation}
where $I$ is a finite or countably infinite set of indices, $(x,H)_i \in \zoospace$ for all $i \in I$ and $\delta_{(x,H)}$ denotes the Dirac mass concentrated on $(x,H) \in \zoospace$. {The topology on $\zooPP$ is the discrete one.}

Let us denote by $\animals(A)$ the $\animals$-measure of $A \in \zooalgebra$. 
We can alternatively think of an $\animals \in \zooPP$ as a multiset of animals, where $\animals(\{ (x,H) \})$ is the number of copies of the animal $(x,H)$ contained in $\animals$.
Since later on we want to consider random variables defined on the space $\zooPP$ of form $\animals(A)$ for any $A \in \zooalgebra$, we define $\zooPPalgebra$ to be the sigma-algebra on $\zooPP$ generated by such variables.

If $\animals  = \sum_{i \in I} \delta_{(x,H)_i} \in \zooPP$ and $A \in \zooalgebra$, let us denote by
\begin{equation}
	\label{def:restricted_zoo}
	\animals \ind[A] := \sum_{i \in I} \delta_{(x,H)_i} \ind[ (x,H)_i \in A ]
\end{equation}
the {\bf restriction} of $\animals$ to the set $A$ of animals. Note that $ \animals \ind[A]$ is also an element of $\zooPP$. More generally, given $A \in \zooalgebra $, if $f : A \lora \zoospace$ is a function and $\animals = \sum_{i \in I} \delta_{(x,H)_i} \in \zooPP$ is a point measure satisfying $\animals = \animals \ind[ A ]$, then let
\begin{equation}
	\label{def:f_image_of_zoo}
	f(\animals):= \sum_{i \in I} \delta_{f((x,H)_i)}
\end{equation}
denote the {\bf image of $\animals$ under $f$}. Note that $f(\animals)$ is also an element of $\zooPP$. {The restriction $\animals \ind[ A ]$ is indeed a special case, where $f$ is multiplication by $\ind[A]$.}

Given $(x,H) \in \zoospace$ and $\animals = \sum_{i \in I} \delta_{(x,H)_i} \in \zooPP$, we define the {\bf traces} $\trace((x,H)) \subset V(\graph)$ and $\trace( \animals ) \subseteq V(\graph)$ by
\begin{equation}
	\label{def:trace_of_animals_and_zoo}
	\trace( (x,H) ):= H, \qquad \text{and} \qquad
	\trace\left( \animals \right):= \bigcup_{i \in I} \trace((x,H)_i);
\end{equation}
that is, 
 $\trace\left( \animals \right)$ is the set of all sites occupied by $(x,H)_i, i \in I$.


Given $(x,H) \in \zoospace$ and $\animals = \sum_{i \in I} \delta_{(x,H)_i} \in \zooPP$, we define the {\bf occupation measures} $\LTmeasure^{(x,H)}$ and $\LTmeasure^{\animals}$ on $\graph$ by
\begin{equation}
	\label{def:local_time_measure}
	\LTmeasure^{(x,H)} := \sum_{y \in H} \delta_{y}, \qquad \text{and} \qquad \LTmeasure^{\animals}:= \sum_{i \in I}  \mu^{(x,H)_i}.
\end{equation}
Thus, if $y \in V$, then $\LTmeasure^{(x,H)}(y) := \LTmeasure^{(x,H)}(\{y\}) = \ind \left[ y \in H \right]$ and $\LTmeasure^{\animals}(y) := \LTmeasure^{\animals}(\{y\})$ equals the number of sets containing $y$.  
 Given some  $\animals = \sum_{i \in I} \delta_{(x,H)_i} \in \zooPP$, we define the {\bf total size} or {\bf total occupation measure} $\totalsize^{\animals}$ of the animals $(x,H)_i, i \in I$, by
\begin{equation}
	\label{def:total_local_time}
	\totalsize^{\animals} :=
	\mu^{\animals}(V) = \sum_{i \in I} \volume((x,H)_i).
\end{equation}

\subsubsection{The zoo as a Poisson point process}

Recall the lattice animal measure $\nu=\nu_o$ from Definition~\ref{def:grp-rotation-invariant_prob_measure}, and that~\eqref{def:translated_prob_measure} gives us, for any $x \in V(\graph)$, a translated version $\nu_{x}$ on $\latticeanimals_{x}$.  Together they induce a probability measure
\begin{equation}
	\label{def:induced_measure_on_zoospace}
	\nu := \mathop{\bigotimes}_{x \in V(\graph)} \nu_{x}
\end{equation}
on the universal zoo $Z$ defined in~\eqref{def:universal_zoo}. That is, for an animal $(x,H) \in \zoospace$, we have $\nu((x,H)) := \nu( \{ (x,H) \} ) = \nu_x(H)$.

\begin{definition}
	\label{def:Poisson_zoo_as_PPP}
	Given some $\lambda \in \R_{>0}$ and a rotation invariant probability measure $\nu_{\origin}$ on $\latticeanimals_{\origin}$, a random element $\animals$ of $\zooPP$ has law $\mathcal{Q}_{\nu}^{\lambda}$ if $\animals$ is a Poisson point process (abbreviated as PPP) on $\zoospace$ with intensity measure $\lambda \cdot \nu( (x,H) )$, where $\nu$ is the measure induced by $\nu_{\origin}$ as in \eqref{def:induced_measure_on_zoospace}.
\end{definition}

%

\begin{claim}
	\label{claim:Poisson_zoo_set_as_the_trace_of_a_PPP}
	The trace $\mathcal{S}_{\nu}^\lambda$ of the Poisson zoo of Definition~\ref{def:poisson_zoo}  can be alternatively defined as $\mathcal{S}_{\nu}^\lambda := \trace ( \animals )$, where $\animals \sim \mathcal{Q}_{\nu}^\lambda$.
\end{claim}


The special case of  the random length worms model deserves a separate notation:

\begin{definition}
	\label{def:random_length_worms_as_PPP}
	Given a probability mass function $\eta \, : \, \N_{>0} \rightarrow [0,1]$ for the random length worms model of Definition~\ref{def:random_length_worms_model}, the law $\mathcal{Q}_{\nu}^\lambda$ of the corresponding Poisson point process on $\zoospace$ will instead be denoted by $\mathcal{P}^\lambda := \mathcal{P}_{\eta}^{\lambda}$.	
\end{definition}


The main reason for looking at the Poisson zoo through the lenses of Poisson point processes is that for the latter there is a well developed theory and hence many tools to use for analysis. Fortunately, there are only a few of them are relevant for our results, which we will collect here from \cite{Kingman1993} and \cite[Chapter 3]{Resnick2008} after translating them into our setting.

\begin{theorem}
	\label{thm:basic_properties_of_PPPs}
	Let $\animals = \sum_{i \in I} \delta_{(x,H)_i} \sim \mathcal{Q}_{\nu}^\lambda$ be as in Definition \ref{def:Poisson_zoo_as_PPP} and given any $A \in \zooalgebra$ introduce the restricted measure $\nu_A(\cdot) := \nu ( \cdot \cap A)$.
	\begin{enumerate}[label=(\roman*)]
		\item If $A \in \zooalgebra$, then the restricted process $\animals \ind [A]$ has law $\mathcal{Q}_{\nu_A}^\lambda$.
		\item If $A_1, A_2, \ldots \in \zooalgebra$ are such that $A_i \cap A_j = \emptyset$ for $i \neq j$, then the restricted processes $\animals \ind [A_1], \animals \ind[A_2], \ldots$ are independent.
		\item Let $p \in [0,1]$ and $\{ \xi_i \}_{i \in I}$ a collection of independent $\Bernoulli(p)$ random variables, and consider the following decomposition:
		\begin{equation}
			\animals = \animals_1 + \animals_2 := \sum_{i \in I} \xi_i \delta_{(x,H)_i} + \sum_{i \in I} (1 - \xi_i) \delta_{(x,H)_i}.
		\end{equation}
		Then $\animals_1 \sim \mathcal{Q}_{\nu}^{p v}$ and $\animals_2 \sim \mathcal{Q}_{\nu}^{ (1-p)v }$ in law, and they are independent of each other.
		\item If $A \in \zooalgebra$, then $\animals(A)$ follows a Poisson distribution.
	\end{enumerate}
\end{theorem}

Let us note here that $(iii)$ is sometimes called the colouring property of the Poisson point process, as intuitively it means that colouring all points of the process with two colours, red and blue, independently of each other, both coloured processes are again Poisson point processes, with the appropriately compensated intensity measures.
Moreover, properties $(ii)$ and $(iv)$ are defining properties of the Poisson point process; see for example~\cite[Section 3.3]{Resnick2008}.

One consequence of the above theorem that we will use is that one can express the expectation of any nonnegative function $f(\animals)$ of a sample $\animals \sim \mathcal{Q}_{\nu}^{\lambda}$ using its intensity measure $\nu$. The proof of this formula follows the lines of \cite[Lemma 5.3]{RathRokob2022}, hence it is omitted.

\begin{lemma}
	\label{lemma:general_Campbell_formulas}
	Given $\animals = \sum_{i \in I} \delta_{(x,H)_i} \sim \mathcal{Q}_{\nu}^\lambda$ and a function $f \, : \, \zoospace \longrightarrow \R^+$, we have
	\begin{equation}
		\label{eq:general_Campbell_expectation_formula}
		\E \left[ \sum_{i \in I} f((x,H)_i) \right] =
		\lambda \cdot \sum_{ x \in V } \sum_{H \in \mathcal{H}_x}  f \left( (x,H) \right) \cdot \nu \left( (x,H) \right).
	\end{equation}
\end{lemma}

\section{First and second moment bounds on the fat}
\label{section:bounds_on_fat}

As we explained in Subsections~\ref{subsection:poisson_zoo_on_free_products} and~\ref{subsection:random_length_worms_model_introduction}, both of our theorems will be proved via certain exploration processes, fattening the so-far explored cluster of the origin step-by-step, by animals that intersect the exterior boundary of the current cluster but not the cluster itself. In both cases, we will need to have further restrictions on where these new animals touch the current cluster, making it useful to introduce the following notation.

\begin{definition}
	\label{def:point_measure_restricted_to_hit_B_but_not_the_closure_of_A}
	Given a point measure $\animals = \sum_{i \in I} \delta_{(x,H)_i} \in \zooPP$, some $R \in \N_{>0}$, a finite connected set $A \subseteq V(\graph)$ and a subset of its exterior boundary $B \subseteq \extboundary A$ let us introduce the following restriction of the point measure:
	\begin{equation}
		\label{eq:point_measure_restricted_to_hit_B_but_not_the_closure_of_A}
		\check{ \animals }_{A,B}^R :=
		\animals \, \ind \left[
			\zoospace^R(B) \setminus \zoospace^R \left( \overline{A} \setminus B \right)
		\right].
	\end{equation}
	That is, we consider those animals only that have volume at most $R$, hit $B$, but do not intersect $\overline{A}$ otherwise. Let us also introduce the notation 
	\begin{equation}
		\label{def:point_measure_restricted_to_hit_B_but_not_the_closure_of_A_total_local_size}
		  \check{\mu}_{A,B}^R (x) := \mu^{ \check{ \animals }_{A,B}^R }(x)
		   \qquad\text{and}\qquad 
		   \check{ \totalsize }_{A,B}^R := \totalsize^{ \check{ \animals }_{A,B}^R } = \sum_{x \in V(\graph)} \check{\mu}_{A,B}^R (x)
	\end{equation}
	for the (total) occupation measure of the restricted point measure.
\end{definition}

The exploration processes will be introduced and analysed in Section~\ref{section:exploration_processes}, but
quite evidently, whatever the exact exploration is, we will need to prove lower bounds on $| \trace( \check{ \animals }_{A,B}^R ) |$: in expectation for the branching process argument on free products; in probability for the more complicated exploration process for worms on general graphs. In the next subsection, we will argue that, instead of the trace, it is enough to examine $\check{ \totalsize }_{A,B}^R$, obviously easier to handle. Then, in Subsection~\ref{ss.fatfree}, we will give first moment lower bounds for general zoos on free products, and in Subsection~\ref{subsection:worms_on_nonamenable_graphs_fattening}, first moment lower and second moment upper bounds for worms.

%
%
%

\subsection{Size-biasing and neglecting multiplicities}

Recall the notation $m_k$ from Definition \ref{def:moments_of_nu}. Due to Lemma~\ref{lemma:two_silly_lemmas}~\ref{lemma:two_silly_lemmas:infinite_m1_implies_percolation}, we can assume that $m_1 < \infty$. Given a fixed integer $R \in \N_{>0} \cup \{ + \infty \}$, called the truncation parameter (whose value will be chosen later), and any $k \in \N_{>0}$, we introduce the notion of truncated moments as
\begin{equation}
	\label{def:truncated_moments_of_nu}
	m_k^{R} := \sum_{H \in \mathcal{H}_{\origin}} |H|^k \cdot \ind[|H| \leq R] \cdot \nu_{\origin}(H).
\end{equation}

We defined unimodularity of a transitive graph in Definition~\ref{def:unimodularity}, via the mass transport principle~\eqref{eq:MTP}. An important consequence is the following \emph{size-biasing} phenomenon.

\begin{lemma}[Size-biasing in the Poisson zoo]
	\label{lemma:size_biasing_for_Poisson_zoo}
	Let us assume that $\grp \subgrp \Aut(\graph)$ is a transitive and unimodular subgroup of automorphisms, $\nu_{\origin}$ is a $\grp$-rotation invariant measure on $\latticeanimals_{\origin}$, $\nu$ is the measure on $\zoospace$ induced from $\nu_{\origin}$ via \eqref{def:induced_measure_on_zoospace} and $\animals = \sum_{i \in I} \delta_{(x,H)_i} \sim \mathcal{Q}_{\nu}^\lambda$. Given some $R \in \N \cup \{ + \infty \}$ and $x \in V(\graph)$, consider the restricted processes $\animals^R := \animals \ind \left[ \zoospace^R \right]$ and $\animals_x^R := \animals \ind \left[ \zoospace^R( \{ x \} ) \right]$. Then we have
	\begin{equation}
		\label{eq:size_biasing_for_Poisson_zoo}
		\E \left[  \LTmeasure^{\animals^R}(x)  \right] = \lambda \cdot m_1^R
		\qquad \text{and} \qquad
		\E \left[ \totalsize^{ \animals_x^R } \right] = \lambda \cdot m_2^R.
	\end{equation}
\end{lemma}

\begin{proof}
	Both equalities follow from the mass transport principle \eqref{eq:MTP} used with some appropriate mass transport functions.
	Both transport functions can be derived from the following collection of functions.
	For $H \in \latticeanimals_{\origin}$ and $y \in V(\graph)$ let us define
	\begin{equation}
		f_H^R(y,x) := \ind \left[ x \in \trace \left( y, \varphi_y(H) \right) \right] \cdot \ind \left[ |H| \leq R \right],
	\end{equation}
	where $\varphi_y \in \grp$ is the one granted by \eqref{def:the_distinguished_automorphism_taking_origin_to_x}.
	Note that $f_H^R$ is not necessarily a diagonally invariant function (hence not a mass transport function), since there could be automorphisms $\gamma \in \grp$  for which $\varphi_{ \gamma(y) }(H) \neq \gamma ( \varphi_y(H) )$ holds.
	However, the functions
	\begin{equation}
		\label{eq:mt_function_for_size_biasing}
		F_1^R(y,x) := \sum_{H \in \latticeanimals_{\origin}} f_H^R(y,x)
		\qquad \text{and} \qquad
		F_2^R(y,x) := \sum_{H \in \latticeanimals_{\origin}} |H| \cdot f_H^R(y,x)
	\end{equation}
	are mass transport functions.
	Indeed, it follows easily that they are diagonally invariant under the action of $\grp$, as it defines a permutation on $\latticeanimals_{\origin}$ without changing the volume of the sets.
	
	For the first equality of \eqref{eq:size_biasing_for_Poisson_zoo}, note that  $F_1^R(y,x)$ counts the number of animals of volume at most $R$ that are rooted in $y \in V(\graph)$ and contain $x$. Consequently,
	\begin{align*}
		\E \left[ \LTmeasure^{ \animals^R }(x) \right] 
		& \overset{\eqref{eq:general_Campbell_expectation_formula}}{=}
		\lambda \cdot \E_\nu \left[ \sum_{y \in V(\graph)} F_1^R(y,x) \right] \overset{\eqref{eq:MTP}}{=}
		\lambda \cdot \E_\nu \left[ \sum_{y \in V(\graph)} F_1^R(x,y) \right] \\ & \overset{(\ref{eq:mt_function_for_size_biasing},~\ref{def:translated_prob_measure})}{=}
		\lambda \cdot \sum_{H \in \latticeanimals_{\origin}} \left| H \right| \cdot \ind \left[ |H| \leq R \right] \cdot \nu_{\origin}(H)  \overset{\eqref{def:truncated_moments_of_nu}}{=}
		\lambda \cdot m_1^R.
	\end{align*}
	
	Meanwhile in $F_2^R(y,x)$ we also take into account the volume of the animals. 	
	Hence,
		\begin{align*}
		\E \left[ \totalsize^{ \animals_x^R } \right] 
		& \overset{\eqref{eq:general_Campbell_expectation_formula}}{=}
		\lambda \cdot \E_\nu \left[ \sum_{y \in V(\graph)} F_2^R(y,x) \right] \overset{\eqref{eq:MTP}}{=}
		\lambda \cdot \E_\nu \left[ \sum_{y \in V(\graph)} F_2^R(x,y) \right] \\ & \overset{(\ref{eq:mt_function_for_size_biasing},~\ref{def:translated_prob_measure})}{=}
		\lambda \cdot \sum_{H \in \latticeanimals_{\origin}} \left| H \right|^2 \cdot \ind \left[ |H| \leq R \right] \cdot \nu_{\origin}(H)  \overset{\eqref{def:truncated_moments_of_nu}}{=}
		\lambda \cdot m_2^R,
	\end{align*}
proving the second equality of \eqref{eq:size_biasing_for_Poisson_zoo}.
 \end{proof}

The appearance of the truncated second moment $m_2^R$ in the expected total occupation measure in~\eqref{eq:size_biasing_for_Poisson_zoo}, supplemented with the assumption that $m_2 = +\infty$, can be used to make this quantity sufficiently large, by taking $R$ large. Unimodularity is crucial for this, as shown by the following example.
 Recall from the paragraphs right after Definition~\ref{def:unimodularity} that a regular tree with a distinguished end is not unimodular.

\begin{remark}
	\label{remark:unimodularity_is_needed}
	Let $\hat{T}$ denote a decorated $3$-regular tree in which there is a distinguished infinite ray.	
	This fixed infinite ray marks a direction in the graph, namely upwards is when we are moving towards to the end represented by the ray.
	It is easy to see that $\Aut(\hat{T})$ acts transitively on the vertices, hence we can build a Poisson zoo on $\hat{T}$ from some $\Aut(\hat{T})$-rotation invariant measure $\nu$ on the animals of $\latticeanimals_{\origin}$ via Definition \ref{def:poisson_zoo}.
	
	We choose $\nu$ to be 
	\begin{equation}
		\nu (H) := \frac{1}{2^r} \cdot \prob \left( \mathcal{L} = r \right) \cdot \ind \left[ H \text{ is a downward path of length $r$ started from $\origin$} \right],
	\end{equation} 
	where $\mathcal{L}$ is such a random variable that $\prob ( \mathcal{L} = r ) \asymp r^{-3}$ for all $r \in \N_{>0}$.
	Since the automorphisms of $\Aut( \hat{T} )$ fix the distinguished infinite ray, $\nu$ will indeed be $\Aut(\hat{T})$-rotation invariant. Moreover, we are in the interesting regime, as $m_1 = \E[\mathcal{L}] < + \infty$ and $m_2 = \E [ \mathcal{L}^2 ] = + \infty$ hold.
	
	However, using the notation of Lemma \ref{lemma:size_biasing_for_Poisson_zoo}, for the expected total size of animals ever visiting some vertex $x \in V(\hat{T})$ (counted with multiplicity), we only have 
$$
		\E \left[ \totalsize^{ \animals_x^R } \right] \overset{(*)}{=}
		\lambda \sum_{r = 1}^R  \sum_{k = 0}^{r-1} 2^{r-k} \cdot \frac{r}{2^r} \cdot \prob \left( \mathcal{L} = r \right)
		\asymp \lambda \sum_{r = 1}^R  \frac{1}{r^2}  \sum_{k = 0}^{r-1} \frac{1}{2^k} \leq \lambda \sum_{r = 1}^\infty \frac{2}{r^2} < \infty\,,
		$$
	which is not controlled by the truncated second moment anymore, and hence can not be chosen arbitrarily large.
	Here in $(*)$ we used \eqref{eq:general_Campbell_expectation_formula} and the fact that the restriction to only downward paths implies that there is only one unique way to get to $x$ from any vertex above.	 
\end{remark}

The next lemma tells us that the expected number of animals hitting some vertex, given that there is at least one of such animal,  can be bounded from above by $\lambda \cdot m_1$. Moreover, this holds for any restricted zoo, as well. Hence, under the natural assumption that $m_1 < \infty$, we will obtain that neglecting multiplicities affects occupation measures only by bounded factors.

\begin{lemma}[Upper bound on the multiplicity]
	\label{lemma:upper_bound_on_expected_number_of_occupation}
	Under the assumptions of 	Lemma~\ref{lemma:size_biasing_for_Poisson_zoo},
	let $\animals = \sum_{i \in I} \delta_{(x,H)_i} \sim \mathcal{Q}_{\nu}^\lambda$.
	Given any $A \in \zooalgebra$, let us denote the restricted zoo by $\bar{ \animals } = \animals \, \ind \left[ A \right]$ and the corresponding total occupation measure by $ \bar{\LTmeasure}(x) := \LTmeasure^{ \bar{ \animals } }(x)$.
	 Then we have
	\begin{equation}
		\label{eq:upper_bound_on_expected_number_of_occupation}
		\sup_{x \in V(\graph)} \E \left. \left[ \bar{\LTmeasure}(x) \, \right| \, \bar{\LTmeasure}(x) > 0 \right] \leq \lambda \cdot m_1 + 1.
	\end{equation}
If $\bar{ \animals }$ and $x$ are such that $\bar{\LTmeasure}(x)=0$ a.s., we consider the conditional expectation to be 0.

\end{lemma}

\begin{proof}
	By $(iv)$ of Theorem \ref{thm:basic_properties_of_PPPs}, it follows that, for any $x \in V(\graph)$, the distribution of the random variable $\bar{\LTmeasure}(x)$ is Poisson with parameter $\E \left[ \bar{ \LTmeasure }(x) \right]$.
	Moreover, if the restriction given by $A$ on $\animals$ is such that $\bar{\LTmeasure}(x)=0$ for all $x$, then \eqref{eq:upper_bound_on_expected_number_of_occupation} follows trivially. Hence let us assume that this is not the case: there exists at least one $x \in V(\graph)$ such that $\bar{\LTmeasure}(x)$ is non-trivial.
	
	From basic probability theory and real analysis $(*)$, for any such $x \in V(\graph)$ we can write the following upper bound:
	\begin{equation}
		\label{eq:upper_bound_on_expected_number_of_occupations_given_there_is_one}
		\E \left. \left[ \bar{\LTmeasure}(x) \, \right| \, \bar{\LTmeasure}(x) > 0  \right] =
		\frac{\E \left[ \bar{\LTmeasure}(x) \right]}{ \prob \left( \bar{\LTmeasure}(x) > 0 \right)} =
		\frac{\E \left[ \bar{\LTmeasure}(x) \right]}{ 1- \exp \left\{ - \E \left[ \bar{\LTmeasure}(x) \right] \right\} } \overset{(*)}{\leq} 
		\E \left[ \bar{\LTmeasure}(x) \right] + 1.
	\end{equation}
	At the same time, the expectation appearing on the right-hand side of \eqref{eq:upper_bound_on_expected_number_of_occupations_given_there_is_one} can be bounded using Lemma \ref{lemma:size_biasing_for_Poisson_zoo} as
	\begin{equation}
		\label{eq:MTP_argument_for_expected_number_of_occupations}
		\E \left[ \bar{\LTmeasure}(x) \right] \overset{(**)}{\leq} 
		\E \left[ \LTmeasure^{ \animals }(x) \right] \overset{\eqref{eq:size_biasing_for_Poisson_zoo}}{=} \lambda \cdot m_1,
	\end{equation}
	where $(**)$ holds since we left the restriction on the animals.
	Putting \eqref{eq:MTP_argument_for_expected_number_of_occupations} back into \eqref{eq:upper_bound_on_expected_number_of_occupations_given_there_is_one} we arrived at an upper bound that does not depend on $x$, hence yields the desired result.
\end{proof}

The previous lemma will only be used via the following corollary:

\begin{corollary}[Neglecting multiplicities]
	\label{coro:expected_size_without_and_with_multiplicity}
	Under the assumptions of 	Lemma~\ref{lemma:size_biasing_for_Poisson_zoo},
	let $\animals = \sum_{i \in I} \delta_{(x,H)_i} \sim \mathcal{Q}_{\nu}^\lambda$.
	Given any $A \in \zooalgebra$, let us denote the restricted zoo by $\bar{ \animals } = \animals \, \ind \left[ A \right]$ and the corresponding total occupation measure, as in~\eqref{def:point_measure_restricted_to_hit_B_but_not_the_closure_of_A_total_local_size}, by
	\begin{equation*}
		\bar{ \totalsize } := 
		\totalsize^{ \bar{ \animals } } :=
		\sum_{x \in V(\graph)} \bar{\LTmeasure}(x) :=
		\sum_{x \in V(\graph)} \LTmeasure^{ \bar{ \animals } }(x).
	\end{equation*}
	Then we have
	\begin{equation}
		\label{eq:expected_size_without_and_with_multiplicity}
		\E \left[ \left| \trace \left( \bar{\animals} \right) \right| \right] \geq 
		\frac{1}{\lambda \cdot m_1 + 1} \cdot \E \left[ \bar{\totalsize} \right].
	\end{equation}
\end{corollary}

\begin{proof}
	From \eqref{def:trace_of_animals_and_zoo} it follows that we have
	\begin{equation}
		\label{obs:size_of_fat}
		\left| \trace \left( \bar{\animals} \right) \right| = \sum_{x \in V(\graph)} \ind \left[ \bar{ \LTmeasure }(x) > 0 \right].
	\end{equation}
	Hence the law of total expectation and Lemma \ref{lemma:upper_bound_on_expected_number_of_occupation} yield
	\begin{align*}
	   \E \left[ \bar{\totalsize} \right] &= 
		\sum_{x \in V(\graph)} \E \left[ \bar{\LTmeasure} (x) \right] =
		\sum_{x \in V(\graph)} \E \left[ \bar{\LTmeasure} (x)  \, \left| \, \bar{\LTmeasure} (x) > 0 \right. \right] \cdot \prob \left( \bar{\LTmeasure} (x) > 0 \right) \\ 
		& \leq
		\sup_{x \in V(\graph)} \E \left. \left[ \bar{ \LTmeasure }(x) \, \right| \, \bar{\LTmeasure}(x) > 0 \right] \cdot \E \left[ \sum_{x \in V(\graph)} \ind \left[ \bar{ \LTmeasure }(x) > 0 \right] \right] \overset{\eqref{eq:upper_bound_on_expected_number_of_occupation}}{\leq}
		(\lambda \cdot m_1 + 1) \cdot \E \left[ \left| \trace \left( \bar {\animals} \right) \right| \right],
	\end{align*}
	from which our statement follows after a rearrangement.